\documentclass[10pt]{article}

\usepackage{amssymb}   
\usepackage{amsthm}    
\usepackage{amsmath}   
\usepackage{stmaryrd}  
\usepackage{titletoc}  
\usepackage{mathrsfs}  
\usepackage{graphicx}

\newlength{\defbaselineskip}
\newcommand{\setlinespacing}[1]%
           {\setlength{\baselineskip}{#1 \defbaselineskip}}

\theoremstyle{plain}
\newtheorem{thm}{Theorem}[section]
\newtheorem{cor}[thm]{Corollary}
\newtheorem{lem}[thm]{Lemma}
\newtheorem{prop}[thm]{Proposition}

\theoremstyle{definition}
\newtheorem{defn}{Definition}[section]

\newtheorem{rmk}{Remark}[section]

\newcommand{\eps}{\varepsilon}

\DeclareMathOperator*{\esssup}{ess\,sup}
\DeclareMathOperator*{\essinf}{ess\,inf}

\newcommand{\cO}{\mathcal{O}}

\newcommand{\cL}{\mathcal{L}}

\newcommand{\cB}{\mathcal{B}}
\newcommand{\cS}{\mathcal{S}}

\newcommand{\bP}{\mathbb{P}}
\newcommand{\bR}{\mathbb{R}}
\newcommand{\bN}{\mathbb{N}}

\newcommand{\sF}{\mathscr{F}}
\newcommand{\sP}{\mathscr{P}}

\makeatletter\@addtoreset{equation}{section} \makeatother
 \allowdisplaybreaks
\begin{document}

\title{On the Quasi-linear Reflected Backward Stochastic Partial Differential Equations
\footnotemark[1]}

\author{Jinniao Qiu \footnotemark[2] ~\footnotemark[3]~~~ and ~~~~
Wenning Wei\footnotemark[2]}

\footnotetext[1]{ Financial support from the chair Applied Financial Mathematics in  Humboldt-Universit\"{a}t zu Berlin  is gratefully acknowledged.}

\footnotetext[2]{Department of Finance and Control Sciences, School of Mathematical Sciences, Fudan University, Shanghai 200433, China.
\textit{E-mail}: \texttt{qiujinn@gmail.com} (Jinniao Qiu), \texttt{wnwei@fudan.edu.cn} (Wenning Wei).}

\footnotetext[3]{Department of Mathematics, Humboldt-Universit\"{a}t zu Berlin, Unter den Linden 6, 10099 Berlin, Germany}


%
%

\maketitle

\begin{abstract}
 This paper is concerned with the quasi-linear reflected backward stochastic partial differential equation (RBSPDE for short). Basing on the theory of backward stochastic partial differential equation and the parabolic capacity and potential, we first associate the RBSPDE to a variational problem, and via the penalization method, we prove the existence and uniqueness of the solution for linear RBSPDE with Lapalacian leading coefficients. With the continuity approach, we further obtain the well-posedness of general quasi-linear RBSPDEs. Related results, including It\^o formulas for backward stochastic partial differential equations with random measures,  the comparison principle for solutions of RBSPDEs and the connections with reflected backward stochastic differential equations and optimal stopping problems,  are addressed as well.
\end{abstract}

{\bf AMS Subject Classification:} 60H15; 31B15; 35K86

{\bf Keywords:} reflected backward stochastic partial differential equation, reflected backward stochastic differential equation, optimal stopping problem, parabolic capacity and potential, obstacle problem.

\section{Introduction}

Let $(\Omega,\sF,\{\sF_t\}_{t\geq0},\bP)$ be a complete filtered probability space on which is defined an $m$-dimensional standard Brownian
motion $W=\{W_t:t\in[0,\infty)\}$ such that $\{\sF_t\}_{t\geq0}$ is the natural filtration generated by $W$ and augmented by all the
$\bP$-null sets in $\sF$. We denote by $\sP$ the $\sigma$-algebra of the predictable sets on $\Omega\times[0,T]$ associated with $\{\sF_t\}_{t\geq0}$. In this paper, we consider the following quasi-linear RBSPDE:
\begin{equation}\label{RBSPDE}
  \left\{\begin{array}{l}
  \begin{split}
  -du(t,x)=\,&\displaystyle \Bigl[\partial_{x_j}\bigl(a^{ij}(t,x)\partial_{x_i} u(t,x)
        +\sigma^{jr}(t,x) v^{r}(t,x) \bigr)+f(t,x,u(t,x),\nabla u(t,x),v(t,x)) \\
        &\displaystyle +\nabla \cdot g(t,x,u(t,x),\nabla u(t,x),v(t,x))
                \Bigr]\, dt+\mu(dt,x)\\ &\displaystyle
           -v^{r}(t,x)\, dW_{t}^{r}, \quad
                     (t,x)\in Q:=[0,T]\times \mathcal {O};\\
    u(T,x)=\, &G(x), \quad x\in\cO;\\
    u(t,x)\geq\,& \xi(t,x),\,\,d\mathbb{P}\otimes dt\otimes dx-a.e.;\\
    \int_Q \big( u(t,x)&-\xi(t,x) \big)\,\mu(dt,dx)=0,\,a.s.. \quad \quad \textrm{(Skorohod condition)}
    \end{split}
  \end{array}\right.
\end{equation}
Here and in the following the usual summation convention is used, $T\in(0,\infty)$  is a fixed deterministic terminal time, $\cO\subset \bR^d$ is a domain and $\nabla=(\partial_{x_1},\cdots,\partial_{x_d})$ is the gradient operator on $\bR^d$. A solution of RBSPDE (\ref{RBSPDE}) is a random triple $(u,v,\mu)$ defined on $\Omega\times[0,T]\times\bR^d$ such that  (\ref{RBSPDE}) holds in the sense of Definition \ref{def-RBSPDE} in section 4.

Since Bismut's pioneering work \cite{Bismut-BSDE-1973} and Pardoux and Peng's seminal work \cite{ParPeng_90}, the theory of backward stochastic differential equations (BSDEs) has been rather complete and the analysis of backward stochastic differential systems has developed into one of the most innovative and competitive areas of probability theory, both pure and applied. In particular, as a generalization of BSDE, backward stochastic partial differential equation (BSPDE) arises in many applications of probability theory and stochastic processes, for
instance in the nonlinear filtering and the non-Markovian control problems (see \cite{Bensousan_83,EnglezosKaratzas09,Hu_Ma_Yong02,Peng_92,Tang_98,Zhou_92}), and it has already received an extensive attention in literature (see e.g. \cite{DuQiuTang10,DuTangZhang-2013,QiuTangBDSDES2010,
QiuTangMPBSPDE11,QiuTangYou-SPA-2011,Tang-Wei-2013,Tessitore_96,Zhou_92}).

The reflected BSDE is a standard BSDE with an increasing process to keep the solution above a given obstacle. El Karoui et al \cite{El_Karoui-reflec-1997} studied the reflected BSDEs first and associated reflected BSDEs to the optimal stopping problems and the deterministic parabolic variational inequalities. We note that the BSDEs with two obstacles were first studied by Cvitanic and Karaztas \cite{Cvitanic-Karatzas-1996}. Compared with reflected BSDEs, the reflected backward stochastic partial differential equation (RBSPDE) \eqref{RBSPDE} is a BSPDE with reflection and the adapted process $u$  of the solution triple is forced to stay above a given random field $\xi$ (called reflecting obstacle, or simply obstacle) and satisfies the Skorohod condition.

 RBSPDE arises as the so-called backward stochastic parabolic partial differential variational inequality, which is the Hamilton-Jacobi-Bellman equation in the study of optimal stopping problem for stochastic differential equations with the dynamic programming method (see Chang et al \cite{ChangPangYong-2009}). When dealing with the singular control problem of stochastic partial differential equations (SPDEs), {\O}ksendal et al \cite{Oksend-Sulem-Zhang-2011} derived RBSPDEs as the adjoint equations for the maximum principle of Pontryagin type, and via  solutions of RBSPDEs, they futher gave a representation for the value function of the optimal stopping problem of SPDEs. Recently,  Tang and Yang \cite{Tang-Yang-2011} studied the Dynkin game for the stochastic differential equations with random coefficients, and characterized the associated Hamilton-Jacobi-Bellman-Isaacs equation by a backward stochastic partial differential variational inequality, which is a BSPDE with two reflecting obstacles. In fact, \cite{ChangPangYong-2009,Tang-Yang-2011} and \cite{Oksend-Sulem-Zhang-2011} only studied the semi-linear RBSPDEs in the whole space and a smooth bounded domain respectively, and the reflecting obstacles are confined to stochastic differential equations of the form
  $$
    \xi(t,x)=\xi(0,x)+\!\int_0^t\!\!\!\beta_0(s,x)\,ds+\int_0^t\beta(s,x)\,dW_s,
    $$
   which keeps out many interesting applications. Hence, it becomes interesting and significant to establish a general theory for the quasi-linear RBSPDEs on general domains with general reflecting obstacles.

   In this paper, we consider the quasi-linear RBSPDE with the reflecting obstacle dominated from above by some SPDE plus a stochastic potential, and prove the existence and uniqueness of the solution triple $(u,v,\mu)$. These include the classical results on the obstacle problems for deterministic parabolic PDEs  (see e.g. \cite[Theorem IV-1]{Pierre-1979}) as particular cases, and it seems to be new.  Related results, including It\^o formulas for backward stochastic partial differential equations with random measures,  the comparison principle for solutions of RBSPDEs and the connections with reflected backward stochastic differential equations and optimal stopping problems,  are addressed as well.

    In RBSPDE \eqref{RBSPDE}, the random measure is required to satisfy the Skorohod condition.
    Indeed, the Skorohod condition guarantees that the random measure is chosen in a minimal way. For the linear  RBSPDE with Laplacian leading coefficients, we prove in Section 4 that the solution coincides with the minimal point of variational problem \eqref{variation} (see Proposition \ref{prop-regular-obst}, Theorem \ref{thm-Lapl} and Corollary \ref{thm-Lapl}). On the other hand, if the reflecting obstacle is regular enough (see Proposition \ref{prop-regular-obst} below), the random measure may be chosen to be absolutely continuous with respect to the Lebesgue measure $dt\otimes dx$ almost surely and we are allowed to write $\mu(dt,dx)=h(t,x)dtdx$. However, for the general reflecting obstacle, the random measure $\mu$ can be a local time and to make senses of the Skorohod condition, we have to borrow some techniques of the parabolic potential and capacity theory (see, for instance, \cite{Pierre-1979,Pierre-1980,Pierre-1983}) into the backward stochastic framework and give a precise version of the solution $(u,v,\mu)$ to RBSPDE \eqref{RBSPDE} with $u$ being almost surely quasi-continuous.

     Recently, on basis of the parabolic potential and capacity theory, Denis, Matoussi and Zhang \cite{Denis-Matous-Zhang-2012} studied the obstacle problems for \emph{forward} stochastic partial differential equations (OSPDEs). The OSPDEs and RBSPDEs are essentially different: on the one hand, in form, the noises in the former are exogenous and play an active role, while in the latter they are from the martingale representation theorems and governed by the random coefficients, the obstacle and the terminal condition and thus, they are endogenous; On the other hand, in methodology, by stopping times the localization method makes the arguments on the obstacle problems of deterministic PDEs work well in OSPDEs for almost every $\omega\in\Omega$, while in RBSPDEs the localization method does not work as well as in OSPDEs, we can not make the path-wise arguments like in OSPDEs and we have to execute deeper investigations in the backward stochastic framework (see Section 3 and Section 4). Moreover, in our \emph{backward} stochastic framework, we generalize the established results for OSPDEs in \cite[Theorem 3]{Denis-Matous-Zhang-2012} and the new results (Proposition \ref{prop-obst-SPDE}) coincide with the classical ones on the obstacle problems of deterministic parabolic PDEs (see e.g. \cite[Theorem IV-1]{Pierre-1979}).


     It is worth noting that El Karoui et al \cite{El_Karoui-reflec-1997} (see also \cite{El_Karoui-reflec-1997,Klimsiak-2012,Matoussi-Xu-2008,Peng-Xu-2005}) established the equivalent relationship between reflected BSDEs and the associated obstacle problems of parabolic partial differential equations, but in the Markovian case where the coefficients are deterministic functions. In the non-Markovian case where the coefficients can be random, we give  the equivalent representation relationship between the reflected BSDEs and the associated RBSPDEs. This seems to be new as well.
%
%

This paper is organized as follows. In Section 2, we set notations and list some assumptions on the coefficients of RBSPDE \eqref{RBSPDE}. In Section 3, we prepare  some auxiliary results in three subsections. In the first subsection we define the solution for BSPDEs and present a result on the relationships between the random PDEs and BSPDEs. In the second subsection, we introduce the parabolic potential and capacity theory into the backward stochastic framework and execute some interesting investigations. In the third subsection, we establish It\^o formulas for BSPDEs with stochastic regular measures. In Section 4, we prove the existence and uniqueness of the solution to quasi-linear RBSPDE \eqref{RBSPDE}. We give the definition of the solution to RBSPDE, and prove the comparison theorem for quasi-linear RBSPDEs and the uniqueness of the solution in the first subsection. The existence and uniqueness of the solution for linear RBSPDEs with Laplacian leading coefficients is established  in the second subsection. In the third subsection, we prove the well-posedness of general quasi-linear RBSPDE \eqref{RBSPDE}.  Finally, in Section 5, we address  the connections with the reflected BSDEs and  optimal stopping time problems.

\section{Preliminaries}


Denote by $\mathbb{Z}$ the set of all the integers and by $\bN$ the set of all the positive integers. $\bar{\bN}:=\bN \cup \{0\}$ and $\bN^{-1}:=\{\frac{1}{n};n\in\bN\}$. By $|\cdot|$ and $\cdot$, we denote the norm and the scalar product in Euclidean spaces respectively.
%
For the sake of convenience, we set
$$\partial_{s}:=\frac{\partial} {\partial {s}} \ \, {\rm and } \ \, \partial_{st}:=\frac{\partial^2}{\partial s\partial t}.$$



%

For each $l\in \mathbb{N}$ and domain $\Pi\subset \bR^l$, denote by $C_c^{\infty}(\Pi)$ the space of infinitely differentiable functions with compact supports in $\Pi$. In this work, we shall use
$\mathcal{D}:=C_c^{\infty}(\bR)\otimes C_c^{\infty}(\cO)$ as the space of test functions. The Lebesgue measure in $\bR^d$ will be denoted by $dx$. $L^2(\cO)$ ($L^2$ for short) is the usual Lebesgue integrable space with scalar product and norm:
$$
\langle \phi,\,\psi\rangle=\int_{\cO}\phi(x)\psi(x)dx,\quad \|\phi\|=\langle\phi,\,\phi\rangle^{1/2},\,\,\forall
\phi,\psi\in L^2.
$$
When dealing with elements of Hilbert space $(L^2)^k$, $k>1$, for simplicity we still use $\|\cdot\|$ and $\langle\cdot,\,\cdot\rangle$ to denote the norm and the scalar product, i.e.,
$$
\langle\phi,\,\psi\rangle =\sum_{i=1}^k \int_{\cO}\!\!\phi^i(x)\psi^i(x)\,dx,\quad
\|\phi\|=\langle\phi,\,\phi \rangle,\,\,\,\forall\,\phi,\psi\in (L^2)^k.
$$

The first order Sobolev space vanishing on the boundary $\partial \cO$ is denoted by $H^1_0(\cO)$ ($H^1_0$ for short) equipped with scalar product and norm:
$$
\langle \phi,\psi \rangle_{1}=\langle \phi,\,
\psi \rangle + \langle\nabla\phi,\,\nabla\psi \rangle, \quad \|\phi\|_{1}=\left( \|\phi\|^2+\|\nabla \phi\|^2  \right)^{1/2},\,\phi,\psi\in H^1_0,
$$
and its dual space is denoted by $H^{-1}(\cO)$ ($H^{-1}$ for short) equipped with norm $\|\cdot\|_{-1}$. When $\cO=\bR^d$, $H_0^1$ becomes the Bessel potential space $H^1=(I-\Delta)^{-1/2}L^2$. It is well known that there exists a continuous linear operator
$$
\mathcal {J}:\,H^{-1}\longrightarrow (L^2)^{d+1}
$$
such that if $h\in H^{-1}$ and $\mathcal{J} h=(f,g)\in L^2\times (L^2)^d$, then $h=f+\nabla\cdot g$ and
\begin{equation}\label{repre-H-1}
\|f\|^2+\sum_{i=1}^d\|g^i\|^2\leq C(d)\|h\|^2_{-1},\quad
\|h\|^2_{-1}\leq {C(d)}\left(\|f\|^2+\sum_{i=1}^d\|g^i\|^2\right).
\end{equation}
Here and in what follows, $C>0$ is a constant which may vary from line to line and $C(\alpha_1,\alpha_2,\cdots)$ is a constant to depend on the parameters  $\alpha_1,\alpha_2,\cdots$. Indeed, we can take $g=-\nabla (1-\Delta)^{-1} h$ and $f=h-\nabla\cdot g=(1-\Delta)^{-1}h$, with $(1-\Delta)^{-1}$ being the inverse operator of Elliptic operator $(1-\Delta)$ from $H^1_0$ to $H^{-1}$. Thus, we define the dual pairing between $H^1_0$ and $H^{-1}$ as
$$
\langle u,\, h \rangle_{1,-1}=\langle u,\,f\rangle-\langle \nabla u,\,g^i\rangle,\quad u\in H^1_0,\,h=f+\nabla\cdot g,f,g^i\in L^2, i=1,\cdots,d,
$$
where the definition of $\langle u,\, h \rangle_{1,-1}$ is independent of the decomposition $h=f+\nabla \cdot g$. Throughout the work, when relating some $h\in H^{-1}$ to its decomposition $h=f+\nabla\cdot g$ with $(f,g)\in L^2\times (L^2)^d$, we take $(f,g)=\mathcal{J} h$ unless stated otherwise.

Let $V$ be a non-empty convex subset of some Banach space ($\mathbb{B}$, $\|\cdot\|_{\mathbb{B}}$). $\cS ^2 (V)$) is the set of all the $V$-valued,
 $(\sF_t)$-adapted and continuous processes $(X_{t})_{t\in [0,T]}$ such
 that
 $$\|X\|_{\cS ^2(V)}:= \left(E \Big[\sup_{t\in [0,T]} \|X_t\|_{\mathbb{B}}^2\Big] \right)^{{1}/{2}}< \infty.$$
 Denote by $\mathcal{L}^2(V)$ the totality of all the the $V$-valued,
 $(\sF_t)$-adapted processes $(X_{t})_{t\in [0,T]}$ such
 that
 $$
 \|X\|_{\mathcal{L}^2(V)}:=\left(E \Big[\int_0^T \|X_t\|_{\mathbb{B}}^2\,dt\Big] \right)^{{1}/{2}}< \infty.
 $$
In particular, we set $\mathcal{H}=\cS^2(L^2)\cap \mathcal{L}^2(H^1_0)$ equipped with norm
$$
\|\phi\|_{\mathcal{H}}:=\left( \|\phi\|_{\cS^2(L^2)}^2+\|\nabla\phi\|_{\mathcal{L}^2((L^2)^d)}^2  \right)^{1/2},\quad \phi\in \mathcal{H}.
$$
Obviously, $(\cS^2(\mathbb{B}),\,\|\cdot\|_{\cS^2(\mathbb{B})})$, $(\mathcal{L}^2(\mathbb{B}),\|\cdot\|_{\mathcal{L}^2(\mathbb{B})})$ and $(\mathcal{H},\|\cdot\|_{\mathcal{H}})$ are all Banach spaces.

By convention, we treat elements of spaces defined above like $\mathcal{H}$ and $\cL^2(L^2)$ as functions rather than distributions or classes of equivalent functions, and if a function of this class admits a version with better properties, we always denote this version by itself. For example, if $u\in \cL^2(H^1_0)$ and $u$ admits a version lying in $\mathcal{S}^2(L^2)$, we always adopt the modification $u\in \mathcal{H}$.
\medskip

Consider quasi-linear RBSPDE~(\ref{RBSPDE}). We define the following assumptions.

\bigskip\medskip
   $({\mathcal A} 1)$ \it The pair of  random functions
\begin{equation*}
  g(\cdot,\cdot,\cdot,\vartheta,y,z):~\Omega\times[0,T]\times\cO\rightarrow\bR^d
  \textrm{ and }
  f(\cdot,\cdot,\cdot,\vartheta,y,z):~\Omega\times[0,T]\times\cO\rightarrow\bR
\end{equation*}
are $\sP\otimes\cB(\cO)$-measurable for any $(\vartheta,y,z)\in \bR\times\bR^{d}\times\bR^{ m}$. There exist positive constants $L,\kappa$
and $\beta$
   such that for all $(\vartheta_1,y_1,z_1),(\vartheta_2,y_2,z_2)\in \bR\times\bR^d\times\bR^{ m}$
   and $(\omega,t,x)\in \Omega\times[0,T]\times\cO$,
   \begin{equation*}
     \begin{split}
       |g(\omega,t,x,\vartheta_1,y_1,z_1)-g(\omega,t,x,\vartheta_2,y_2,z_2)|\leq& L|\vartheta_1-\vartheta_2|+\frac{\kappa}{2}|y_1-y_2|+\beta^{1/2}|z_1-z_2|,\\
       |f(\omega,t,x,\vartheta_1,y_1,z_1)-f(\omega,t,x,\vartheta_2,y_2,z_2)|\leq& L(|\vartheta_1-\vartheta_2|+|y_1-y_2|+|z_1-z_2|).
     \end{split}
   \end{equation*}\rm

\medskip
   $({\mathcal A}2)$ \it The functions $a$ and $\sigma$ are $\sP\otimes\cB(\cO)$-measurable. There exist positive constants $\varrho>1, \lambda$ and $\Lambda$ such that
   the following  hold for all $\xi\in\bR^d$ and $(\omega,t,x)\in \Omega\times[0,T]\times\cO$,
   \begin{equation*}
     \begin{split}
       &\lambda|\xi|^2\leq (2a^{ij}(\omega,t,x)-\varrho\sigma^{ir}\sigma^{jr}(\omega,t,x))\xi^i\xi^j\leq \Lambda|\xi|^2;\textrm{ (Super-parabolicity)}\\
       &|a(\omega,t,x)|+|\sigma(\omega,t,x)|\leq \Lambda;\textrm{ (Boundedness)}\\
       &\hbox{ \rm and  }\lambda-\kappa-\varrho'\beta>0 \textrm{ \rm with }\varrho':=\frac{\varrho}{\varrho-1}.
     \end{split}
   \end{equation*}\rm

\medskip
   $({\mathcal A} 3)$ \it $G\in L^{2}(\Omega,\sF_T,L^2)$, and
$$f_0:=f(\cdot,\cdot,\cdot,0,0,0)\in \mathcal{L}^2(L^2),\ g_0:=g(\cdot,\cdot,\cdot,0,0,0)\in\mathcal{L}^2((L^2)^d).$$\rm

\begin{rmk}
  By the boundedness of $a$ and $\sigma$, the super-parabolicity can be equivalently written
  $$
  \lambda|\xi|^2\leq (2a^{ij}(\omega,t,x)-\sigma^{ir}\sigma^{jr}(\omega,t,x))\xi^i\xi^j\leq \Lambda|\xi|^2, \,\,a.s.,\quad \forall\,\xi\in\bR^d,\, \forall\, (t,x)\in \Omega\times[0,T]\times\cO,
  $$
  which coincides with the common super-parabolicity assumption on BSPDEs (see, for instance, \cite{DuQiuTang10,Hu_Ma_Yong02,Peng_92,Tang-Wei-2013,Zhou_92}). In this paper, we adopt the form of assumption $(\mathcal{A}3)$, as it helps to clarify the dependence relationships between the constants of the estimates and the coefficients, in what follows.
\end{rmk}

\section{Auxiliary results}
In this section, we shall give some auxiliary results. First we recall several results on backward stochastic partial differential equations (BSPDEs).
\subsection{On the solution for BSPDE}

Define Banach space $(\mathscr{K},\,\|\cdot\|_{\mathscr{K}})$ as the totality of $\phi\in L^2(Q)$ such that
$$
\|\phi\|_{\mathscr{K}}^2:=\int_{0}^T\|\phi(t)\|_{1}^2\,dt+\esssup_{t\in[0,T]}\|\phi(t)\|^2<+\infty,
$$
and set
$$\mathscr{W}=\{\phi\in L^2(0,T;H^1_0): \partial_t\phi\in L^2(0,T;H^{-1}) \}$$
equipped with the norm
$$
\|\phi\|_{\mathscr{W}}:=\left\{  \|\phi\|_{L^2(0,T;H^1_0)}^2+\|\partial_t\phi\|^2_{L^2(0,T;H^{-1})}  \right\}^{1/2},\quad \forall\,\phi\in \mathscr{W}.
$$
 In fact,  each $\phi\in\mathscr{W}$ corresponds to an $L^2$-weak solution of the following parabolic PDE:
$$
  -\partial_t\phi(t,x)=\,\Delta \phi(t,x) +h(t,x)+\nabla\cdot\tilde{h},\,\, (t,x)\in Q;\,\quad \phi(T)=\,\phi_T,
$$
with $(h,\tilde{h})\in L^2(Q)\times L^2(0,T;(L^2)^d)$ and $\phi_T\in L^2$.  By PDE theory (refer to \cite{Ladyzhenskaia_68}), $\mathscr{W}$ is continuously embedded into $C([0,T];L^2)$ and also into $\mathscr{K}$. In particular, there exist two positive constants $c$ and $C$ depending on $T$ such that
\begin{equation}\label{norm-W-space}
c\|\phi\|_{\mathscr{W}}\leq \|\phi_T\|+\|h+\nabla\cdot\tilde{h}\|_{L^2(0,T;H^{-1})}
\leq C\|\phi\|_{\mathscr{W}}.
\end{equation}
We note that all the arguments on $\mathscr{W}$ still hold by reversing the time, since $\phi\in\mathscr{W}$ if and only if $\phi(T-\cdot)\in\mathscr{W}$. Furthermore, we set
$$
\mathscr{W}_T=\{\phi\in\mathscr{W}:\phi(T)=0\},\,\,\mathscr{W}^{+}=\{\phi\in\mathscr{W}:\phi\geq 0\} \textrm{ and } \mathscr{W}^+_T=\mathscr{W}_T\cap \mathscr{W}^+.
$$

Consider the following BSPDE
  \begin{equation}\label{BSPDE-linear-1}
  \left\{\begin{array}{l}
  \begin{split}
  -du(t,x)=\,&\displaystyle \left[\partial_{x_j}\bigl(a^{ij}(t,x)\partial_{x_i} u(t,x)
        +\sigma^{jr}(t,x) v^{r}(t,x)     \bigr)+(\bar{f}+\nabla\cdot \bar{g} )(t,x)\right]\,dt \\
        &\displaystyle      -v^{r}(t,x)\, dW_{t}^{r}, \quad
                     (t,x)\in Q;\\
    u(T,x)=\, &G(x), \quad x\in\cO,
    \end{split}
  \end{array}\right.
\end{equation}
with $a$ and $\sigma$ satisfying the super-parabolicity and boundedness condition in assumption $(\mathcal{A}2)$, and $G\in L^2(\Omega,\sF_T;L^2)$ and $\bar{f},\,\bar{g}^i\in \cL^2(L^2)$, $i=1,2,\cdots,d$.

\begin{defn}\label{definition of weak solution}
  A pair of processes $(u,v)\in \mathcal{H}\times\cL^2((L^2)^m)$
  is called a solution to BSPDE \eqref{BSPDE-linear-1}
  if it holds in the weak sense, i.e.
   for any $\zeta\in \mathcal{D}$ there holds almost surely
  \begin{equation*}
    \begin{split}
      \langle \zeta(t),\,u(t)\rangle
      =&\langle \zeta(T),\, G\rangle+\!\int_t^T\!\! \langle \zeta(s),\, \bar{f}(s)\rangle ds -\!\int_t^T\!\!\langle \partial_s \zeta(s),\, u(s) \rangle ds
      -\!\int_t^T \!\!\langle \zeta(s),\, v^r(s)\rangle  dW_s^r\\
      &-\!\int_t^T\!\! \langle \partial_{x_j} \zeta(s),\quad a^{ij}\partial_{x_i} u(s)+\sigma^{j r}v^r(s)
      +\bar{g}^j(s) \rangle ds, \quad \forall \ t\in[0,T].
    \end{split}
  \end{equation*}

  For $G\in L^2(\Omega,\sF_T;L^2)$ and $\bar{f},\,\bar{g}^i\in \cL^2(L^2)$, $i=1,2,\cdots,d$,  denote by $\mathscr{U}(a,\sigma,G,\bar{f},\bar{g})$  the totality of $u\in\mathcal{H}$ which together with some $v\in \cL^2((L^2)^m)$ consists of a solution to BSPDE \eqref{BSPDE-linear-1}. By \cite{Qiu-2012,QiuTangBDSDES2010}, each $\mathscr{U}(a,\sigma,G,\bar{f},\bar{g})$ admits one and only one element.
\end{defn}

\begin{rmk}
  By \cite[Remark 2.1]{QiuTangMPBSPDE11}, Definition \ref{definition of weak solution} can be equivalently stated with the test function space being replaced by $C_c^{\infty}(\cO)$.
\end{rmk}

\begin{rmk}\label{rmk-math-U}
  By taking $g=a^{i\cdot}\partial_{x^i}u+\sigma^{\cdot r}v^r+\bar{g}-\nabla u$ and $f=\bar{f}$, we write BSPDE \eqref{BSPDE-linear-1} equivalently into the following form
  \begin{equation}\label{BSPDE-laplac}
  \left\{\begin{array}{l}
  \begin{split}
  -du(t,x)=\,&\displaystyle \Delta u(t,x)+(f+\nabla\cdot g )(t,x)\,dt      -v^{r}(t,x)\, dW_{t}^{r}, \quad
                     (t,x)\in Q;\\
    u(T,x)=\, &G(x), \quad x\in\cO.\\
    \end{split}
  \end{array}\right.
\end{equation}
  On the contrary, each BSPDE of the above form
   can also be written equivalently into some BSPDE associated with $(a,\sigma)$ like \eqref{BSPDE-linear-1}.   Moreover, by It\^o  formula (see \cite[Theorem 1.2 of Chapter 1]{Qiu-2012} or \cite[Theorem 4.2]{RenRocknerWang2007}), we have
  \begin{align*}
    \begin{split}
      &\|u(t)\|^2+\int_t^T\Big(2\|\nabla u(s)\|^2+\|v(s)\|^2\Big)\,ds\\
      =&\|G\|^2+\int_t^T 2\Big[\langle u(s),\,f(s) \rangle-\langle \nabla u(s),\,g(s) \rangle\Big]  \,ds-2\int_t^T \langle u(s),\, v(s)\,dW_s\rangle ,
    \end{split}
  \end{align*}
  which together with
  \begin{align*}
    \begin{split}
      E\Big[ \sup_{s\in [t,T]}\Big| \int_s^T\langle u(\tau),\, v(\tau)\,dW_{\tau}\rangle  \Big| \Big]
      \leq &\, 2E\Big[ \sup_{s\in [t,T]}\Big| \int_t^s\langle u(\tau),\, v(\tau)\,dW_{\tau}\rangle  \Big| \Big]\\
     \textrm{(by BDG inequality) } \leq &\, C E\Big[\Big(\int_t^T\|u(s)\|^2\|v(s)\|^2\,ds  \Big)^{1/2}\Big]
    \end{split}
  \end{align*}
  implies
  \begin{align*}
    \begin{split}
      &\alpha E[\sup_{s\in[t,T]}\|u(s)\|^2] + (1-\alpha) E\left[ \|u(t)\|^2 \right]
      +E\Big[\int_t^T\Big(2\|\nabla u(s)\|^2+\|v(s)\|^2\Big)\,ds\Big]\\
      \leq &\,E\bigg[
      \|G\|^2+\int_t^T\Big(\|u(s)\|^2+\|\nabla u(s)\|^2
          \Big)\,ds
      +\|f\|^2_{\cL^2(L^2)}+\|g\|^2_{\cL^2((L^2)^d)}\bigg]\\
      &\,+C\alpha E\bigg[\Big(\int_t^T\|u(s)\|^2\| v(s)\|^2\,ds  \Big)^{1/2}\bigg]\\
      \leq &\,E\bigg[
      \|G\|^2+\int_t^T\Big(\|u(s)\|^2+\|\nabla u(s)\|^2
      \Big)\,ds\bigg]
     +\|f\|^2_{\cL^2(L^2)}+\|g\|^2_{\cL^2((L^2)^d)}
     \\
     &\,+\frac{\alpha}{2} E[\sup_{s\in[t,T]}\|u(s)\|^2]
     +2C\alpha
     \int_t^T E\left[  \|v(s)\|^2 \right]\,ds ,
    \end{split}
  \end{align*}
  with $\alpha\in\{0,1\}$. Applying successively  Gronwall inequality to the above estimate with $\alpha=0$ and $\alpha=1$, we obtain
  \begin{align*}
    \begin{split}
      &E[\sup_{s\in[t,T]}\|u(s)\|^2]
      +E\Big[\int_t^T\Big(\|\nabla u(s)\|^2
      +\|v(s)\|^2\Big)\,ds\Big]\\
      \leq &\, C(T)\left(
      E[\|G\|^2]+\|f\|^2_{\cL^2(L^2)}+\|g\|^2_{\cL^2((L^2)^d)}\right)\leq \, C(T)
      \left( E[\|G\|^2]+\|f+\nabla\cdot g\|_{\cL^2(H^{-1})}  \right).
    \end{split}
  \end{align*}

Setting
$$\mathscr{U}=\left\{u\in\mathscr{U}(a,\sigma,G,\bar{f},\bar{g});\, G\in L^2(\Omega,\sF_T;L^2) \textrm{ and } \bar{f},\,\bar{g}^i\in \cL^2(L^2),\, i=1,2,\cdots,d   \right\},$$
we have
$$
\mathscr{U}=\left\{u\in\mathscr{U}(I,0,G,f,g);\, G\in L^2(\Omega,\sF_T;L^2) \textrm{ and } {f},\,g^i\in \cL^2(L^2),\, i=1,2,\cdots,d   \right\}.
$$
   In view of the above estimates and relation \eqref{repre-H-1}, we equip $\mathscr{U}$ with the following norm
   $$
   \|u\|_{\mathscr{U}}^2=E[\|G\|^2]+\|f+\nabla g\|^2_{\cL^2(H^{-1})},\quad
   \forall \,u\in\mathscr{U}(I,0,G,f,g).
   $$
   Then, $(\mathscr{U},\,\|\cdot\|_{\mathscr{U}})$ is a Hilbert space.
\end{rmk}

  In view of the duality between \emph{forward} SPDEs and BSPDEs, we introduce the following lemma which, basically, is due to Bensoussan \cite[Lemma 2.2]{Bensousan_83}. 

\begin{lem}\label{lem-BSPDE-condition-expect}
  Let $G\in L^2(\Omega,\sF_T;L^2)$ and $\hat{f},\,\hat{g}^i\in L^2(\Omega,\sF_T;L^2(Q))$, $i=1,2,\cdots,d$. Assume $\hat{u}\in L^2(\Omega,\sF_T;\mathscr{W})$ and $\hat{u}$ satisfies almost surely the following parabolic PDE
  $$
  -\partial_t \hat{u}=\Delta \hat{u}+\hat{f}+\nabla\cdot\hat{g};\quad \hat{u}(T)=G,
  $$
  in the weak sense (see \cite{Ladyzhenskaia_68}). Taking conditional expectations in Hilbert spaces (see \cite{Prato-Zabczyk-1992}), we set
  $$
  (\bar{f},\bar{g})(t)=E\left[ (\hat{f},\hat{g})(t)\big|\sF_t   \right],\quad t\in[0,T].
  $$
  If $(u,v)\in\mathcal{H}\times \cL^2((L^2)^m)$ is the unique solution of BSPDE (for the uniqueness and existence of the solution to general quasi-linear BSPDEs, see \cite{Qiu-2012,QiuTangBDSDES2010})
  \begin{equation}\label{BSPDE-lem-1}
  \left\{\begin{array}{l}
  \begin{split}
  -du(t,x)=\,&\displaystyle \big[\Delta u(t,x)+(\bar{f}+\nabla\cdot \bar{g} )(t,x)\big]\,dt      -v^{r}(t,x)\, dW_{t}^{r}, \quad
                     (t,x)\in Q;\\
    u(T,x)=\, &G(x), \quad x\in\cO,
    \end{split}
  \end{array}\right.
\end{equation}
then,
\begin{align}\label{eq-lem-bensoussan}
u(t)=E[\hat{u}(t)|\sF_t],\,\, a.s.,\quad \textrm{ for each } t\in [0,T]
\end{align}
and $\|u\|_{\mathscr{U}}\leq C  \|\hat{u}\|_{L^2(\Omega,\sF_T;\mathscr{W})}$ with the constant $C$ being independent of $u$ and $\hat{u}$.
\end{lem}

\begin{proof} 
  For the reader's convenience, we sketch the proof of the lemma.  For each $t_0\in[0,T)$ and $\eta_{t_0}\in L^2(\Omega,\sF_{t_0};L^2)$, consider the following SPDE
    \begin{equation}\label{eq-spde-lem-1}
  \left\{\begin{array}{l}
  \begin{split}
  \partial_t\eta(t,x)=\,&\displaystyle \Delta \eta (t,x), \quad
                     (t,x)\in [t_0,T]\times\cO;\\
    \eta(t_0,x)=\, &\eta_{t_0}(x),\quad x\in \cO.
    \end{split}
  \end{array}\right.
\end{equation}
It\^o  formula for SPDEs yields
\begin{align*}
\begin{split}
&\langle G,\,\eta(T) \rangle+\int_{t_0}^T\left(\langle\eta(s),\,\bar{f}(s) \rangle-\sum_{i=1}^d\langle\partial_{x^i}\eta(s),\bar{g}^i(s) \rangle \right)\,ds
\\
=&\,\langle u(t_0),\,\eta_{t_0} \rangle+\sum_{r=1}^m\int_{t_0}^T\langle u(s),\,v^r(s)\rangle\,dW^r_s,\,\,\,a.s.,
\\
\textrm{and }\quad&\\
&\langle G,\,\eta(T) \rangle+\int_{t_0}^T\left(\langle\eta(s),\,\hat{f}(s) \rangle-\sum_{i=1}^d\langle\partial_{x^i}\eta(s),\hat{g}^i(s) \rangle \right)\,ds
\\
=&\,\langle \hat{u}(t_0),\,\eta_{t_0} \rangle,\,\,\,a.s..
\end{split}
\end{align*}
Taking conditional expectations on both sides of the above equations, we obtain
$$
\langle u(t_0),\,\eta_{t_0} \rangle=E[\langle \hat{u}(t_0),\,\eta_{t_0} \rangle|\sF_{t_0}]=\langle E[\hat{u}(t_0)|\sF_{t_0}],\,\eta_{t_0} \rangle,\,a.s.. 
$$
Thanks to the arbitrariness of $t_0\in[0,T)$ and $\eta_{t_0}\in L^2(\Omega,\sF_{t_0})$, if follows that $u(t)=E[\hat{u}(t)|\sF_t]$, a.s., for each $t\in[0,T]$. By Remark \ref{rmk-math-U},
$$\|u\|^2_{\mathscr{U}}=  E[\|G\|^2]+\|\bar{f}+\nabla\cdot \bar{g}\|^2_{\cL^2(H^{-1})}
\leq E[\|G\|^2]+\|\hat{f}+\nabla\cdot \hat{g}\|^2_{L^2(\Omega\times[0,T];H^{-1})}
 \leq C \|\hat{u}\|^2_{L^2(\Omega,\sF_T;\mathscr{W})},$$
 with the constant $C$ being independent of $u$ and $\hat{u}$.  We complete the proof.
\end{proof}

\subsection{Parabolic capacity and potential}

In this subsection, let us recall briefly parts of the theory on parabolic capacity and potentials which was established by Pierrer \cite{Pierre-1979,Pierre-1980,Pierre-1983}. We note that these tools were used by Klimsiak \cite{Klimsiak-2012}  to provide a probabilistic representation for the semilinear PDEs with obstacle in terms of reflected BSDEs and by Denis, Matoussi and Zhang \cite{Denis-Matous-Zhang-2012} to study the obstacle problems for \emph{forward} stochastic partial differential equations (OSPDEs). Taking into account the backward randomness of our reflected BSPDE \eqref{RBSPDE}, we shall execute several deep and interesting investigations under the backward stochastic framework.

\begin{defn}\label{def-potential}
  It is called a parabolic potential when belonging to
  \begin{equation*}
    \begin{split}
      \mathcal{P}=\left\{
      u\in \mathscr{K}: \int_0^T\!\!\!\big( \langle- \partial_t\phi(t),\,u(t) \rangle +\langle\partial_{x^i}\phi(t),\,\partial_{x^i} u(t)\rangle \big)\,dt\geq 0,\quad\forall \phi\in \mathscr{W}^+_T
      \right\}.
    \end{split}
  \end{equation*}
\end{defn}


$\mathcal{P}$ is a closed convex subset of $(\mathscr{K},\,\|\cdot\|_{\mathscr{K}})$. Denote by $\mathscr{L}^0(\mathscr{W})$ the totality of the measurable maps from $(\Omega,\sF_T)$ to $\mathscr{W}$ such that each $u\in \mathscr{L}^0(\mathscr{W})$ is a $L^2$-valued adapted process. Define
$$\mathscr{L}^2(\mathscr{W})=L^2(\Omega,\sF_T;\mathscr{W})\cap \mathscr{L}^0(\mathscr{W})$$
equipped with the norm
$$
\|u\|_{\mathscr{L}^2(\mathscr{W})}:=\left(E\left[\|u\|^2_{\mathscr{W}}\right]\right)^{1/2},\quad \forall\, u\in \mathscr{L}^2(\mathscr{W}).
$$
Similarly, we define $\|\cdot\|_{\mathscr{L}^2(\mathbb{B})}$ and $\mathscr{L}^p(\mathbb{B})$ with $\mathbb{B}=\mathscr{K}$ or $L^{q}(Q)$, for $p\in \{0,2\}$ and $q\in[1,+\infty]$. Spaces like $(\mathscr{L}^2(\mathbb{B}),\|\cdot\|_{\mathscr{L}^2(\mathbb{B})})$ are complete when $\mathbb{B}=\mathscr{W},\mathscr{K}$ or $L^{q}(Q)$. Denote
$$\mathscr{L}^2(\mathcal{P})=\mathscr{L}^2(\mathscr{K})\cap\mathscr{L}^0(\mathcal{P}),\ \ \textrm{ and }\ \
\|\phi\|_{\mathscr{L}^2(\mathcal{P})}=\|\phi\|_{\mathscr{L}^2(\mathscr{K})},\ \ \forall\,\phi\in\mathcal{P}.$$
Then, $\mathscr{L}^2(\mathcal{P})$ is not a linear space but a closed convex set of $(\mathscr{L}^2(\mathscr{K}),\|\cdot\|_{\mathscr{L}^2(\mathscr{K})})$.

\begin{rmk}\label{rmk-closedness-potential}
   By Definition \ref{def-potential}, $\mathscr{L}^2(\mathcal{P})$ is closed in the sense that if $\{u^n\}_{n\in\bN}\subset\mathscr{L}^2(\mathcal{P})$ is a bounded sequence of $\mathscr{L}^2(L^{\infty}(0,T;L^2))$ and converges weakly to some $u$ in $\cL^2(H^1_0)$, we have $u\in \mathscr{L}^2(\mathcal{P})$.
\end{rmk}

Denote by $\mathscr{C}(Q)$ the totality of continuously differentiable functions on $Q$ with compact support. Thanks to the Hahn-Banach theorem and the denseness of $\mathscr{W}_T\cap\mathscr{C}(Q)$ in $\mathscr{C}(Q)$, we have the following representation for the parabolic potential.
\begin{prop}[Proposition I-1 of \cite{Pierre-1980}]\label{prop-repre-potential}
  For each $u\in\mathcal{P}$, there exists one and only one Radon measure on $[0,T)\times\cO$ denoted by $\mu^u$, such that
  \begin{equation*}
    \forall \phi\in\mathscr{W}_T\cap \mathscr{C}(Q),\quad \int_0^T\!\!\left(\langle -\partial_t\phi(t),\,u(t)\rangle+\langle \partial_{x^i}\phi(t),\partial_{x^i}u(t) \rangle\right)\,dt
    =\int_{0}^T\int_{\cO}\phi(t,x)\mu^{u}(dt,dx).
  \end{equation*}
\end{prop}


\begin{defn}\label{def-capacity}
  For any open set $A\subset [0,T)\times \cO$, set
  $$
  cap (A)=\inf\left\{\|\phi\|^2_{\mathscr{W}}:\,\,\phi\in \mathscr{W}^+,\,\phi\geq 1 \textrm{ a.e. on }
  A  \right\}.
  $$
  For any Borelian $B\subset [0,T)\times \cO$, we define the capacity:
  $$
  cap(B)=\inf \left\{cap(A):\,\,A\supset B \textrm{ open}\right\}.
  $$
\end{defn}


By \cite[Theorem 1]{Pierre-1983}, the above definition of capacity is equivalent to that of \cite{Pierre-1979,Pierre-1980}. In the following, we say a property holds quasi-everywhere (q.e. in short),
if it holds outside a polar set that is of zero capacity.
  In addition, though the parabolic capacity is defined on $[0,T)\times\cO$ rather than $Q=[0,T]\times\cO$, we shall treat the capacity on $Q$ as the projection of that defined on some time interval $[0,T+\delta)$ with $\delta>0$
  (for instance, see the proof for (i) of Proposition \ref{prop-quasi-conti-Cont_TL2} below).


\begin{defn}\label{def-quasi-contin}
  A real valued function $u$ on $[0,T)\times\cO$ is said to be quasi-continuous, if there exists a sequence of open sets $A_n\subset [0,T)\times\cO$ such that

  (i) for each $n$, $u$ is continuous on the complement of $A_n$;

  (ii) $\lim_{n\rightarrow +\infty} cap (A_n)=0$.
\end{defn}

 In what follows, denote by $\mathcal{P}_0$ the totality of $u\in\mathcal{P}$ such that $u$ is quasi-continuous and  $u(0)=0$ in $L^2$. Each $u\in\mathcal{P}_0$ is called a regular potential and the associated Radon measure $\mu^u$ is called a regular measure and written
  $$\mu^u=\partial_tu -\Delta u.$$
  In addition, define
  $$
  \mathscr{L}^2(\mathcal{P}_0)=\mathscr{L}^2(\mathcal{P})\cap\mathscr{L}^0(\mathcal{P}_0).
  $$
  Each $u\in\mathscr{L}^2(\mathcal{P}_0)$ is called a stochastic regular potential, and the associated random Radon measure $\mu^u$ is called a stochastic regular measure.

Basing on the BSPDE theory, we shall generalize the existing results on the obstacle problems for \emph{forward} SPDEs (for instance, see \cite[Theorem 3]{Denis-Matous-Zhang-2012}). Before the generalization, we give a lemma first.
\begin{lem}\label{lem-K-W}
  There exists $C>0$ such that
   $\forall u\in\mathscr{L}^2(\mathcal{P})$, $\exists\, \bar{\phi}\in \mathscr{U}$ with
  $$\bar{\phi}\geq u,\,\, d\mathbb{P}\otimes dt \otimes dx\textrm{-a.e.};\quad \|\bar{\phi}\|_{\mathscr{U}}\leq C \|u\|_{\mathscr{L}^2(\mathcal{P})}.$$
\end{lem}
\begin{proof}

  For almost every $\omega\in\Omega$, we consider the solution $\phi$ of PDE:
  \begin{align*}
      -\partial_t\phi=&\Delta \phi-2\Delta u ;\quad
      \phi(T)=u(T-),
  \end{align*}
  where by \cite[Lemma I.3]{Pierre-1980}, the limit $u(T-):=\lim_{\tau\uparrow T}\frac{1}{T-\tau}\int^T_{\tau}u(t)dt \ (\textrm{weakly in }L^2)$ exists for almost every $\omega\in\Omega$.
  Note that the separability of $\mathscr{W}$ allows us to choose a measurable version of $\phi:(\Omega,\sF_T)\rightarrow \mathscr{W}$.  By \cite[Proposition 3]{Pierre-1983},
  we have
  $\phi\in L^2(\Omega,\sF_T;\mathscr{W})$ and
  $$\phi\geq u,\,\, d\mathbb{P}\otimes dt \otimes dx\textrm{-a.e.};\quad \|\phi\|_{L^2(\Omega,\sF_T;\mathscr{W})}\leq C \|u\|_{\mathscr{L}^2(\mathcal{P})},$$
  with positive constant $C$ being independent of $u$ and $\phi$.
   Since $u\in\mathscr{L}^2(\mathcal{P})$, there exist $h,\tilde{h}^i\in L^2(\Omega,\sF_T;L^2)$ $i=1,\cdots,d$, such that $-2\Delta u=h+\nabla\cdot\tilde{h}$.
Take conditional expectations in Hilbert spaces (see \cite{Prato-Zabczyk-1992}) and set $$\bar{\phi}(t)=E[\phi(t)|\sF_t],\,\,\bar{h}(t)=E[h(t)|\sF_t]\textrm{ and }\bar{\tilde{h}}(t)=E[\tilde{h}(t)|\sF_t],\quad \forall\,t\in[0,T]$$
By Lemma \ref{lem-BSPDE-condition-expect}, $\bar{\phi}\in \mathscr{U}(I,0,u(T-),\bar{h},\bar{\tilde{h}})$.
    Hence, we have $\bar{\phi}\in \mathscr{U}$, $\bar{\phi}\geq u,\,\, d\mathbb{P}\otimes dt \otimes dx\textrm{-a.e.}$ and by the estimates of Remark \ref{rmk-math-U},
 $$ \|\bar{\phi}\|_{\mathscr{U}}\leq C
 \left( \|u(T-)\|_{L^2(\Omega,\sF_T;L^2)}+\|\bar{h}\|_{\cL^2(L^2)}+\|\bar{\tilde{h}}\|_{\cL^2((L^2)^d)}  \right)
 \leq C \|u\|_{\mathscr{L}^2(\mathcal{P})},\,a.s..$$
We complete the proof.
\end{proof}

\begin{prop}\label{prop-obst-SPDE}
Let $\xi$ be an almost surely quasi-continuous adapted process such that $\xi(0)\leq u_0$ almost surely and
$$
\xi\leq \hat{\xi}+\breve{\xi},\quad d\bP\otimes dt \otimes dx-a.e.,
$$
 with $\breve{\xi}\in \mathscr{L}^2(\mathcal{P})$  and $\hat{\xi}$ being the solution of SPDE
\begin{equation*}
  \left\{\begin{array}{l}
  \begin{split}
  d\hat{\xi}(t,x)=\,&\displaystyle \Bigl[\Delta \hat{\xi}(t,x)+(\hat{f}+\nabla\cdot \hat{g} )(t,x)\,dt      +\hat{h}(t,x)\, dW_{t}, \quad
                     (t,x)\in Q;\\
    \hat{\xi}(0,x)=\, &\hat{\xi}_0(x), \quad x\in\cO,
    \end{split}
  \end{array}\right.
\end{equation*}
where $\hat{\xi}_0,u_0\in L^2$, $\hat{f},\hat{g}^i\in\cL^2(L^2)$, $i=1,\cdots,d$ and $\hat{h}\in\cL^2((L^2)^m)$. Given $(\bar{f},\bar{g},\bar{h})\in \cL^2(L^2)\times\cL^2((L^2)^d)\times\cL^2((L^2)^m) $,
there exists a unique pair $(u,\mu)$ such that

(i)  $u\in\mathcal{H}$ is almost surely quasi-continuous and $u\geq \xi$, $ d\mathbb{P}\otimes dt \otimes dx$-a.e.;

(ii) $\mu$ is a stochastic regular measure, and for any $\varphi\in\mathcal{D}$, there holds almost surely
\begin{equation*}
  \begin{split}
    &\langle u(t),\,\varphi(t) \rangle +\int_0^t\left[\langle \nabla u(s),\,\nabla \varphi(s) \rangle-\langle u(s),\,\partial_s \varphi(s)   \rangle  \right]\,ds-\int_{[0,t]\times\cO}\varphi(s,x)\mu(ds,dx)\\
    =\,&
    \langle u_0,\varphi(0) \rangle+\int_0^t\left[\langle \bar{f}(s),\,\varphi(s) \rangle-\langle \bar{g}(s),\, \nabla\varphi(s) \rangle \right]\,ds+\int_0^t\langle \varphi(s),\,\bar{h}(s)\,dW_s\rangle,\ \ \forall \,t\in[0,T];
  \end{split}
\end{equation*}

(iii)
$$
\int_Q\left(u(t,x)-\xi(t,x)\right)\mu(dt,dx)=0,\quad a.s.;
$$

(iv)
\begin{equation*}
  \begin{split}
    \|u\|_{\mathcal{H}}\leq C \Big\{&  \|\hat{\xi}_0\|+\|u_0\|+\|\bar{f}\|_{\cL^2(L^2)}
+\|\hat{f}\|_{\cL^2(L^2)}+\|\bar{g}\|_{\cL^2((L^2)^d)}+\|\hat{g}\|_{\cL^2((L^2)^d)}
\\
&+\|\bar{h}\|_{\cL^2((L^2)^m)}+\|\hat{h}\|_{\cL^2((L^2)^m)}
+\|\breve{\xi}\|_{\mathscr{L}^2(\mathcal{P})}
 \Big\}.
  \end{split}
\end{equation*}
\end{prop}
\begin{proof}[Sketch of the proof.]
  By Lemma \ref{lem-K-W}, there exists $\phi\in\mathscr{U}$ such that
  $$
  \phi\geq \breve{\xi},\,\, d\mathbb{P}\otimes dt \otimes dx\textrm{-a.e.};\quad \|\phi\|_{\mathscr{U}}\leq C \|\breve{\xi}\|_{\mathscr{L}^2(\mathcal{P})}.
  $$
  For $\phi\in\mathscr{U}$, there exist $(f,\tilde{g},v)\in \cL^2(L^2)\times\cL^2((L^2)^d)\times\cL^2((L^2)^m) $ and $G\in L^2(\Omega,\sF_T;L^2)$ such that
  \begin{equation*}
    \begin{cases}
     -d\phi(t,x)=&\, \Bigl[\Delta \phi(t,x)+(f+\nabla\cdot\tilde{g} )(t,x)\,dt      -v^{r}(t,x)\, dW_{t}^{r}, \quad
                     (t,x)\in Q;\\
    \phi(T,x)=&\,G(x), \quad x\in\cO,
    \end{cases}
  \end{equation*}
  which can be equivalently written into  the form of SPDE
    \begin{equation*}
    \begin{cases}
     d\phi(t,x)=\,&\displaystyle \Bigl[\Delta \phi(t,x)-(f+\nabla\cdot g )(t,x)\,dt      +v^{r}(t,x)\, dW_{t}^{r}, \quad
                     (t,x)\in Q;\\
    \phi(0,x)=\, &\phi(0,x), \quad x\in\cO,
    \end{cases}
  \end{equation*}
  with $g:=\tilde{g}+2\nabla \phi$.
   As
  $$\xi(t)\leq\hat{\xi}(t)+\breve{\xi}(t)\leq \hat{\xi}(t)+\phi(t),\quad t\in[0,T], $$
  by \cite[Theorem 3]{Denis-Matous-Zhang-2012}, there exists a unique pair $(u,\mu)$ such that
 the assertions (i)-(iii) hold.
In a similar way to \cite[Lemma 4 and Theorem 3]{Denis-Matous-Zhang-2012}, applying It\^o  formula to the penalized approximation sequences and taking limits, we obtain
  \begin{equation*}
  \begin{split}
    \|u\|_{\mathcal{H}}\leq C \Big\{&  \|\hat{\xi}_0\|+\|u_0\|+\|\bar{f}\|_{\cL^2(L^2)}
+\|\hat{f}\|_{\cL^2(L^2)}+\|\bar{g}\|_{\cL^2((L^2)^d)}+\|\hat{g}\|_{\cL^2((L^2)^d)}
\\
&+\|\bar{h}\|_{\cL^2((L^2)^m)}+\|\hat{h}\|_{\cL^2((L^2)^m)}
+\|\phi\|_{\mathscr{U}}
 \Big\}\\
 \leq C \Big\{&  \|\hat{\xi}_0\|+\|u_0\|+\|\bar{f}\|_{\cL^2(L^2)}
+\|\hat{f}\|_{\cL^2(L^2)}+\|\bar{g}\|_{\cL^2((L^2)^d)}+\|\hat{g}\|_{\cL^2((L^2)^d)}
\\
&+\|\bar{h}\|_{\cL^2((L^2)^m)}+\|\hat{h}\|_{\cL^2((L^2)^m)}
+\|\breve{\xi}\|_{\mathscr{L}^2(\mathcal{P})}
 \Big\}.
  \end{split}
\end{equation*}
This completes the proof.
\end{proof}

\begin{rmk}\label{rmk-prop-OSPDE}
  In Proposition \ref{prop-obst-SPDE}, the obstacle process $\xi$ is dominated by a stochastic regular potential plus a solution of some SPDE, while in 
  \cite[Theorem 3]{Denis-Matous-Zhang-2012}, the obstacle process is only allowed to be dominated by a solution of some SPDE. In this sense, we generalize the results of
  \cite[Theorem 3]{Denis-Matous-Zhang-2012}.  In fact, Proposition \ref{prop-obst-SPDE} includes the classical deterministic results (see \cite{Pierre-1979,Pierre-1980}) as particular cases, i.e., when all the terms involved in Proposition \ref{prop-obst-SPDE} are deterministic, the assertions coincide with those of \cite[Theorem IV-1]{Pierre-1979}. Furthermore, in a similar way to \cite{Denis-Matous-Zhang-2012}, we can extend the results herein to the obstacle problems for general quasi-linear SPDEs. However, we do not seek such a generality in this work.
\end{rmk}

An immediate consequence of this proposition is the corollary.
\begin{cor}\label{cor-prop-obstac-SPDE}
There exists $C>0$ such that
$\forall \phi\in\mathscr{U}$, $\exists\, v\in\mathscr{L}^2(\mathcal{P})$ with $v$ being quasi-continuous a.s. and
  $$v\geq \phi,\,\, d\mathbb{P}\otimes dt \otimes dx\textrm{-a.e.};\quad \|v\|_{\mathscr{K}}\leq C \|\phi\|_{\mathscr{U}},\,\,a.s..    $$
  In particular, when $\phi$ is deterministic, $v$ becomes deterministic as well.
\end{cor}

To approximate a parabolic potential by regular ones, we introduce the following lemma.
\begin{lem}[Proposition I-3 and Corollary II-1 of \cite{Pierre-1980}]\label{lemm-approx-potential}
  Let $u$ belong to $\mathcal{P}$ with associated Radon measure $\mu^u$. For each $\delta>0$, let $u_{\delta}\in\mathscr{W}$ be the weak solution of the following parabolic PDE
  $$
  (P_{\delta})\quad u_{\delta}(0)=u(0+),\,\,
  \partial_t u_{\delta}(t)=\Delta u_{\delta}(t)+\frac{u-u_{\delta}}{\delta},
  $$
   where, by \cite[Lemma I.3]{Pierre-1980}, there exists the limit $u(0+):=\lim_{\tau\downarrow 0}\frac{1}{\tau}\int_0^{\tau}u(t)dt \ (\textrm{weakly in }L^2)$. We assert that

  (1) $\{u_{\delta}\}_{\delta>0}\subset \mathcal{P}$ converges increasingly in $L^2(0,T;L^2)$ and weakly in $L^2(0,T;H^1_0)$ to $u$, as $\delta\rightarrow 0^+$;

  (2) $\mu^{u_{\delta}}:={\delta}^{-1}{(u-u_{\delta})}$ converges vaguely to $\mu^u$ as $\delta\rightarrow 0^+$ and for each $\delta>0$,
  $$\int_{[0,T)\times\cO} \mu^{u_{\delta}}(dt,dx)\leq \int_{[0,T)\times\cO} \mu^{u}(dt,dx);$$

  (3) if we set
  \begin{equation}\label{eq-q.s.c.version}
  \bar{u}=\sup_{\delta>0}\textrm{ quasi ess }u_{\delta},\quad q.e.,
  \end{equation}
  then $\bar{u}$ is right-continuous from $[0,T]$ to $L^2$ and for each $\phi\in\mathcal{P}$, $\bar{u}\in L^1(Q,\mu^{\phi})$ and there holds the estimate
  $$
  \int_{[0,T)\times\cO}\bar{u}(t,x)\mu^{\phi}(dt,dx)\leq C \|u\|_{\mathscr{K}}\|\phi\|_{\mathscr{K}}.
  $$
\end{lem}

\begin{rmk}\label{rmk-capacity}
  Letting $A\subset\cO$ be a compact set, we have by \cite[Proposition II.4]{Pierre-1980}
  $$cap (\{t\}\times A)=\int_A dx,\quad \forall t\in[0,T).$$
  Therefore, if $u$ and $\bar{u}$ are two measurable function on $Q$ and $u=\bar{u}$ q.e., then $u(t,\cdot)=\bar{u}(t,\cdot)$, $\forall t\in[0,T)$. Moreover, for any open set $A\subset[0,T)\times\cO$ and $\phi\in\mathscr{W}^+$ satisfying $\phi\geq 1$ a.e. on $A$, by Corollary \ref{cor-prop-obstac-SPDE}, there exists $\varphi\in\mathcal{P}$ such that $\varphi\geq \phi$ and $\|\varphi\|_{\mathscr{K}}\leq C \|\phi\|_{\mathscr{W}}$ and thus,
  we have
  \begin{equation}
  \begin{split}
    \int_{[0,T)\times\cO}1_A\,\mu^u(dt,dx)\leq
    \int_{[0,T)\times\cO} \phi(t,x)\mu^u(dt,dx)
    \leq& \int_{[0,T)\times\cO} \bar{\varphi}(t,x)\mu^u(dt,dx)\\
    \textrm{(by Lemma \ref{lemm-approx-potential}) }
    \leq&\, C \|\varphi\|_{\mathscr{K}}\|u\|_{\mathscr{K}}\\
    \leq& \,C\|\phi\|_{\mathscr{W}}\|u\|_{\mathscr{K}}.
  \end{split}
  \end{equation}
  where $\bar{\varphi}$ denotes a quasi-everywhere precisely defined version of $\varphi$ in  \eqref{eq-q.s.c.version}. Hence, in view of Definition \ref{def-capacity}, we see that for any $u\in\mathcal{P}$, $\mu^u$ does not charge polar sets.
\end{rmk}

%

To approximate the obstacle for deterministic parabolic PDEs, Pierre \cite{Pierre-1979} introduced the following lemma, from which we shall derive a useful corollary.

\begin{lem}[Proposition II-2, Page 1165 of \cite{Pierre-1979}]\label{lem-approx}
  Suppose that $v:Q\rightarrow \bR$ is quasi-continuous and there exists $u\in\mathcal{P}$ such that $|v|\leq u$ q.e.. Then there exist $\phi\in\mathcal{P}$ and $\{v_n;n\in\mathbb{N}\}\subset \mathscr{W}\cap \mathscr{C}(Q) $, such that
  $$
  \|v-v_n\|_{\phi}:=\inf\{\alpha;\,|v-v_n|\leq \alpha \phi,\,\, q.e.\}\longrightarrow 0,\textrm{ as }n\rightarrow +\infty.
  $$
\end{lem}

\begin{cor}\label{cor-lem-approx}
  Under the hypothesis of Lemma \ref{lem-approx}, there exist $\phi\in\mathscr{W}^+$, $\{v_n;n\in\mathbb{N}\}\subset \mathscr{W}\cap \mathscr{C}(Q) $, and $\{\theta_n;n\in\bN\}\subset \bN^{-1}$ such that $\theta_n$ converges decreasingly to $0$ and
  $$
  |v-v_n|\leq \theta_n\phi,\,\,\, \|\phi\|_{\mathscr{W}}\leq C\|u\|_{\mathscr{K}},\,\,\,
  \|v_n\|_{\mathscr{W}}\leq n,\quad n=1,2,\cdots,
  $$
  where the constant $C$ is independent of $v$, $u$, $n$ and $\phi$.
\end{cor}
\begin{proof}
  When $u\equiv 0$, the proof is trivial. Thus, we assume $\|u\|_{\mathscr{K}}>0$.
  By Lemma \ref{lem-approx}, there exist  $\{\tilde{v}_n;n\in\bN\}\subset \mathscr{C}(Q)\cap \mathscr{W}$, $\bar{\phi}\in\mathcal{P}$ and $\{\alpha_n;n\in\bN\}$ converging decreasingly to $0$ such that
  $|\tilde{v}_n-v|\leq \alpha_n\bar{\phi}$, $\forall\, n\in\bN$. In fact, we can always take $\|\bar{\phi}\|_{\mathscr{K}}\leq \|u\|_{\mathscr{K}}$, otherwise replace $(\alpha_n,\bar{\phi})$ by $\left(\frac{1\vee \|\bar{\phi}\|_{\mathscr{K}}}{\|u\|_{\mathscr{K}}}\alpha_n, \frac{\|u\|_{\mathscr{K}}}{1\vee \|\bar{\phi}\|_{\mathscr{K}}}\bar{\phi}\right)$.

  By Lemma \ref{lem-K-W}, there exists $\phi\in\mathscr{W}^{+}$ such that
  $$\bar{\phi}+u\leq \phi, \textrm{ a.e. and }
  \|\phi\|_{\mathscr{W}}\leq\, C\|\bar{\phi}+u\|_{\mathscr{K}} \leq \, C\|u\|_{\mathscr{K}},
  $$
  with constant $C$ being independent of $\phi$, $u$ and $\bar{\phi}$.
  Setting
  $$\beta_n=\inf\{k\geq 1;\alpha_k\leq \frac{1}{n+1}\},\quad\hat{v}_n=\tilde{v}_{\beta_n}, \quad n=1,2,\cdots,$$
   we have $|\hat{v}_n-v|\leq \frac{1}{n}\phi$, $\forall \,n\in\bN$.
  For each $n\in\bN$, set
  $$
  \tilde{\gamma}_n=\inf\{k\geq 1; \|\hat{v}_k\|_{\mathscr{W}}>n\},\,\, \gamma_n=\tilde{\gamma}_n \wedge n,
  \,\,\,\theta_n=\frac{1}{\gamma_n}.
  $$
For each $k\in\bN$, define
\begin{equation*}
  \quad v_k=
  \begin{cases}
    0,\,\,&\textrm{if }\gamma_k=1;\\
    \hat{v}_{\gamma_k-1},\,\,&\textrm{if } \gamma_k>1.
  \end{cases}
\end{equation*}
Then, $\{\theta_n;n\in\bN\}\subset \bN^{-1}$ converges decreasingly to $0$ and for each $n\in\bN$,
$$
|v_n-v|\leq \theta_n \phi,\quad \|v_n\|_{\mathscr{W}}\leq n.
$$
 We complete the proof.
\end{proof}


Now, we are in a position to execute careful investigations, which are listed in the following proposition and will be used frequently in what follows.

\begin{prop}\label{prop-quasi-conti-Cont_TL2}
There hold the following assertions.

  (i) Each $u\in\mathscr{U}$ admits an almost surely quasi-continuous version and there exists $\xi\in \mathscr{L}^2(\mathcal{P})$ such that $\xi$ is almost surely quasi-continuous and $|u|\leq \xi$, $ d\mathbb{P}\otimes dt \otimes dx$-a.e..

  (ii) Let sequence $\{u^n\}$ be bounded in $\mathscr{L}^2(\mathcal{P})$ and converge weakly to some $u$ in $\cL^2(H^1_0)$. Let sequence $\{v_n\}$ consist of almost surely quasi-continuous elements and for each $n$, $|v_n|\leq v_0$ with $v_0\in\mathscr{L}^2(\mathcal{P})$. Suppose that there exist almost surely quasi-continuous function $v$ and $\{\phi_n\}\subset \mathscr{U}$ converging decreasingly to $0$ such that
  $$
  |v-v_n|\leq \phi_n,\, d\mathbb{P}\otimes dt \otimes dx\textrm{-a.e..}
  $$
  Then $u\in\mathscr{L}^2(\mathcal{P})$ and
  $$
  \lim_{n\rightarrow +\infty}\int_{[0,T)\times\cO}v_n\,d\mu^{u^n}=\int_{[0,T)\times\cO}v\,d\mu^{u},\,\,a.s.,
  $$
  where $\mu^{u^n}$ and $\mu^u$ are the stochastic Radon measures associated with $u^n$ and $u$ respectively.

  (iii) Suppose that $v\in\mathscr{L}^2(L^{\infty}(0,T;L^2))$ is almost surely quasi-continuous and there exists $u\in\mathscr{L}^2(\mathcal{P})$ such that $|v|\leq u$ q.e., a.s.. Then there exist $\{\phi_n;n\in\mathbb{N}\}\subset\mathscr{U}$ and $\{v_n;n\in\mathbb{N}\}\subset \mathscr{U} $, such that
  $\{\phi_n;n\in\mathbb{N}\}$ converges decreasingly to $0$, $ d\mathbb{P}\otimes dt\otimes dx\textrm{-a.e.}$,
  $$
  \lim_{n\rightarrow +\infty}\|\phi_n\|_{\mathscr{U}}=0 \textrm{ and } |v-v_n|\leq \phi_n,\,\, d\mathbb{P}\otimes dt\otimes dx\textrm{-a.e.},\,\,n=1,2,\cdots.
  $$

  (iv) Let $u\in L^2(\Omega,\sF_T;\mathcal{P})$ be almost surely quasi-continuous. Then $u\in L^2(\Omega,\sF_T;C([0,T];L^2))$ and there exist $\{ u_n;\,\,n\in\mathbb{N}  \}\subset L^2(\Omega,\sF_T;\mathscr{W})\cap L^2(\Omega,\sF_T;\mathcal{P})$ and $v\in L^2(\Omega,\sF_T;\mathscr{W})$ such that
  \begin{equation}\label{eq-prop-0}
  \begin{split}
  \lim_{n \rightarrow +\infty}\|u-u_n\|_{L^2(\Omega,\sF_T;\mathscr{K})}=\lim_{n\rightarrow +\infty}\|u-u_n\|_{v}=0,\quad a.s..
  \end{split}
  \end{equation}
  In particular, if $u\in L^2(\Omega,\sF_T;\mathcal{P}_0)$, we can choose the above sequence $\{u_n;\,\,n\in\mathbb{N}\}\subset L^2(\Omega,\sF_T;\mathscr{W})\cap L^2(\Omega,\sF_T;\mathcal{P}_0)$.
\end{prop}
\begin{proof}
 We first prove assertion (i). In view of the definition of $\mathscr{U}$, there exist $v\in\cL^2((L^2)^m)$, $\psi\in L^2(\Omega,\sF_T;L^2)$ and $(\bar{f},\bar{g})\in \cL^2(L^2)\times \cL^2((L^2)^d)$ such that $u\in\mathscr{U}(I,0,\psi,\bar{f},\bar{g})$ with $v$ being the diffusion term of $u$. Set $(\bar{f},\bar{g},v)(t)=0$, $\forall\,t\in(T,T+1]$ and let $u1_{[T,T+1]}$ solve stochastic PDE
 $$
 \partial_t u(t,x)=\Delta u(t,x),\,\,(t,x)\in(T,T+1]\times \cO;\quad u(T)=\psi.
 $$
Consider \emph{forward} SPDE
\begin{equation}\label{eq-spde-prop-CTL}
  \left\{\begin{array}{l}
  \begin{split}
  -d\tilde{u}(t,x)=\,&\displaystyle \Bigl[\Delta (2 u 1_{[0,T]}-\tilde{u})   +(\bar{f}+\nabla\cdot \bar{g} )\Bigr](t,x)\,dt     -v^{r}(t,x)\, dW_{t}^{r}, \quad
                     (t,x)\in [0,T+1]\times\cO;\\
    \tilde{u}(0,x)=\, &u(0,x), \quad x\in\cO.\\
    \end{split}
  \end{array}\right.
\end{equation}
By the uniqueness of the solution, we must have $\tilde{u}(t)=u(t)$, for all $t\in [0,T]$. From the quasi-continuity of the solutions for SPDEs (see \cite[Theorem 3]{Denis-Matous-Zhang-2012}), we conclude that $\tilde{u}$ is quasi-continuous on $[0,T+1)\times\cO$. Consequently, $u$ is endowed with a quasi-continuous version on $Q$. On the other hand, as $u$ satisfies \emph{forward} SPDE \eqref{eq-spde-prop-CTL} on time interval $[0,T]$, by Proposition \ref{prop-obst-SPDE}, there exist $\xi^1,\xi^2\in\mathscr{L}^2(\mathcal{P})$ with $\xi^1$ and $\xi^2$ being almost surely quasi-continuous, such that
$-\xi^1\leq u\leq \xi^2$, $ d\mathbb{P}\otimes dt \otimes dx$-a.e.. Taking $\xi=\xi^1+\xi^2$, we prove assertion (i).
\medskip

 As for (ii), we note that by Proposition \ref{prop-obst-SPDE} and the comparison principle for the obstacle problems of SPDEs (see \cite[Theorem 8]{Denis-Matous-Zhang-2012}), there exists a decreasing sequence $\{\tilde{\phi}_n;n\in\bN\}\subset \mathscr{L}^2(\mathcal{P})$ such that for each $n\in\bN$, $\phi_n\leq \tilde{\phi}_n$ and $\|\tilde{\phi}_n\|_{\mathscr{L}^2(\mathcal{P})}\leq C \|\phi_n\|_{\mathscr{U}}$. Following the proofs of \cite[Theorem III.1]{Pierre-1980} and \cite[Lemma III.8]{Pierre-1980} for almost every $\omega\in\Omega$, we prove assertion (ii).
\medskip

    Using Corollary \ref{cor-lem-approx} for every $\omega\in\Omega$ and in view of the separability of $\mathscr{W}$, we conclude that there exist $\hat{\phi}\in L^2(\Omega,\sF_T;\mathscr{W})$ and $\{\hat{v}_n;n\in\bN\}\subset L^2(\Omega,\sF_T;\mathscr{W})$ and $\{\alpha_n;n\in\bN\}\subset L^{\infty}(\Omega,\sF_T;\bN^{-1})$ such that $\alpha_n$ converges decreasingly to $0$ almost surely and
   $|v-\hat{v}_n|\leq \alpha_n\hat{\phi}$, $n=1,2,\cdots.$
   Set
   $$
   \phi_n(t)=E[\alpha_n\hat{\phi}(t)|\sF_t] \textrm{ and } v_n(t)=E[\hat{v}_n(t)|\sF_t],\quad t\in[0,T],\,n\in\bN.
   $$
   Then, $\phi_n$ converges decreasingly to $0$, $ d\mathbb{P}\otimes dt\otimes dx\textrm{-a.e.}$, $|v-v_n|\leq \phi_n,\,\, d\mathbb{P}\otimes dt\otimes dx\textrm{-a.e.},\,\,n=1,2,\cdots$. Moreover, by Lemma \ref{lem-BSPDE-condition-expect}, $\phi_n,v_n\in\mathscr{U}$ for each $n\in\bN$ and
   $$\lim_{n\rightarrow \infty}\|\phi_n\|_{\mathscr{U}}\leq C
   \lim_{n\rightarrow \infty}\|\alpha_n\hat{\phi}\|_{L^2(\Omega,\sF_T;\mathscr{W})}=0.  $$
   Hence, $(iii)$ is proved.
   \medskip

   By \cite[Lemma II-7]{Pierre-1979}, for almost every $\omega\in\Omega$, there exist a sequence $\{ u_n;\,\,n\in\mathbb{N}  \}\subset \mathscr{W}\cap \mathcal{P}$ and $v\in \mathscr{W}$ such that
  \begin{equation*}
  \begin{split}
  \lim_{n \rightarrow +\infty}\|u-u_n\|_{L^2(0,T;H_0^1)}=\lim_{n\rightarrow +\infty}\|u-u_n\|_{v}=0.
  \end{split}
  \end{equation*}
  Moreover, choosing subsequence if necessary, we take
  \begin{align}
    \|u-u_n\|_{L^2(0,T;H_0^1)}\leq 2^{-n} \|u\|_{L^2(0,T;H_0^1)}, a.s.,\quad \textrm{for each } n\in\bN.\label{eq-prf01}
  \end{align}

  In a similar way to Corollary \ref{cor-lem-approx}, we can choose $\{u_n\}$ and $v$ such that
$$
\|v\|_{\mathscr{W}}\leq C\|u\|_{\mathscr{K}},\,\, \|u_n\|_{\mathscr{W}}\leq n,\,\,|u-u_n|\leq \theta_n v,\,a.s.\quad n=1,2,\cdots,
$$
where the constant C is independent of $v$, $u$ and $n$, and $\{\theta_n;n\in\bN\}$ is $\bN^{-1}$-valued and converges decreasingly to $0$ almost surely. Furthermore, in view of the separability of $\mathscr{W}$, we choose $\{u_n;n\in\bN\}\subset L^2(\Omega,\sF_T;\mathscr{W})\cap L^2(\Omega,\sF_T;\mathcal{P})$ and $v\in L^2(\Omega,\sF_T;\mathscr{W})$.

  As $\mathscr{W}$ is continuously embedded into $C([0,T];L^2)$ and
  $$
  \sup_{t\in[0,T]}\|u-u_n\|\leq \theta_n\sup_{t\in[0,T]}\|v\|\rightarrow 0,\textrm{ a.s., as }n\rightarrow +\infty,
  $$
   it follows that $u\in L^2(\Omega,\sF_T;C([0,T];L^2))$. In view of \eqref{eq-prf01}, we further have
     \begin{align}
    \|u-u_n\|_{\mathscr{K}}\leq ( \theta_n+2^{-n})\left(\|u\|_{L^2(0,T;H_0)}+C\|v\|_{\mathscr{W}}\right), a.s.,\quad \textrm{for each } n\in\bN,\label{eq-prf-001}
  \end{align}
  with positive constant $C$ being independent of $n$. Consequently, $\lim_{n \rightarrow +\infty}\|u-u_n\|_{L^2(\Omega,\sF_T;\mathscr{K})}=0$.

  For the particular case where $u\in L^2(\Omega,\sF_T;\mathcal{P}_0)$, letting $\tilde{u}^n\in\mathscr{W}$ be the weak solution of PDE
  $$
  \partial_t \tilde{u}^n(t,x)=\,\displaystyle \Delta \tilde{u}^n(t,x),\,(t,x)\in Q;\quad \tilde{u}^n(0,x)=\, u_n(0,x), \quad x\in\cO,
  $$
  and taking $\bar{u}_n=u_n-\tilde{u}^n$ for each $n\in\mathbb{N}$,
  we must have $\bar{u}_n\in L^2(\Omega,\sF_T;\mathscr{W})\cap L^2(\Omega,\sF_T;\mathcal{P}_0)$, and
  $$\lim_{n\rightarrow +\infty} \| \tilde{u}^n\|_{L^2(\Omega,\sF_T;\mathscr{K})}=\lim_{n\rightarrow +\infty}\|\tilde{u}^n\|_v=\lim_{n \rightarrow +\infty}\|u-\bar{u}_n\|_{L^2(\Omega,\sF_T;\mathscr{K})}=\lim_{n\rightarrow +\infty}\|u-\bar{u}_n\|_{v}=0,\,a.s..$$
   The proof is complete.
\end{proof}

Define
\begin{align}
  (H^{-1})^+=&\{f\in H^{-1}: \langle\phi,\,f \rangle_{1,-1}\geq 0,\,\forall\,\phi\in H^{1}_0,\,\phi\geq 0\},\nonumber \\
  \mathscr{W}_{\Delta}=&
  \{\phi\in\mathscr{W}:\,-\partial_t \phi(t)-\Delta\phi(t) \in (H^{-1})^+,\textrm{ for almost every }t\in(0,T)\}, \nonumber\\
  \mathscr{R}=&
  \{\phi\in\mathscr{W}:\,-\partial_t \phi(t)-\Delta\phi(t) \in L^2,\textrm{ for almost every }t\in(0,T)\}. \nonumber
\end{align}

\begin{rmk}\label{rmk-dense-H-1}
Setting
   $h=\sum_{n=2}^{\infty}\delta_{1/n}-\delta_{1/n+1/n^2}$, where $\delta$ denotes the dirac function, we have
   $$h\in H^{-1}((0,1)) \setminus (H^{-1}(0,1))^+-(H^{-1}(0,1))^+.$$
    Therefore, $H^{-1}\neq  (H^{-1})^+-(H^{-1})^+\supset L^2$ (see \cite{Hanouzet-Joly-1979} for more details). On the other hand, $(L^2)^+:=\{\max\{h,0\}:\,h\in L^2\}$ is dense in $(H^{-1})^+$, i.e., for each $h\in (H^{-1})^+$, there exists $\{h_n\}\subset (L^2)^+$ such that $\lim_{n\rightarrow \infty} \|h_n-h\|_{-1}=0$. Indeed, let $u\in C([0,1]; H^{-1})\cap C((0,1];H_0^1)$ be the weak solution (see \cite{Ladyzhenskaia_68}) of parabolic PDE:
   $$
   \partial_t u=\Delta u;\quad u(0)=h.
   $$
   Then, $\lim_{n\rightarrow \infty}\|u(\frac{1}{n})-h\|_{-1}=0$ with $u(\frac{1}{n})\in (L^2)^+$ for each $n\in\bN$.
\end{rmk}

\begin{rmk}\label{rmk-prop-quasi-cont}
  In assertion (iv) of Proposition \ref{prop-quasi-conti-Cont_TL2}, we have $u_n\in \mathscr{W}_{\Delta}$, a.s.. If we assume further  $u\in\mathscr{L}^2(\mathcal{P})$, then $u\in \cS^2(L^2)$ and by taking conditional expectations in the proof,
  $$
  \bar{u}_n(t)=E[u_n(t)|\sF_t],\,\,v_n(t)=E[\theta_n v(t)|\sF_t],\quad \forall\, t\in[0,T],
  $$
  we have $\{\bar{u}_n;n\in\bN\} \cup\{ v_n;n\in\bN\}\subset \mathscr{U}$ with
  $\{v_n;n\in\mathbb{N}\}$ converging decreasingly to $0$, $ d\mathbb{P}\otimes dt\otimes dx\textrm{-a.e.}$, and
  $$
  \lim_{n\rightarrow +\infty}\|u-\bar{u}_n\|_{\mathscr{L}^2(\mathscr{K})}=\lim_{n\rightarrow +\infty}\|v_n\|_{\mathscr{U}}=0 \textrm{ and } |u-\bar{u}_n|\leq v_n,\,\, d\mathbb{P}\otimes dt\otimes dx\textrm{-a.e.},\,\,n=1,2,\cdots.
  $$

Assume further $u\in\mathscr{L}^2(\mathcal{P}_0)$. By (iv) of Proposition \ref{prop-quasi-conti-Cont_TL2}, we are allowed to take $\bar{u}_n(0)=0$, for each $n\in\bN$. For each $n$, there exists $(\tilde{f}_n,g_n)\in\cL^2(L^2)\times\cL^2((L^2)^d)$ such that $\bar{u}_n\in\mathscr{U}(I,0,u_n(T),\tilde{f}_n,g_n)$.  On the other hand,
there exists
  $\hat{f}_n\in L^2(\Omega\times[0,T];(H^{-1})^+)$
  such that
   $$
   \partial_t u_n=\Delta u_n +\hat{f}_n.
   $$
   Taking conditional expectations
   $$f_n(t)=E\left[\hat{f}_n(t)\big|\sF_t\right],\quad \forall t\in[0,T],$$
    we have $f_n=-\tilde{f}_n-\nabla\cdot g_n-2\Delta \bar{u}_n\in\cL^2((H^{-1})^+)$.
    Then, $\bar{\mu}_n(dt,dx):=f_n(t,x)dtdx$ is a stochastic regular measure associated with some stochastic regular potential $u^{\bar{\mu}_n}$. In particular, for each $n$, there exists $\bar{v}_n\in\cL^2((L^2)^m)$ such that
    \begin{equation*}
  \left\{\begin{array}{l}
  \begin{split}
  d\bar{u}_n(t,x)=\,&\displaystyle \Bigl[\Delta \bar{u}_n(t,x)  +f_n(t,x)\Bigr]\,dt   +\bar{v}_n(t,x)\, dW_{t}, \quad
                     (t,x)\in Q;\\
    \bar{u}_n(0,x)=\, &0, \quad x\in\cO.
    \end{split}
  \end{array}\right.
\end{equation*}
Letting $\check{u}_n^v$ solve SPDE
   \begin{equation*}
  \left\{\begin{array}{l}
  \begin{split}
  d\check{u}_n^v(t,x)=\,&\displaystyle \Delta \check{u}_n^v(t,x)  \,dt   +\bar{v}_n^{r}(t,x)\, dW_{t}^{r}, \quad
                     (t,x)\in Q;\\
    \check{u}_n^v(0,x)=\, &0, \quad x\in\cO,
    \end{split}
  \end{array}\right.
\end{equation*}
we have $\bar{u}_n=\check{u}^{v}_n+u^{\bar{\mu}_n}$. It\^o  formula yields
\begin{align}
  E\bigg[ \|\bar{u}_n(T)\|^2+2\!\int_0^T\!\!\!\|\nabla \bar{u}_n(s)\|^2ds \bigg]
  &=E\bigg[\int_0^T\!\! 2\langle\bar{u}_n(s),\,f_n(s)\rangle_{1,-1}\,ds
  +\!\int_0^T\!\!\!\|\bar{v}_n(s)\|^2\,ds\bigg].\label{eq-rmk-prop-iv}
\end{align}
By \cite[Lemma II-6]{Pierre-1979},
\begin{align}
  E\bigg[ \|u(T)\|^2+2\!\int_0^T\!\!\!\|\nabla u(s)\|^2ds \bigg]
  =E\bigg[\int_Q 2u(s,x)\,\mu(ds,dx)\bigg],\label{eq-rmk-prop-iv-1}
\end{align}
where $\mu$ is the stochastic Radon measure associated with $u$.
As
\begin{align*}
  \left|\int_0^T\!\! \langle\bar{u}_n(s),\,\hat{f}_n(s)\rangle_{1,-1}\,ds\right|
  \leq
  \int_0^T\!\! \langle u(s)+v_1(s),\,\hat{f}_n(s)\rangle_{1,-1}\,ds
  &\leq
  C \|u+v_1\|_{\mathscr{K}}\|u_n\|_{\mathscr{K}}\\
  (\textrm{by \eqref{eq-prf-001}})\,\,
  &\leq
  C \|u+v_1\|_{\mathscr{K}}\left(\|u\|_{\mathscr{K}}+\|v\|_{\mathscr{W}}\right),
\end{align*}
by Lebesgue's domination convergence theorem and (ii) of Proposition \ref{prop-quasi-conti-Cont_TL2}, it follows that
$$
\lim_{n\rightarrow\infty}
E\bigg[\int_0^T\!\! \langle\bar{u}_n(s),\,f_n(s)\rangle_{1,-1}\,ds\bigg]
=
\lim_{n\rightarrow\infty}
E\bigg[\int_0^T\!\! \langle\bar{u}_n(s),\,\hat{f}_n(s)\rangle_{1,-1}\,ds\bigg]
=E\bigg[\int_Q u(s,x)\,\mu(ds,dx)\bigg].$$
Taking limits on both sides of \eqref{eq-rmk-prop-iv} and combining \eqref{eq-rmk-prop-iv} and \eqref{eq-rmk-prop-iv-1}, we obtain
\begin{align*}
  \lim_{n\rightarrow \infty} E\bigg[\int_0^T\!\!\! \|\bar{v}_n(s)\|^2\,ds\bigg]=0,
\end{align*}
which implies $\lim_{n\rightarrow \infty}\|\check{u}^v_n\|_{\mathcal{H}}=0$ and thus,
$$
\|u-u^{\bar{\mu}_n}\|_{\mathscr{L}^2(\mathscr{K})}
 =\|u-\bar{u}_n+\check{u}^v_n\|_{\mathscr{L}^2(\mathscr{K})}\longrightarrow 0,\textrm{ as }n\rightarrow \infty.
$$
\end{rmk}

\subsection{It\^o formula for BSPDEs with regular potentials}

Denote by $\mathcal {C}^{1,2}$ the totality of function $\psi\in C(\bR^2)$ such that drivatives $\partial_t\psi(t,y),\partial_y\psi(t,y)$ and $\partial_{yy}\psi(t,y)$ exist with $\partial_y\psi(\cdot,0)\equiv 0$ and
$$\esssup_{t\in\bR,y\in\bR\setminus \{0\}}
 \left\{\left|\partial_{yy}\psi(t,y)\right|
  +\frac{1}{y^2}\left| \partial_{t} \phi(t,y)
  -\partial_{t} \phi(t,0)  \right| \right\}<+\infty.$$

\begin{thm}\label{thm-ito}
  Let $\mu_1$ and $\mu_2$ be two stochastic regular measures. Suppose that the following BSPDE
    \begin{equation}\label{BSPDE-refl-linear-1}
  \left\{\begin{array}{l}
  \begin{split}
  -du(t,x)=\,&\displaystyle \Bigl[\partial_{x_j}\bigl(a^{ij}(t,x)\partial_{x_i} u(t,x)
        +\sigma^{jr}(t,x) v^{r}(t,x)     \bigr)+(\bar{f}+\nabla\cdot \bar{g} )(t,x)\Bigr]\,dt \\
        &\displaystyle     +\mu_1(dt,x)-\mu_2(dt,x) -v^{r}(t,x)\, dW_{t}^{r}, \quad
                     (t,x)\in Q;\\
    u(T,x)=\, &G(x), \quad x\in\cO,
    \end{split}
  \end{array}\right.
\end{equation}
holds in the weak sense, i.e.,
   for any $\zeta\in \mathcal{D}$ there holds almost surely
  \begin{equation*}
    \begin{split}
      \langle \zeta(t),\,u(t)\rangle
      =&\langle \zeta(T),\, G\rangle+\!\int_t^T\!\! \langle \zeta(s),\, \bar{f}(s)\rangle ds -\!\int_t^T\!\!\langle \partial_s \zeta(s),\, u(s) \rangle ds
      -\!\int_t^T \!\!\langle \zeta(s),\, v^r(s)\rangle  dW_s^r\\
      &+\!\int_t^T\int_{\cO}\!\!\zeta(s,x)\mu_1(dt,dx)-\!\int_t^T\int_{\cO}\!\!\zeta(s,x)\mu_2(dt,dx)
      \\
      &-\!\int_t^T\!\! \langle \partial_{x_j} \zeta(s),\, a^{ij}\partial_{x_i} u(s)+\sigma^{j r}v^r(s)
      +\bar{g}^j(s) \rangle ds, \quad \forall \ t\in[0,T],
    \end{split}
  \end{equation*}
where $(u,v)\in \mathcal{H}\times \cL((L^2)^m)$,
$a$ and $\sigma$ satisfy the super-parabolicity and boundedness conditions of assumption $(\mathcal{A}2)$, $G\in L^2(\Omega,\sF_T;L^2)$ and $\bar{f},\,\bar{g}^i\in \cL^2(L^2)$, $i=1,2,\cdots,d$.
Then for each $\Phi\in\mathcal{C}^{1,2}$, there holds with probability 1
\begin{equation}\label{eq-ito}
  \begin{split}
    &\int_{\cO}\Phi(t,u(t,x))\,dx+\frac{1}{2}\int_t^T\!\!\langle \partial_{yy}\Phi(s,u(s)),\, |v(s)|^2\rangle \,ds\\
    =&
    \int_{\cO}\!\!\Phi(T,G(x))\,dx-\!\int_t^T\!\!\int_{\cO}\!\!\partial_{s}\Phi(s,u(s,x))\,dxds
         +\!\int_t^T\!\!\langle \partial_y\Phi(s,u(s)),\, \bar{f}(s)\rangle \,ds          \\
   &+\!\int_t^T\!\!\int_{\cO} \partial_y\Phi(s,u(s,x))\,\mu_1(ds,dx)-\!\int_t^T\!\!\int_{\cO} \partial_{y}\Phi(s,u(s,x))\,\mu_2(ds,dx)\\
       &-\!\int_{t}^T\!\!\langle \partial_{yy}\Phi(s,u(s))\partial_{x_i} u(s),\, a^{ji}(s)\partial_{x_j}  u(s) +\sigma^{ir}(s)v^{r}(s)+\bar{g}^{i}(s)\rangle  \,ds\\
    &
    -\!\int_t^T\langle \partial_y\Phi(s,u(s)),\,v^{r}(s)\rangle\, dW_s^r,\quad \forall t\in[0,T].
  \end{split}
\end{equation}
\end{thm}

\begin{proof}
 Denote by $u_k$ the stochastic regular potential associated with $\mu_k$, $k=1,2$.   Letting $\tilde{u}\in \mathcal{H}$ be the unique solution of the following SPDE
   \begin{equation}
  \left\{\begin{array}{l}
  \begin{split}
  d\tilde{u}(t,x)=\,&\displaystyle \Bigl[\Delta \tilde{u}  -(\bar{f}+\nabla\cdot \tilde{g} )\Bigr](t,x)\,dt   +v^{r}(t,x)\, dW_{t}^{r}, \quad
                     (t,x)\in Q;\\
    \tilde{u}(0,x)=\, &u(0,x), \quad x\in\cO,
    \end{split}
  \end{array}\right.
\end{equation}
with
$$\tilde{g}^j=\partial_{x_j} u+ a^{ij}\partial_{x_i} u +\sigma^{jr} v^{r}+\bar{g}^j,\quad
j=1,\cdots,d.$$
Then $u=\tilde{u}-u_1+u_2$. By (i) of Proposition \ref{prop-quasi-conti-Cont_TL2}, $u$ is almost surely quasi-continuous. We check that all the terms involved in \eqref{eq-ito} are well defined.

  By (iv) of Proposition \ref{prop-quasi-conti-Cont_TL2} and Remark \ref{rmk-prop-quasi-cont},
 there exist $$\{u_k^n;\,n\in\mathbb{N}\}\subset\mathscr{U}\cap\mathscr{L}^2(\mathcal{P}_0)\ \textrm{ and }\  \{\phi_k^n;\,n\in\mathbb{N}\}\subset \mathscr{U},$$
  such that $\{\phi_k^n;n\in\mathbb{N}\}$ converges decreasingly to $0$, $ d\mathbb{P}\otimes dt\otimes dx\textrm{-a.e.}$, and
  $$
  \lim_{n\rightarrow +\infty}\|u_k-{u}_k^n\|_{\mathscr{L}^2(\mathscr{K})}=\lim_{n\rightarrow +\infty}\|\phi_k^n\|_{\mathscr{U}}=0 \textrm{ and } |u_k-{u}_k^n|\leq \phi_k^n,\,\, d\mathbb{P}\otimes dt\otimes dx\textrm{-a.e.},\,k=1,2;\,n\in\bN.
  $$
%
%
%
%
%
  Moreover, there exist $\{f_k^n\}\subset \cL^2((H^{-1})^+)$ and $\{v_k^n;n\in\bN\}\subset\cL^2((L^2)^m)$ such that
  $\lim_{n\rightarrow\infty}\|v_k^n\|_{\cL^2((L^2)^m)}=0$, $k=1,2$, and
      \begin{equation*}
  \left\{\begin{array}{l}
  \begin{split}
  du_k^n(t,x)=\,&\displaystyle \Bigl[\Delta u_k^n(t,x)  +f_k^n(t,x)\Bigr]\,dt   +v_k^n(t,x)\, dW_{t}, \quad
                     (t,x)\in Q;\\
    u_k^n(0,x)=\, &0, \quad x\in\cO.
    \end{split}
  \end{array}\right.
\end{equation*}


For each $n$, set $u^n=\tilde{u}-u_1^n+u_2^n$. Then,
$$\lim_{n\rightarrow +\infty}\|u^n-u\|_{\mathscr{L}^2(\mathscr{K})}=0 \textrm{ and }
   |u^n-u|\leq \phi_1^n+\phi_2^n,\,\, d\mathbb{P}\otimes dt\otimes dx\textrm{-a.e.},\,n\in\bN.$$
    On the other hand, by It\^o formulas for SPDEs without random measures (see \cite[Lemma 7]{DenisMatoussiStoica2005} and \cite[Lemma 3.3]{QiuTangMPBSPDE11}), we have almost surely
\begin{equation}\label{eq-ito-n}
  \begin{split}
    &\int_{\cO}\Phi(t,u^n(t,x))\,dx+\frac{1}{2}\int_t^T\!\!\langle \partial_{yy}\Phi(s,u^n(s)),\, |v(s)-v_1^n(s)+v_2^n(s)|^2\rangle \,ds\\
    =&
    \int_{\cO}\!\!\Phi(T,u^n(T,x))\,dx-\!\int_t^T\!\!\int_{\cO}\!\!\partial_{s}\Phi(s,u^n(s,x))\,dxds
         +\!\int_t^T\!\!\langle \partial_y\Phi(s,u^n(s)),\, \bar{f}(s)\rangle \,ds          \\
   &+\!\int_t^T \langle \partial_y\Phi(s,u^n(s)),\, f^n_1(s)\rangle_{1,-1}\,ds
   -\!\int_t^T \langle\partial_{y}\Phi(s,u^n(s)),\,f^n_2(s)\rangle_{1,-1}\,ds\\
       &+\!\int_{t}^T\!\!\langle \partial_{yy}\Phi(s,u^n(s))\partial_{x_i} u^n(s),\, \partial_{x_i}  u^n(s) -\tilde{g}^{i}(s)\rangle  \,ds\\
    &
    -\!\int_t^T\langle \partial_y\Phi(s,u^n(s)),\,(v-v_1^n+v_2^n)(s)\, dW_s\rangle,\quad \forall t\in[0,T].
  \end{split}
\end{equation}
Since $\Phi\in\mathcal{C}^{1,2}$ and $\phi_1^1,\phi_2^2,\tilde{u}\in\mathscr{U}$, there exist $\hat{\tilde{u}}\in\mathscr{L}^2(\mathcal{P})$ and generic constant $K$ such that
$|\tilde{u}|+\phi_1^1+\phi_2^2 \leq \hat{\tilde{u}}$, $ d\mathbb{P}\otimes dt \otimes dx$-a.e.,
\begin{align}
  |\partial_y\Phi(\cdot,u^n)|\leq K|u^n|\leq K \left(\hat{\tilde{u}}+u_1+u_2\right),\,\, d\mathbb{P}\otimes dt \otimes dx-a.e.\,\, n\in\bN,\nonumber
\end{align}
and
\begin{align}
  |\partial_y\Phi(\cdot,u^n)-\partial_y\Phi(\cdot,u)|
  \leq K\big(|u_1^n-u_1|+|u_2^n-u_2|\big)
  \leq K\big( \phi_1^n+\phi_2^n  \big),\,\, d\mathbb{P}\otimes dt \otimes dx-a.e.\,\, n\in\bN.\nonumber
\end{align}
Thus, by (ii) of Proposition \ref{prop-quasi-conti-Cont_TL2}, we have
\begin{align}
  \lim_{n\rightarrow \infty}
  \int_t^T \langle\partial_y\Phi(s,u^n(s)),\,f^n_k(s)\rangle_{1,-1}\,ds
  =\int_{[t,T]\times\cO} \partial_y\Phi(s,u(s,x))\,\mu_k(ds,dx),\,a.s.,\quad k=1,2.\nonumber
\end{align}
For the martingale part,
\begin{align*}
  &E\left[
  \sup_{t\in[0,T]}\left|
  \int_t^T\!\!\langle \partial_y\Phi(s,u^n(s))   ,\,  (v-v_1^n+v_2^n)(s)\,  dW\rangle
  -\int_t^T\!\!\langle \partial_y\Phi(s,u(s))   ,\,  v(s)\,  dW\rangle
  \right|^2
  \right]
  \\
  &\leq C
  E\left[
  \int_0^T\!\!\left|
\langle \partial_y\Phi(s,u^n(s))   ,\,  (v-v_1^n+v_2^n)(s)\rangle
  -\langle \partial_y\Phi(s,u(s))   ,\,  v(s)\rangle
  \right|^2\,ds
  \right]
  \\
  &\leq C
   E\left[
  \int_0^T\!\!\left|
\langle \partial_y\Phi(s,u^n(s))   ,\,  (v_2^n-v_1^n)(s)\rangle
  +\langle \partial_y\Phi(s,u^n(s))-\partial_y\Phi(s,u(s))   ,\,  v(s)\rangle
  \right|^2\,ds
  \right]
  \\
  &\leq C
  \left(
  \|u^n\|_{\mathscr{L}^2(\mathscr{K})}\|v_2^n-v_1^n\|_{\cL^2((L^2)^m)}
  +\|u-u^n\|_{\mathscr{L}^2(\mathscr{K})}\|v\|_{\cL^2((L^2)^m)}
  \right)\\
  &\longrightarrow 0, \quad\textrm{as n} \rightarrow +\infty.
\end{align*}

Letting $n\rightarrow +\infty$,
we obtain
\begin{align*}
    &\int_{\cO}\Phi(t,u(t,x))\,dx+\frac{1}{2}\int_t^T\!\!\langle \partial_{yy}\Phi(s,u(s)),\, |v(s)|^2\rangle \,ds\\
    =&
    \int_{\cO}\!\!\Phi(T,u(T,x))\,dx-\!\int_t^T\!\!\int_{\cO}\!\!\partial_{s}\Phi(s,u(s,x))\,dxds
         +\!\int_t^T\!\!\langle \partial_y\Phi(s,u(s)),\, \bar{f}(s)\rangle \,ds          \\
   &+\!\int_t^T\!\!\int_{\cO} \partial_y\Phi(s,u(s,x))\,\mu_1(ds,dx)-\!\int_t^T\!\!\int_{\cO} \partial_{y}\Phi(s,u(s,x))\,\mu_2(ds,dx)\\
       &+\!\int_{t}^T\!\!\langle \partial_{yy}\Phi(s,u(s))\partial_{x_i} u(s),\, \partial_{x_i}  u(s) -\tilde{g}^{i}(s)\rangle  \,ds
    -\!\int_t^T\langle \partial_y\Phi(s,u(s)),\,v^{r}(s)\rangle\, dW_s^r
    \\
    =&
    \int_{\cO}\!\!\Phi(T,G(x))\,dx-\!\int_t^T\!\!\int_{\cO}\!\!\partial_{s}\Phi(s,u(s,x))\,dxds
         +\!\int_t^T\!\!\langle \partial_y\Phi(s,u(s)),\, \bar{f}(s)\rangle \,ds          \\
   &+\!\int_t^T\!\!\int_{\cO} \partial_y\Phi(s,u(s,x))\,\mu_1(ds,dx)-\!\int_t^T\!\!\int_{\cO} \partial_{y}\Phi(s,u(s,x))\,\mu_2(ds,dx)\\
       &-\!\int_{t}^T\!\!\langle \partial_{yy}\Phi(s,u(s))\partial_{x_i} u(s),\, a^{ji}(s)\partial_{x_j}  u(s) +\sigma^{ir}(s)v^{r}(s)+\bar{g}^{i}(s)\rangle  \,ds\\
    &
    -\!\int_t^T\langle \partial_y\Phi(s,u(s)),\,v^{r}(s)\rangle\, dW_s^r,\ a.s.,\quad  \forall t\in[0,T].
\end{align*}
We complete the proof.
\end{proof}

\begin{rmk}
  In our proof, we used It\^o formulas of \cite[Lemma 7]{DenisMatoussiStoica2005} and \cite[Lemma 3.3]{QiuTangMPBSPDE11}, which serve to study the maximum principles for SPDEs and \emph{backward} SPDEs in bounded domains respectively. It is worth noting that It\^o formulas of \cite[Lemma 7]{DenisMatoussiStoica2005} and \cite[Lemma 3.3]{QiuTangMPBSPDE11} actually hold for any domain, since the proofs therein are independent of the unboundedness of the domain.
  On the other hand, Denis et al \cite{Denis-Matous-Zhang-2012} proved a similar It\^o formula for the obstacle problems of SPDEs, while It\^o formula of Theorem \ref{thm-ito} herein is independent of the obstacle problems.
\end{rmk}

\begin{cor}\label{cor-ito}
  Under the hypothesis of Theorem \ref{thm-ito}, there holds with pobability 1
  \begin{align}
    &\|u^+(t)\|^2+\!\int_t^T\!\!\!\|v(s)1_{\{u> 0\}}\|^2\,ds\nonumber\\
    =&\|G^+\|^2
    -2\!\int_t^T\!\!\!\langle \partial_{x_j} u^+(s),\,
                 a^{ij}\partial_{x_i} u(s)+\sigma^{jr}v^r(s)+ \bar{g}^j(s)\rangle\,ds
                 +\!\int_t^T\!\!\!2\langle u^+(s),\,\bar{f}(s) \rangle\,ds\rangle\, ds\nonumber\\
    &+\!\int_{[t,T]\times\cO}\!2u^+(s,x)\,\mu_1(ds,dx)
    -\!\int_{[t,T]\times\cO}\!2u^+(s,x)\,\mu_2(ds,dx)\nonumber\\
   & -\!\int_t^T\!\!\!2\langle u^+(s),\,v(s)\,dW_s\rangle,\quad \forall t\in[0,T],     \label{eq-cor-ito}
  \end{align}
  where $u^+:=\max\{u,0\}$.
\end{cor}
\begin{proof}
  For $k\in\bN$, define
\begin{equation}\label{eq-func}
  \psi_k(s):=
  \left\{\begin{array}{l}
  \begin{split}
    0,\quad &s\in(-\infty,\frac{1}{k}];\\
    \frac{8k^2}{9}\left(s-\frac{1}{k}\right)^3,\quad &s\in (\frac{1}{k},\frac{5}{4k}];\\
    \frac{2k}{3}\left(s-\frac{5}{4k}\right)^2+\frac{1}{6}\left( s-\frac{5}{4k} \right)+\frac{1}{72k},
    \quad &s\in (\frac{5}{4k},\frac{7}{4k}];\\
    -\frac{8k^2}{9}\left( s-\frac{2}{k} \right)^3+s-\frac{3}{2k},\quad &s\in (\frac{7}{4k},\frac{2}{k}];\\
    s-\frac{3}{2k},\quad &s\in(\frac{2}{k},+\infty).
    \end{split}
  \end{array}\right.
 \end{equation}
 By Theorem \ref{thm-ito}, there holds with probability 1
 \begin{align}
   &\int_{\cO}|\psi_k(u(t,x))|^2\,dx
   +\int_{[t,T]\times\cO}\!\left(|\psi_k'|^2
   +\psi_k\psi_k''\right)\left(u(s,x)\right)|v(s,x)|^2\,dsdx
   \nonumber\\
   =&
   \int_{\cO}\!\!|\psi_k(G(x))|^2\,dx
   -\!\int_t^T\!\!\!2\langle(\psi_k\psi_k''+|\psi_k'|^2)(u(s))\partial_{x_j}  u(s),\,
                 a^{ij}\partial_{x_i} u(s)+\sigma^{jr}v^r(s)+ \bar{g}^j(s)\rangle\,ds\nonumber\\
   &+\!\int_{[t,T]\times\cO}\!2\psi_k(u(s,x))\psi_k'(u(s,x))\,\mu_1(ds,dx)
    -\!\int_{[t,T]\times\cO}\!2\psi_k(u(s,x))\psi_k'(u(s,x))\,\mu_2(ds,dx)\nonumber\\
   &+\!\int_t^T\!\!\!2\langle \psi_k\psi_k'(u(s)),\,\bar{f}(s) \rangle\,ds
   -\!\int_t^T\!\!\!2\langle\psi_k\psi_k'(u(s)),\,v(s)\,dW_s\rangle,\quad \forall t\in[0,T]. \label{eq-cor-ito-1}
 \end{align}
 In view of \eqref{eq-func}, we have for any $(s,x)\in [0,T]\times\cO$,
\begin{align*}
      \left|\psi_k(u(s,x))\right|
         \left| \psi_k''(u(s,x))   \right|
    \leq
        \ \frac{1}{2k}\times \frac{4k}{3} 1_{[\frac{1}{k},\frac{2}{k}]}(u(s,x))
    \leq
         \  1_{[\frac{1}{k},\frac{2}{k}]}(u(s,x)).
\end{align*}
On the other hand, we check that $\lim_{k\rightarrow \infty}\|\psi_k(u)-u^+\|_{\mathcal{H}}=0$.
 Therefore, by the dominated convergence theorem and taking limits in $L^1([0,T]\times\Omega,\sP;\bR)$
 on both sides of \eqref{eq-cor-ito-1}, we prove our assertion.
\end{proof}

\section{Existence and uniqueness of the solution to RBSPDE}

\subsection{Solution for RBSPDE \eqref{RBSPDE}}

First, we introduce the assumption on the obstacle process $\xi$.

\medskip
   $({\mathcal A} 4')$ \it $\xi$ is almost surely quasi-continuous on $Q$ and there exist $(\tilde{\xi},\tilde{v})\in\mathcal{H}\times \cL((L^2)^m)$ and a stochastic regular measure $\tilde{\mu}$ such that $\xi\leq\tilde{\xi}$, $ d\mathbb{P}\otimes dt\otimes dx$-a.e. and
     \begin{equation}\label{BSPDE-ass-Obstac}
  \left\{\begin{array}{l}
  \begin{split}
  -d\tilde{\xi}(t,x)=\,&\displaystyle \Bigl[\partial_{x_j}\bigl(a^{ij}(t,x)\partial_{x_i} \tilde{\xi}(t,x)
        +\sigma^{jr}(t,x) \tilde{v}^{r}(t,x)     \bigr)+(\tilde{f}+\nabla\cdot \tilde{g} )(t,x)\Bigr]\,dt \\
        &\displaystyle     +\tilde{\mu}(dt,x) -\tilde{v}^{r}(t,x)\, dW_{t}^{r}, \quad
                     (t,x)\in Q;\\
    \tilde{\xi}(T,x)=\, &\tilde{\xi}_T(x), \quad x\in\cO,
    \end{split}
  \end{array}\right.
\end{equation}
    holds in the weak sense, where $\tilde{\xi}_T\in L^2(\Omega,\sF_T;L^2)$ and $\tilde{f},\tilde{g}^i\in \cL^2(L^2)$, $i=1,2,\cdots,d$.\medskip
\medskip\\
In $({\mathcal A} 4')$, let $u^{\tilde{\mu}}\in \mathscr{L}^2(\mathcal{P})$ be the  stochastic regular potential associated with the stochastic regular measure $\tilde{\mu}$ and let $(\check{\xi},\check{v})\in \mathcal{H}\times \cL((L^2)^m)$ satisfy BSPDE
     \begin{equation*}
  \left\{\begin{array}{l}
  \begin{split}
  -d\check{\xi}(t,x)=\,&\displaystyle \Bigl[
  \Delta \check{\xi}(t,x)+\partial_{x_j}\bigl(a^{ij}\partial_{x_i} \tilde{\xi}(t,x)
        +\sigma^{jr} \tilde{v}^{r}(t,x)     \bigr)+(\tilde{f}+\nabla\cdot \tilde{g} )(t,x)
        -\Delta\tilde{\xi}(t,x)\Bigr]\,dt \\
        &\displaystyle  -2\Delta u^{\tilde{\mu}}(t,x)\,dt   -\check{v}^{r}(t,x)\, dW_{t}^{r}, \quad
                     (t,x)\in Q;\\
    \check{\xi}(T,x)=\, &\tilde{\xi}_T(x)+\tilde{u}^{\mu}(T,x), \quad x\in\cO.
    \end{split}
  \end{array}\right.
\end{equation*}
We have $\tilde{\xi}=\check{\xi}-u^{\tilde{\mu}}\leq \check{\xi}$, $ d\mathbb{P}\otimes dt\otimes dx$-a.e.. Consequently, the assumption $({\mathcal A} 4')$ is equivalent to the following $({\mathcal A} 4)$.\medskip

\medskip
   $({\mathcal A} 4)$ \it $\xi$  is almost surely quasi-continuous on $[0,T]\times\cO$ and there exists $\check{\xi}\in\mathscr{U}$ such that $\xi\leq\check{\xi}$, $ d\mathbb{P}\otimes dt\otimes dx$-a.e..

\begin{defn}\label{def-RBSPDE}
We say that a triple $(u,v,\mu)$ is a solution of RBSPDE \eqref{RBSPDE}, if

(1) $
(u,v)\in\mathcal{H}\times\cL^2((L^2)^m)$ and $\mu$ is a stochastic regular measure;

(2) RBSPDE \eqref{RBSPDE} holds in the weak sense, i.e., for each $\varphi\in\mathcal{D}_T$ and $t\in[0,T]$
\begin{equation*}
  \begin{split}
    &\langle u(t),\,\varphi(t) \rangle +\int_t^T\left[\langle u(s),\,\partial_s \varphi(s)   \rangle + \langle \partial_{x_j} \varphi(s),\,a^{ij}(s)\partial_{x_i} u(s)
        +\sigma^{jr}(s) {v}^{r}(s) \rangle \right]\,ds\\
    =\,&
    \langle G,\varphi(T) \rangle+\int_t^T\left[\langle f(s,u(s),\nabla u(s),v(s)),\,\varphi(s) \rangle-\langle g(s,u(s),\nabla u(s),v(s)),\, \nabla\varphi(s) \rangle \right]\,ds\\
    &+\int_t^T\!\!\int_{\bR^d}\varphi(s,x)\mu(ds,dx)-\int_t^T\langle \varphi(s),\,v^r(s)\,dW_s^r\rangle ;
  \end{split}
\end{equation*}

(3) $u$ is almost surely quasi-continuous, $u(t,x)\geq \xi(t,x)$, $ d\mathbb{P}\otimes dt \otimes dx$-a.e. and
$$
\int_0^T\!\!\int_{\cO}\left(u(s,x)-\xi(s,x)\right)\mu(ds,dx)=0,\quad a.s..
$$
\end{defn}

By It\^o formulas in Theorem \ref{thm-ito} and Corollary \ref{cor-ito}, we are ready for the comparison principle.

\begin{thm}\label{thm-comprsn}
  Suppose that $(u,v,\mu)$ is a solution of RBSPDE \eqref{RBSPDE} under assumptions $(\mathcal{A}1)-(\mathcal{A}4)$. Let $(\xi_1,f_1,G_1)$ be another triple which together with $(a,\sigma,g)$ satisfies assumptions $(\mathcal{A}1)-(\mathcal{A}4)$. Let $(u_1,v_1,\mu_1)$ be a solution of RBSPDE \eqref{RBSPDE} associated with $(a,\sigma,G_1,f_1,g,\xi_1)$. Suppose further that
  $$
  f(u,\nabla u,v)\leq f_1(u,\nabla u,v),\,\,\,\xi\leq \xi_1,\,\, d\mathbb{P}\otimes dt\otimes dx-a.e.
  \textrm{ and }G\leq G_1,\,\, d\mathbb{P}\otimes dx-a.e..
  $$
  Then, with probability 1 there holds $u(t,x)\leq u_1(t,x)$, q.e..
\end{thm}
\begin{proof}
  Set $(\tilde{u},\tilde{v})=(u-u_1,v-v_1)$. By Corollary \ref{cor-ito}, we have
  \begin{align}
    &E\bigg[ \|\tilde{u}^+(t)\|^2+\!\int_t^T\!\!\!\|\tilde{v}(s)1_{\{u> u_1\}}\|^2\,ds \bigg]
      \nonumber\\
    =&
        E\bigg[
         -\!\int_t^T\!\!\!2\langle
            \partial_{x_j}\tilde{u}^+(s),\,a^{ij}(s)\partial_{x_i}\tilde{u}^+(s)+\sigma^{jr}(s)\tilde{v}^r(s)
            +g^j(s,u,\nabla u,v)-g^j(s,u_1,\nabla u_1,v_1)
                \rangle\,ds
        \nonumber\\
        &\quad
            \,+\int_t^T\!\!\!2\langle
            \tilde{u}^+(s),\,f(s,u,\nabla u,v)-f_1(s,u,\nabla u,v)+f_1(s,u,\nabla u,v)-f_1(s,u_1,\nabla u_1,v_1)
              \rangle\,ds
        \nonumber\\
        &\quad
            \,+\int_{[t,T]\times\cO}\!(u-\xi+\xi-\xi_1+\xi_1-u_1)^+(s,x)\,(\mu-\mu_1)(ds,dx)
            \bigg]
        \nonumber\\
    \leq&\,
        E\bigg[
         -\!\int_t^T\!\!\!2\langle
            \partial_{x_j}\tilde{u}^+(s),\,a^{ij}(s)\partial_{x_i}\tilde{u}^+(s)+\sigma^{jr}(s)\tilde{v}^r(s)
            +g^j(s,u,\nabla u,v)-g^j(s,u_1,\nabla u_1,v_1)
                \rangle\,ds
        \nonumber\\
        &\quad
            \,+\int_t^T\!\!\!2\langle
            \tilde{u}^+(s),\, f_1(s,u,\nabla u,v)-f_1(s,u_1,\nabla u_1,v_1)
              \rangle\,ds
                \bigg]
        \nonumber\\
    \leq&\,
          E\bigg[
            -\int_t^T\!\!\!  \langle \partial_{x^i}\tilde{u}^+(s),\,(2a^{ij}-\varrho \sigma^{jr}\sigma^{ir})\partial_{x^j}\tilde{u}^+(s)\rangle\, ds
            +\int_t^T\!\!\frac{1}{\varrho}\|\tilde{ v}(s)1_{\{u>u_1\}}\|^2\,ds
        \nonumber\\
         &\,\,\quad
            +\int_t^T\!\!\! \Big(2\|\nabla \tilde{u}^+(s)\|\big(L\|\tilde{u}^+(s)\|+\frac{\kappa}{2} \|\nabla \tilde{u}^+(s)\| + \beta^{1/2}\|\tilde{v}(s)1_{\{u> u_1\}}\| \big)
        \nonumber \\
        &\, \quad \quad \quad
            +2 L \| \tilde{u}^+(s)\| \big( \|\tilde{u}^+(s)\|+ \|\nabla \tilde{u}^+(s)\|+
            \| \tilde{v}(s)1_{\{u> u_1\}}\|\big) \Big)\,ds
            \bigg]
        \nonumber\\
    \leq&\,
        E\bigg[
        -\!\!\int_t^T\!\!\!
        \big(\lambda-\kappa-\beta(\varrho' +2\eps)\big)\|\nabla \tilde{u}^+(s)\|^2\, ds
        +\!\!\int_t^T\!\!\!\Big(\frac{1}{\varrho}+\frac{1}{\varrho'+\eps}\Big)\|\tilde{v}(s)1_{\{u> u_1\}}\|^2
        +C
        \!\int_t^T\!\!\! \| \tilde{u}^+(s)\|^2\,ds
         \bigg],\nonumber
  \end{align}
  where  $\frac{1}{\varrho}+\frac{1}{\varrho'}=1$ and $\eps>0$ is small enough. By Gronwall inequality, we obtain
  \begin{align*}
    \sup_{t\in[0,T]}E\bigg[ \|\tilde{u}^+(t)\|^2\bigg]+E\bigg[\int_0^T\!\!\!\|\tilde{v}(s)1_{\{u> u_1\}}\|^2\,ds \bigg]=0.
  \end{align*}
  By the quasi-continuity of $u$ and $u_1$, there follows $u(t,x)\leq u_1(t,x)$, q.e., with probability 1. The proof is complete.
\end{proof}

\begin{lem}\label{lem-unique}
  Under assumptions $(\mathcal{A}1-\mathcal{A}4)$, the solution of RBSPDE \eqref{RBSPDE} is unique.
\end{lem}
\begin{proof}[Sketch of the proof]
  Let $(u_1,v_1,\mu_1)$ and $(u_2,v_2,\mu_2)$ be two solutions of RBSPDE \eqref{RBSPDE}.
  Setting $  (\delta u,\,\delta v)=(u_1-u_2,\,v_1-v_2)$, we have by Theorem \ref{thm-ito},
  \begin{align}
    &E\bigg[  \|\delta u(t)\|^2+\int_t^T \!\!\!\|\delta v(s)\|^2\,ds\bigg]\nonumber\\
    =&\,
    E\bigg[
    -\int_t^T\!\!\!  \langle 2\partial_{x^i}\delta u(s),\,a^{ij}\partial_{x^j}\delta u(s) + \sigma^{ir}\delta v^r(s)+g^i(s,u_1,\nabla u_1,v_1)-g^i(s,u_2,\nabla u_2,v_2) \rangle\,ds
    \nonumber\\
    &\,\quad+\int_t^T\!\!\! \langle 2 \delta u(s),\, f(s,u_1,\nabla u_1,v_1)-f(s,u_2,\nabla u_2,v_2)\rangle\,ds
    \nonumber\\
    &\,\quad
    +\int_{[t,T]\times \cO}\!\!2\big(u_1-\xi-(u_2-\xi)\big)(s,x)\,\big(\mu_1-\mu_{2}\big)(ds,dx)
    \bigg]\nonumber\\
    \leq &\,
    E\bigg[
    -\int_t^T\!\!\!  \langle \partial_{x^i}\delta u(s),\,(2a^{ij}-\varrho \sigma^{jr}\sigma^{ir})\partial_{x^j}\delta u (s)\rangle\, ds
    +\int_t^T\!\!\!\frac{1}{\varrho}\|\delta v(s)\|^2\,ds \nonumber\\
     &\,
     +\int_t^T\!\!\! \Big(2\|\nabla \delta u(s)\|\big(L\|\delta u(s)\|+\frac{\kappa}{2} \|\nabla \delta u(s)\|+\beta^{1/2}\|\delta v(s)\|  \big)
     \nonumber \\
     &\,\quad \quad
     +2 L \| \delta u(s)\| \big(\|\delta u(s)\|+ \|\nabla \delta u(s)\|+\|\delta v(s)\|\big) \Big)\,ds\nonumber\\
    &\,-2\int_{[t,T]\times \cO}\!\!(u_2-\xi)(s,x)\,\mu_1(ds,dx)-2\int_{[t,T]\times \cO}\!\!(u_1-\xi)(s,x)\mu_{2}(ds,dx)
    \bigg]\nonumber\\
    \leq &\,
    E\bigg[
    -\!\!\int_t^T\!\!\!
    \big(\lambda-\kappa-\beta(\varrho' +2\eps)\big)\|\nabla \delta u(s)\|^2\, ds
    +\!\!\int_t^T\!\!\!\Big(\frac{1}{\varrho}+\frac{1}{\varrho'+\eps}\Big)\|\delta v(s)\|^2
     +C(\eps,\lambda,\beta,\kappa,
     \varrho) \!\int_t^T\!\!\!\| \delta u(s)\|^2 \,ds\bigg],\nonumber
  \end{align}
  where $\frac{1}{\varrho}+\frac{1}{\varrho'}=1$ and $\eps>0$ is small enough. By Gronwall inequality, we obtain
  $$
  \sup_{t\in[0,T]}E\big[  \|\delta u(t)\|^2\big]+E\bigg[\int_t^T \!\!\!\|\delta v(s)\|^2\,ds\bigg]=0.
  $$
  Thus, $(u_1,v_1)=(u_2,v_2)$ and in view of (2) in Definition \ref{def-RBSPDE}, we further get $\mu_1=\mu_2$. We complete the proof.
\end{proof}

\begin{rmk}\label{rmk-assumpt4}
Let assumptions $(\mathcal{A}1-\mathcal{A}4)$ hold with $\xi(T)\leq G$, $d\mathbb{P}\otimes dx$-a.e.. By the theory on quasi-linear BSPDEs (see \cite[Chapter 1]{Qiu-2012}), there exists a unique solution $(\bar{u},\bar{v})\in \mathcal{H}\times \cL^2((L^2)^m)$ to the following BSPDE
    \begin{equation*}
  \left\{\begin{array}{l}
  \begin{split}
  -d\bar{u}(t,x)=\,&\displaystyle \Bigl[\partial_{x_j}\bigl(a^{ij}(t,x)\partial_{x_i} \bar{u}(t,x)
        +\sigma^{jr}(t,x) \bar{v}^{r}(t,x)     \bigr) \\
        &\displaystyle +(f+ \nabla \cdot g)(t,x,\bar{u}(t,x),\nabla \bar{u}(t,x),\bar{v}(t,x))
                \Bigr]\, dt\\ &\displaystyle
           -\bar{v}^{r}(t,x)\, dW_{t}^{r}, \quad
                     (t,x)\in Q;\\
    \bar{u}(T,x)=\, &G(x), \quad x\in\cO.
    \end{split}
  \end{array}\right.
\end{equation*}
Suppose that $(u,v,\mu)$ is the solution of RBSPDE \eqref{RBSPDE}. By Theorem \ref{thm-comprsn} and Lemma \ref{lem-unique}, $(\bar{u},\bar{v},0)$ must be the unique solution to RBSPDE \eqref{RBSPDE} associated with obstacle process $\xi \wedge \bar{u}$ and furthermore, $\bar{u}(t,x)\leq u(t,x)$, q.e.. Therefore, $(u,v,\mu)$ coincides with the solution of RBSPDE \eqref{RBSPDE} with the obstacle process being replaced by $\xi \vee \bar{u}$. In other words, assumption $({\mathcal A} 4)$ is equivalent to the following one:\medskip
\\
$(\mathcal{A}4b)$ $\xi$  is almost surely quasi-continuous on $[0,T]\times\cO$ and there exists $\check{\xi}\in\mathscr{U}$ such that $|\xi|\leq\check{\xi}$, $ d\mathbb{P}\otimes dt\otimes dx$-a.e..

\end{rmk}

\subsection{RBSPDEs with Laplacian leading coefficients}

Let $\tilde{\xi}$ satisfy assumption $(\mathcal{A}4)$ with $\tilde{\xi}(T)\leq G$, $d\mathbb{P}\otimes dx$-a.e., $G\in L^2(\Omega,\sF_T;L^2)$ and $\bar{f},\,\bar{g}^i\in \cL^2(L^2)$, $i=1,2,\cdots,d$. Consider the following reflected BSPDE
\begin{equation}\label{RBSPDE-Laplac}
  \left\{\begin{array}{l}
  \begin{split}
  -du(t,x)=\,&\displaystyle \big[\Delta u(t,x) +(\bar{f}+ \nabla \cdot \bar{g})(t,x)
                \big]\, dt+\mu(dt,x)
           -v^{r}(t,x)\, dW_{t}^{r},\,
                     (t,x)\in Q;\\
    u(T,x)=\, &G(x), \quad x\in\cO;\\
    u(t,x)\geq\,& \tilde{\xi}(t,x),\,\, d\mathbb{P}\otimes dt\otimes dx-a.e.;\quad
    \int_Q \big( u(t,x)-\tilde{\xi}(t,x) \big)\,\mu(dt,dx)=0,\,a.s..
    \end{split}
  \end{array}\right.
\end{equation}
Let $\phi\in\mathscr{U}(I,0,G,\bar{f},\bar{g})$. Then, $\xi:=\tilde{\xi}-\phi$ satisfies assumption $(\mathcal{A}4)$ with $\xi(T)\leq 0$, $d\mathbb{P}\otimes dx$-a.e. and RBSPDE \eqref{RBSPDE-Laplac} is equivalent to the following one
\begin{equation}\label{RBSPDE-Laplac1}
  \left\{\begin{array}{l}
  \begin{split}
  -du(t,x)=\,&\displaystyle \Delta u(t,x) \, dt+\mu(dt,x)
           -v^{r}(t,x)\, dW_{t}^{r},\,
                     (t,x)\in Q;\\
    u(T,x)=\, &0, \quad x\in\cO;\\
    u(t,x)\geq\,& {\xi}(t,x),\,\, d\mathbb{P}\otimes dt\otimes dx-a.e.;\quad
    \int_Q \big( u(t,x)-{\xi}(t,x) \big)\,\mu(dt,dx)=0,\,a.s..
    \end{split}
  \end{array}\right.
\end{equation}

Before solving RBSPDEs \eqref{RBSPDE-Laplac} and \eqref{RBSPDE-Laplac1}, we investigate a class of BSPDEs with stochastic measures and a related variational problem.

Set
$$\mathscr{U}_p=\{u=\bar{u}-\tilde{u};\,\tilde{u}\in \mathscr{L}^2(\mathcal{P}),\,\bar{u}\in \mathscr{U}(I,0,\tilde{u}(T),  0,-2\nabla \tilde{u})\}.$$
Thus, for each $u\in\mathscr{U}_p$, there exist $v\in\mathcal{L}^2((L^2)^m)$ and stochastic  measure $\mu$  associated with some stochastic potential
$\tilde{u}\in \mathscr{L}^2(\mathcal{P})$,  such that
\begin{equation*}
  \left\{\begin{array}{l}
  \begin{split}
  -du(t,x)=\,&\displaystyle \Delta u(t,x) \, dt+\mu(dt,x)
           -v^{r}(t,x)\, dW_{t}^{r},\,
                     (t,x)\in Q;\\
    u(T,x)=\, &0, \quad x\in\cO,
    \end{split}
  \end{array}\right.
\end{equation*}
holds in the weak sense.

For each $\zeta $ satisfying assumption $(\mathcal{A}4)$ and each $u\in \mathscr{U}$, we introduce the following variational problem
\begin{align}
\Gamma(u,\zeta)=\essinf \{u+\phi:\,\phi\in\mathscr{U}_p,\,u+\phi\geq \zeta, d\mathbb{P}\otimes dt\otimes dx\textrm{-a.e.}\}. \label{variation}
 \end{align}
 It is not hard to verify 
 \medskip
 \\
 (i) $\Gamma(u,\zeta)=\Gamma(0,\zeta-u)+u$, for each $\zeta$ satisfying assumption $(\mathcal{A}4)$ and $u\in\mathscr{U}$;\\
  (ii) $\Gamma(0,\xi)=\xi$, for each $\xi\in\mathscr{U}_p$;\\
   (iii) for any $\zeta_1,\zeta_2$ satisfying assumption $(\mathcal{A}4)$, $\Gamma(0,\zeta_1)\leq \Gamma(0,\zeta_2)+\Gamma(0,\zeta_1-\zeta_2)$.
\medskip

Define
$$
\mathcal{P}_T=\left\{u\in\mathscr{K}:\,u(T-\cdot)\in\mathcal{P} \textrm{ with }u(T-)=0\right\}.
$$
\begin{lem}\label{lem-back-potential}
  For each $\hat{u}\in L^2(\Omega,\sF_T;\mathcal{P}_T)$, there exists $u\in\mathscr{U}_p$ such that $u(t)=E[\hat{u}(t)|\sF_t]$, $\forall\, t\in[0,T]$.
\end{lem}

\begin{proof}
  Put $u(t)=E[\hat{u}(t)|\sF_t]$, $ t\in[0,T]$. It is sufficient to prove $u\in\mathscr{U}_p$.

  First, by Lemma \ref{lem-K-W}, there exists $\hat{\tilde{u}}\in L^2(\Omega,\sF_T;\mathscr{W})$ such that $\hat{u}\leq \hat{\tilde{u}}$, $ d\mathbb{P}\otimes dt\otimes dx-a.e.$ and $\|\hat{\tilde{u}}\|_{\mathscr{W}}\leq C \|\hat{u}\|_{\mathscr{K}}$, a.s.. Set $\tilde{u}(t)=E[\hat{\tilde{u}}(t)|\sF_t]$, $t\in[0,T]$. Then $\tilde{u}\in\mathscr{U}$ and there exist $(\phi,\tilde{f},\tilde{v})\in L^2(\Omega,\sF_T;L^2)\times\cL^2(H^{-1})\times\cL^2((L^2)^m)$ such that $(\tilde{u},\tilde{v})$ is the solution of BSPDE
  \begin{equation*}
  \left\{\begin{array}{l}
  \begin{split}
  -d\tilde{u}(t,x)=\,&\displaystyle \big(\Delta \tilde{u}(t,x) \, +\tilde{f}(t,x)\big)\, dt
           -\tilde{v}^{r}(t,x)\, dW_{t}^{r},\,
                     (t,x)\in Q;\\
    \tilde{u}(T,x)=\, &\phi(x), \quad x\in\cO.
    \end{split}
  \end{array}\right.
\end{equation*}
  It is clear that $u\leq \tilde{u}$, $ d\mathbb{P}\otimes dt\otimes dx$-a.e..

  For each $\delta>0$, consider $\hat{u}_{\delta}\in \mathcal{H}$ satisfying
  $$
  -\partial_t \hat{u}_{\delta}(t,x)=\Delta \hat{u}_{\delta}(t,x)+\frac{\hat{u}(t,x)-\hat{u}_{\delta}(t,x)}{\delta},\,\,\,(t,x)\in Q;\quad \hat{u}_{\delta}(T)=0.
  $$
   From assertion (i) of Lemma \ref{lemm-approx-potential}, we deduce that  $\{\hat{u}_{\delta}\}_{\delta>0}\subset L^2(\Omega,\sF_T;\mathcal{P})$ converges increasingly in $L^2(\Omega\times[0,T];L^2)$ and weakly in $L^2(\Omega\times[0,T];H^1_0)$ to $\hat{u}$, as $\delta\rightarrow 0^+$. Taking 
   \begin{align}
   u_{\delta}(t)=E[\hat{u}_{\delta}(t)|\sF_t],\,\textrm{ for each }t\in[0,T], \label{eq-relation}
   \end{align}
    we conclude from Lemma \ref{lem-BSPDE-condition-expect} that $u_{\delta}$, together with some $v_{\delta}\in\cL^2((L^2)^m)$, satisfies BSPDE
  \begin{equation*}
  \left\{\begin{array}{l}
  \begin{split}
  -du_{\delta}(t,x)=\,&\displaystyle \left( \Delta u_{\delta}(t,x)+\frac{{u}(t,x)-{u}_{\delta}(t,x)}{\delta}\right)\,dt      -v_{\delta}^{r}(t,x)\, dW_{t}^{r}, \quad
                     (t,x)\in Q;\\
    u_{\delta}(T,x)=\, &0, \quad x\in\cO.
    \end{split}
  \end{array}\right.
\end{equation*}
  Moreover, from relation \eqref{eq-relation} and Lemma \ref{lemm-approx-potential}, it follows that $\{{u}_{\delta}\}_{\delta>0}$ is bounded in $\mathscr{L}^2(\mathscr{K})$, and converges increasingly in $\cL^2(L^2)$ and weakly in $\cL^2(H^1_0)$ to ${u}$, as $\delta\rightarrow 0^+$. By It\^o formula, we obtain
  \begin{align}
    &E\bigg[\|u_{\delta}(t)-\tilde{u}(t)\|^2-\|\phi\|^2+\int_t^T\!\!\left(2\|\nabla (u_{\delta}(s)-\tilde{u}(s))\|^2+\|v_{\delta}(s)-\tilde{v}(s)\|^2 \right)\,ds\bigg]\nonumber\\
    =\,&
    E\bigg[\int_t^T\!\!\int_{\cO} 2(u_{\delta}-\tilde{u})\frac{{u}-{u}_{\delta}}{\delta}(s,x)
    \,dxds-    \int_t^T\!\!\!2\langle u_{\delta}-\tilde{u}(s),\,\tilde{f}(s)\rangle_{1,-1}\,ds   \bigg]\nonumber\\
    =\,&E\bigg[\int_t^T\!\!\int_{\cO} 2(u_{\delta}-u+u-\tilde{u})\frac{{u}-{u}_{\delta}}{\delta}(s,x)\,dxds
    -    \int_t^T\!\!\!2\langle u_{\delta}(s)-\tilde{u}(s),\,\tilde{f}(s)\rangle_{1,-1}\,ds   \bigg]\nonumber\\
    \leq\,&
    E\bigg[\int_t^T\!\!\|u_{\delta}(s)-\tilde{u}(s)\|_1^2\,ds+\int_t^T\!\!\|\tilde{f}(s)\|_{-1}^2\,ds
  \bigg],\nonumber
  \end{align}
  which implies $\{v_\delta\}_{\delta>0}$ is bounded in $\cL^2((L^2)^m)$ and admits a subsequence converging weakly. Without any loss of generality,  we assume  $\{v_{1/n}\}$ converges weakly to some $v\in\cL^2((L^2)^m)$. We choose a subsequence of convex combinations $(\check{u}_{1/n},\check{v}_{1/n})\in conv\{(u_{1/k},v_{1/k}):\,k\geq n,k\in \bN\}$ such that $(\check{u}_{1/n},\check{v}_{1/n})$ converges strongly to $(u,v)$ in $\cL^2(H^1_0)\times \cL^2((L^2)^m)$. In particular, $\{\check{u}_{1/n}\}$ is chosen to be an increasing sequence. Denote by $\{\check{g}_{1/n}\}$ the corresponding subsequence of convex combinations of $\{n(u-u_{1/n}):n\in\bN\}$.

  For each $n\in\bN$, let $u^{\mu}_{1/n}\in \mathscr{L}^2(\mathcal{P})$ solve PDE
  $$
  \partial_tu^{\mu}_{1/n}(t,x)=\Delta u^{\mu}_{1/n}(t,x)+\check{g}_{1/n},\ \ (t,x)\in Q;
  \quad u^{\mu}_{1/n}(0)=0;
  $$
  and $\bar{u}_{1/n}$ satisfies SPDE
  $$
  d\bar{u}_{1/n}(t,x)=\big(\Delta \bar{u}_{1/n}(t,x)-2\Delta \check{u}_{1/n}(t,x) \big)\,dt+\check{v}_{1/n}(t,x)\,dW_t,\ \ (t,x)\in Q;
  \quad \bar{u}_{1/n}(0)=\check{u}_{1/n}(0+).
  $$
  Then we have $\check{u}_{1/n}=\bar{u}_{1/n}-u^{\mu}_{1/n}$. As $\{\check{u}_{1/n}(0+)\}_{n\in\bN}$ converges increasingly to $u(0+)$ in $L^2$ and $\{(\check{u}_{1/n},\check{v}_{1/n})\}$ converges strongly in $\cL^2(H^1_0)\times \cL^2((L^2)^m)$. By SPDE theory (see \cite{DenisMatoussiStoica2005}), $\bar{u}_{1/n}$ converges strongly to some $\bar{u}$ in $\mathcal{H}$. Consequently, $\{u^{\mu}_{1/n}\}$ is bounded in $L^2(\Omega,\sF_T;\mathscr{K})$, and converges strongly to some $u^{\mu}$ in $\cL^2(H^1_0)$. In view of the closedness of $\cL^2(\mathcal{P})$ (see Remark \ref{rmk-closedness-potential}), we have $u^{\mu}\in \mathscr{L}^2(\mathcal{P})$. Hence, $u=\bar{u}-u^{\mu}\in \mathscr{U}_p$. We complete the proof.
\end{proof}
\begin{rmk}\label{rmk-back-potential}
   Conversely, given $u\in\mathscr{U}_p$ with
   \begin{equation*}
  \left\{\begin{array}{l}
  \begin{split}
  -du(t,x)=\,&\displaystyle \Delta u(t,x) \, dt+\mu(dt,x)
           -v^{r}(t,x)\, dW_{t}^{r},\,
                     (t,x)\in Q;\\
    u(T,x)=\, &0, \quad x\in\cO,
    \end{split}
  \end{array}\right.
\end{equation*}
    let $u^{\mu}\in\mathscr{L}^2(\mathcal{P}_0)$ be the stochastic regular potential associated with $\mu$, and let $\tilde{u}\in L^2(\Omega,\sF_T;\mathscr{W})$ satisfy almost surely PDE
    $$
    -\partial_t\tilde{u}=\Delta \tilde{u}+2\Delta u^{\mu};\quad \tilde{u}(T)=-u^{\mu}(T).
    $$
    Then $\hat{u}:=\tilde{u}+u^{\mu}$ belongs to $ L^2(\Omega,\sF_T;\mathcal{P}_T)$ and satisy
    $$
    -\partial_t \hat{u}-\Delta \hat{u}=\mu;\quad u(T)=0.
    $$
   By approximating the stochastic regular potential $u^{\mu}$, it is easy to verify that
   $$
   u(t)=E\left[\hat{u}(t)\big|\sF_t\right],\quad \forall\, t\in[0,T].
   $$
%
\end{rmk}

To study RBSPDE \eqref{RBSPDE-Laplac1}, we consider the following penalized BSPDE for each $n\in\bN$,
\begin{equation}\label{BSPDE-penalized}
  \left\{\begin{array}{l}
  \begin{split}
  -du_{n}(t,x)=\,&\displaystyle \big(\Delta u_{n}(t,x)+n(u_n(t,x)-\xi(t,x))^-\big)\,dt      -v_{n}^{r}(t,x)\, dW_{t}^{r}, \quad
                     (t,x)\in Q;\\
    u_{n}(T,x)=\, &0, \quad x\in\cO,
    \end{split}
  \end{array}\right.
\end{equation}
which admits a unique solution $(u_n,v_n)\in\mathcal{H}\times\cL^2((L^2)^m)$ with $u_n$ being quasi-continuous almost surely. Let $\psi\in L^2(\Omega,\sF_T;L^2(\cO))$ and $\check{f},\check{g}^i\in \cL^2(L^2)$, $i=1,2,\cdots,d$. Suppose that $\check{\xi}\in \mathscr{U}(I,0,\psi,\check{f},\check{g})$ with diffusion term $\check{v}\in\cL^2((L^2)^m)$,  dominates the obstacle process $\xi$ from above in assumption $(\mathcal{A}4)$.  It\^o formula yields
\begin{align}
  &\|(u_n-\check{\xi})(t)\|^2+\int_t^T\!\!\!\|(v_n-\check{v})(s)\|^2\,ds
  +2\int_t^T\!\!\!\|\nabla(u_n-\check{\xi})(s)\|^2\,ds\nonumber\\
  =&\,\|\psi\|^2+
  \int_t^T\!\!2\langle (u_n-\check{\xi})(s),\,n(u_n-\xi)^-(s)-\check{f}(s)   \rangle\,ds
  +\int_t^T\!\!2\langle \nabla (u_n-\check{\xi})(s),\, \check{g}(s)   \rangle\,ds
  \nonumber\\
  &\,-\int_t^T\!\!2\langle (u_n-\check{\xi})(s),\, (v_n-\check{v})(s)\,dW_s\rangle\nonumber\\
  \leq &\,
  \int_t^T\!\!\left(\|(u_n-\check{\xi})(s)\|^2+\frac{1}{2}\|\nabla(u_n-\check{\xi})(s)\|^2
  +C\left(\|\check{f}(s)\|^2+\|\check{g}(s)\|^2\right) \right) \,ds
  -\int_t^T\!\!2n\|(u_n-\xi)^-(s)\|^2\,ds
  \nonumber\\
  &\,-\int_t^T\!\!2\langle (u_n-\check{\xi})(s),\, (v_n-\check{v})(s)\,dW_s\rangle+\|\psi\|^2,\label{eq-penaliz-1}
\end{align}
which together with
\begin{align}
  E\bigg[\sup_{\tau\in[t,T]}\bigg| \int_{\tau}^T
  \!\!2\langle (u_n-\check{\xi})(s),\, (v_n-\check{v})(s)&\,dW_s\rangle    \bigg|
    \bigg]
    \leq\, 4E \bigg[\sup_{\tau\in[t,T]}\bigg| \int_t^{\tau}
  \!\!\langle (u_n-\check{\xi})(s),\, (v_n-\check{v})(s)\,dW_s\rangle    \bigg|
    \bigg] \nonumber\\
    \textrm{(by BDG inequality)}\leq&\, C E\bigg[\bigg(  \int_t^T \|(u_n-\check{\xi})(s)\|^2
    \| (v_n-\check{v})(s) \|^2  \,ds \bigg)^{1/2} \bigg]\nonumber\\
    \leq&\, E\bigg[\eps\sup_{s\in[t,T]}\|(u_n-\check{\xi})(s)\|^2
    +C_{\eps} \int_t^T \!\!\!
    \| (v_n-\check{v})(s) \|^2\,ds  \bigg],
    \label{eq-penaliz-2}
\end{align}
implies by Gronwall inequality
\begin{align}
  &\|(u_n-\check{\xi})\|_{\mathcal{H}}^2+E\bigg[\int_0^T\!\!\!\|(v_n-\check{v})(s)\|^2\,ds
  + \int_0^T\!\!\!n\|(u_n-\xi)^-(s)\|^2\,ds  \bigg]
  \nonumber\\
  \leq&\,C\,E\bigg[\|\psi\|^2+\int_0^T \!\!\!(\|\check{f}(s)\|^2+\|\check{g}(s)\|^2)\,ds  \bigg].\label{eq-penaliz-3}
\end{align}
Thus, there exists positive constant $C$ independent of $n$, such that
\begin{align}
  & \|u_n\|_{\mathcal{H}}^2+E\bigg[\int_0^T\!\!\!\|v_n(s)\|^2\,ds
  + \int_0^T\!\!\!n\|(u_n-\xi)^-(s)\|^2\,ds  \bigg]
  \leq\,C\|\check{\xi}\|_{\mathscr{U}}.\label{eq-penaliz-4}
\end{align}

By the comparison principles for BSPDEs, $\{u_n\}$ is an increasing sequence in $\cL^2(L^2)$. Consequently, we are allowed to choose a subsequence $\{(u_n,v_n)\}$ (denoted by itself) such that $\{u_n\}$ converges increasingly to some $u$ in $\cL^2(L^2)$ and $\{(u_n,v_n)\}$ converges weakly to $(u,v)$ in $\cL^2(H_0^1)\times \cL^2((L^2)^m)$. We further choose a subsequence of convex combinations $(\check{u}_{n},\check{v}_{n})\in conv\{(u_{k},v_{k}):\,k\geq n,k\in \bN\}$ such that $(\check{u}_{n},\check{v}_{n})$ converges strongly to $(u,v)$ in $\cL^2(H^1_0)\times \cL^2((L^2)^m)$. In particular, $\{\check{u}_{n}\}$ is chosen to be an increasing sequence. Denote by $\{\check{g}_{n}\}$ the corresponding subsequence of convex combinations of $\{n(u_{n}-\xi)^-:n\in\bN\}$.

In view of \eqref{eq-penaliz-4}, we know $u\geq \xi$, $ d\mathbb{P}\otimes dt\otimes dx$-a.e..     Through similar arguments to the proof for Lemma \ref{lem-back-potential}, we can check $u\in\mathscr{U}_p$. Therefore, 
\begin{align}
  u\geq \Gamma(0,\xi),\quad  d\mathbb{P}\otimes dt\otimes dx\textrm{-a.e..}\label{eq-relation-geq}
\end{align}
In fact, we further have

\begin{prop}\label{prop-U_p}
  Let $\xi$ satisfy assumption $(\mathcal{A}4)$ with $\xi(T)\leq 0$, $ d\mathbb{P}\otimes dx$-a.e.. For the random field $(u,v)$ obtained through the penalized procedure \eqref{BSPDE-penalized}-\eqref{eq-penaliz-4}, we assert that $u=\Gamma(0,\xi)$ and by  \eqref{eq-penaliz-4},
  \begin{align}
    \|u\|_{\mathscr{L}^2(\mathscr{K})}+\|v\|_{\cL^2((L^2)^m)} \leq\, C\,\|\check{\xi}\|_{\mathscr{U}}.\nonumber
  \end{align}
\end{prop}

\begin{proof}
  It is sufficient to prove $u=\Gamma(0,\xi)$. For each $\bar{u}\in\mathscr{U}_p$ satisfying $\bar{u} \geq \xi$, $ d\mathbb{P}\otimes dt\otimes dx$-a.e., let $(\bar{u}^k,\bar{v}^k)\in\mathcal{H}\times\cL^2((L^2)^m)$ satisfy the following BSPDE for each $k\in\bN$,
    \begin{equation*}
  \left\{\begin{array}{l}
  \begin{split}
  -d\bar{u}^k(t,x)=\,&\displaystyle \left[ \Delta \bar{u}^k(t,x)+k(\bar{u}(t,x)-\bar{u}^k(t,x))\right]\,dt      -\bar{v}^k(t,x)\, dW_{t}, \quad
                     (t,x)\in Q;\\
    \bar{u}^k(T,x)=\, &0, \quad x\in\cO.
    \end{split}
  \end{array}\right.
\end{equation*}
  By Remark \ref{rmk-back-potential}, there exists $\hat{u}\in L^2(\Omega,\sF_T;\mathcal{P})$ such that $\bar{u}(t)=E\left[\hat{u}(t)|\sF_t\right]$, $\forall\,t\in [0,T]$. In a similar way to the proof of Lemma \ref{lem-back-potential}, it follows that $\{\bar{u}^k\}_{k\in\bN}$ is bounded in $\mathscr{L}^2(\mathscr{K})$, and converges increasingly in $\cL^2(L^2)$ and weakly in $\cL^2(H^1_0)$ to $\bar{u}$, as $k\rightarrow +\infty$.

  By Corollary \ref{cor-ito}, we have
  \begin{align}
    &E\bigg[
    \|(\bar{u}^k-u_n)^-(t)\|^2+\int_t^T\!\!\!\Big(2\|\nabla (\bar{u}^k-u_n)(s)1_{\{\bar{u}^k(s)< u_n(s)\}}\|^2+\|(\bar{v}^k-v_n)(s)1_{\{\bar{u}^k(s)< u_n(s)\}}\|^2\Big)\,ds
    \bigg]
    \nonumber\\
    &=\,
    -2E\bigg[\int_t^T\!\!\! \langle (\bar{u}^k-u_n)^-(s),\, k(\bar{u}-\bar{u}^k)(s)-n(u_n-\xi)^-(s) \rangle   \,ds
    \bigg]  \nonumber\\
    &\leq\,
    2E\bigg[
    \int_t^T\!\!\!   \langle (\bar{u}^k-\bar{u})^-(s)+(\bar{u}-\xi)^-(s)+(\xi-u_n)^-(s),\, n(u_n-\xi)^-(s)-k(\bar{u}-\bar{u}^k)(s)    \rangle\,ds
    \bigg]
    \nonumber\\
    &\leq\,
    2E\bigg[
    \int_t^T\!\!\!   \langle (u^k-\bar{u})^-(s),\, n(u_n-\xi)^-(s)    \rangle\,ds
    \bigg]\longrightarrow 0,\quad \textrm{ as } k\rightarrow +\infty.\nonumber
  \end{align}
  Here, $\{u_n\}_{n\in\bN}$ and $\{v_n\}_{n\in\bN}$ are from the penalized procedure \eqref{BSPDE-penalized}-\eqref{eq-penaliz-4}.  Therefore, $\bar{u}\geq u^n$  for each $n\in\bN$ and by taking limits, there follows $\bar{u}\geq u$, $ d\mathbb{P}\otimes dt\otimes dx$-a.e..  In view of the definition of $\Gamma(0,\xi)$ and relation \eqref{eq-relation-geq}, we have $u=\Gamma(0,\xi)$. This completes the proof.
\end{proof}

\begin{cor}\label{cor-compr}
   Suppose that $\varphi$, $\psi$ and $\phi$ satisfy assumption $(A4)$, and $\psi\geq\phi\geq \varphi $, $ d\mathbb{P}\otimes dt\otimes dx$-a.e. with $\psi(T)\geq\phi(T)\geq 0 \geq \varphi(T)$, $ d\mathbb{P}\otimes dx$-a.e.. Then,
  \begin{align}
  \Gamma(0,\varphi)\leq \Gamma(\bar{\phi},\phi)
  \leq \Gamma(\bar{\psi},\psi)
  ,\,\,\, d\mathbb{P}\otimes dt \otimes dx-a.e.,\label{eq-est-compari}
  \end{align}
  where $\bar{\phi}\in \mathscr{U}(I,0,\phi(T),0,0)$ and $\bar{\psi}\in \mathscr{U}(I,0,\psi(T),0,0)$.
\end{cor}
\begin{proof}
  In view of Proposition \ref{prop-U_p}, we consider the following well-posed BSPDEs for each $n\in\bN$
\begin{equation*}
  \left\{\begin{array}{l}
  \begin{split}
  -d\bar{u}_{n}(t,x)=\,&\displaystyle \big(\Delta \bar{u}_{n}(t,x)+n(\bar{u}_n(t,x)-\phi(t,x)+\bar{\phi}(t,x))^-\big)\,dt      -\bar{v}_{n}^{r}(t,x)\, dW_{t}^{r}, \quad
                     (t,x)\in Q;\\
    \bar{u}_{n}(T,x)=\, &0, \quad x\in\cO,
    \end{split}
  \end{array}\right.
\end{equation*}
\begin{equation*}
  \left\{\begin{array}{l}
  \begin{split}
  -d\hat{u}_{n}(t,x)=\,&\displaystyle \big(\Delta \hat{u}_{n}(t,x)+n(\hat{u}_n(t,x)-\psi(t,x)+\bar{\psi}(t,x))^-\big)\,dt      -\hat{v}_{n}^{r}(t,x)\, dW_{t}^{r}, \quad
                     (t,x)\in Q;\\
    \hat{u}_{n}(T,x)=\, &0, \quad x\in\cO,
    \end{split}
  \end{array}\right.
\end{equation*}
and
\begin{equation*}
  \left\{\begin{array}{l}
  \begin{split}
  -du_{n}(t,x)=\,&\displaystyle \big(\Delta u_{n}(t,x)+n(u_n(t,x)-\varphi(t,x))^-\big)\,dt      -v_{n}^{r}(t,x)\, dW_{t}^{r}, \quad
                     (t,x)\in Q;\\
    u_{n}(T,x)=\, &0, \quad x\in\cO.
    \end{split}
  \end{array}\right.
\end{equation*}
Set $\tilde{u}_n=\bar{u}_n+\bar{\phi}$ and $\check{u}_n=\hat{u}_n+\bar{\psi}$. There exists
$(\tilde{v}_n,\check{v}_n)\in\cL^2((L^2)^m)\times \cL^2((L^2)^m)$, which together with $(\tilde{u}_n,\check{u}_n)$ satisfies BSPDEs
\begin{equation*}
  \left\{\begin{array}{l}
  \begin{split}
  -d\tilde{u}_{n}(t,x)=\,&\displaystyle \big(\Delta \tilde{u}_{n}(t,x)+n(\tilde{u}_n(t,x)-\phi(t,x))^-\big)\,dt      -\tilde{v}_{n}^{r}(t,x)\, dW_{t}^{r}, \quad
                     (t,x)\in Q;\\
    \tilde{u}_{n}(T,x)=\, &\phi(T), \quad x\in\cO,
    \end{split}
  \end{array}\right.
\end{equation*}
and
\begin{equation*}
  \left\{\begin{array}{l}
  \begin{split}
  -d\check{u}_{n}(t,x)=\,&\displaystyle \big(\Delta \check{u}_{n}(t,x)+n(\check{u}_n(t,x)-\psi(t,x))^-\big)\,dt      -\check{v}_{n}^{r}(t,x)\, dW_{t}^{r}, \quad
                     (t,x)\in Q;\\
    \check{u}_{n}(T,x)=\, &\psi(T), \quad x\in\cO.
    \end{split}
  \end{array}\right.
\end{equation*}
By the comparison principles of BSPDEs, we have $\check{u}_n\geq \tilde{u}_n\geq u_n$, $ d\mathbb{P}\otimes dt \otimes dx$-a.e., $n=1,2,\cdots$. In view of the penalized procedure \eqref{eq-penaliz-1}-\eqref{eq-penaliz-4} and Proposition \ref{prop-U_p}, we prove \eqref{eq-est-compari}.
\end{proof}



Before studying the resolution of RBSPDEs \eqref{RBSPDE-Laplac} and \eqref{RBSPDE-Laplac1}, we present an approximation result for the elements of $\mathscr{W}$.

\begin{lem}\label{lem-W-delta}
  For each $\phi\in \mathscr{W}$, there exist $\phi_0\in \mathcal{P}_T$ and $\{\phi_n\}\subset \mathscr{R}$,  such that $\phi_n(T)=\phi(T)$,
  $$
  \|\phi_n\|_{\mathscr{W}}+\|\phi_0\|_{\mathscr{K}}\leq C \|\phi\|_{\mathscr{W}}, \|\phi_n-\phi\|_{\mathscr{W}}\leq 2^{-n} \|\phi\|_{\mathscr{W}}\textrm{ and }
  |\phi-\phi_n|\leq n^{-1} \phi_0,\quad n\in\bN.
  $$
\end{lem}
\begin{proof}
  If $\|\phi\|_{\mathscr{W}}=0$, we set $\phi_n=0$. In the following, we assume $\|\phi\|_{\mathscr{W}}>0$.

  Since $L^2$ is dense in $H^{-1}$, there exists $\{f_n;n\in\bN\}\subset L^2(Q)$ such that
  $$\lim_{n\rightarrow +\infty}\|-\partial_t\phi-\Delta\phi-f_n\|_{-1}=0 \textrm{ and }
  \|\phi-\phi_n\|_{\mathscr{W}}\leq 2^{-n} \|\phi\|_{\mathscr{W}},\quad n\in\bN,$$
  with $\phi_n$ satisfying
  $$
  -\partial_t\phi_n-\Delta \phi_n=f_n;\quad \phi_n(T)=\phi(T).
  $$
  Then by Proposition \ref{prop-U_p},
  \begin{align}
    \|\Gamma(0,\phi-\phi_n)+\Gamma(0,\phi_n-\phi)\|_{\mathscr{K}}\leq C 
    \|\phi_n-\phi\|_{\mathscr{W}}
    \leq C 2^{-n}\|\phi\|_{\mathscr{W}},\quad n\in\bN.\nonumber
  \end{align}
  Therefore,
  $\sum_{n=1}^{+\infty}n \big\{ \Gamma(0,\phi-\phi_n)+\Gamma(0,\phi_n-\phi))   \big\}$ converges strongly in $\mathscr{K}$ to some $\phi_0\in\mathcal{P}_T$. Moreover,
  $$
  \|\phi_0\|_{\mathscr{K}}\leq \,C \|\phi\|_{\mathscr{W}}
  \textrm{ and } \left|\phi-\phi_n\right|\leq n^{-1}\phi_0,\quad dt\otimes dx-a.e..
  $$
  We complete the proof.
\end{proof}

Applying Lemma \ref{lem-W-delta} point-wisely and taking conditional expectations, we obtain the corollary.

\begin{cor}\label{cor-obst-regular-approx}
For each $\phi\in \mathscr{U}$, there exist $\phi_0\in \mathscr{U}_p$ and a sequence $\{\phi_n\}\subset \mathscr{L}^2(\mathscr{R})$  such that  for any $n\in\bN$,
  $$ \phi_n(T)=\phi(T),\,
  \|\phi_n\|_{\mathscr{U}}+\|\phi_0\|_{\mathscr{L}^2(\mathscr{K})}\leq C \|\phi\|_{\mathscr{U}},\, \|\phi_n-\phi\|_{\mathscr{U}}\leq 2^{-n} \|\phi\|_{\mathscr{U}},\,
  |\phi-\phi_n|\leq n^{-1} \phi_0,\,\, d\mathbb{P}\otimes dt \otimes dx-a.e..
  $$
  In particular, if $\phi\in \mathscr{U}(I,0,\psi,\mathcal{J}(f_+-f_-))$ with $\psi\in L^2(\Omega,\sF_T;L^2)$ and $f_+,f_-\in\cL^2((H^{-1})^+)$,  then, in view of the above proof of Lemma \ref{lem-W-delta} and the denseness of $(L^2)^+$ in $(H^{-1})^+$, we are allowed to choose $\phi_0\in\mathscr{U}_p$, $\{f_+^n,f_-^n;\,n\in\bN\}\subset \cL^2((L^2)+)$ and $\phi_n\in\mathscr{U}(I,0,\psi,f_+^n-f_-^n,0)$, such that
  $$\lim_{n\rightarrow \infty} \big( \|f_+^n-f_+\|_{\cL^2(H^{-1})} + \|f_-^n-f_-\|_{\cL^2(H^{-1})} \big)=0,$$
   and for each $n$,
  $$
  \|\phi_n\|_{\mathscr{U}}+\|\phi_0\|_{\mathscr{L}^2(\mathscr{K})}\leq C \|\phi\|_{\mathscr{U}},\, \|\phi_n-\phi\|_{\mathscr{U}}\leq 2^{-n} \|\phi\|_{\mathscr{U}},\,
  |\phi-\phi_n|\leq n^{-1} \phi_0,\,\, d\mathbb{P}\otimes dt \otimes dx-a.e..
  $$
\end{cor}

Now, we are ready for the resolution of RBSPDE \eqref{RBSPDE-Laplac1}
with $\xi$ satisfying assumption $(\mathcal{A}4)$. First, we investigate a special class of RBSPDEs, of whose solutions the stochastic regular measure admits a density w.r.t. Lebesgue measure in the distributional sense. The assertions herein will include the deterministic results (see \cite{CharrierTroianiello-1975}) as particular cases.

\begin{prop}\label{prop-regular-obst}
  Let $\xi$ of RBSPDE \eqref{RBSPDE-Laplac1} lie in $\mathscr{U}$. Assume further that $(\xi,\zeta)\in \mathscr{U}\times \cL^2((L^2)^m)$ be the solution of the following BSPDE
    \begin{equation*}
  \left\{\begin{array}{l}
  \begin{split}
  -d\xi(t,x)=\,&\displaystyle \left[ \Delta \xi(t,x)+f_+(t,x)-f_-(t,x)\right]\,dt      -\zeta(t,x)\, dW_{t}, \quad
                     (t,x)\in Q;\\
    \xi(T,x)=\, &\psi, \quad x\in\cO;
    \end{split}
  \end{array}\right.
\end{equation*}
where $\psi\in L^2(\Omega,\sF_T;L^2)$, $\psi\leq 0$, $ d\mathbb{P}\otimes dx$-a.e.,  $f_+,f_-\in\cL^2((H^{-1})^+)$. Then RBSPDE \eqref{RBSPDE-Laplac1} admits a unique solution $(u,v,\mu)$. And for this solution,  we have $u=\Gamma(0,\xi)\in\mathscr{U}\cap \mathscr{U}_p$, $\mu(dt,dx)=\beta(t,x)dtdx$ with $\beta\in\cL^2(H^{-1})$,
$0\leq\beta(t)\leq f_+(t)$ in $H^{-1}$, $ d\mathbb{P}\otimes dt$-a.e., and
\begin{align}
  &\|u\|_{\mathcal{H}}^2+ \|v\|^2_{\cL^2((L^2)^m)}
  \leq\,C\|\xi\|^2_{\mathscr{U}}.\label{est-prop-regular-obst}
\end{align}
\end{prop}
\begin{proof}
The uniqueness and the estimate \eqref{est-prop-regular-obst} follows from Lemma \ref{lem-unique} and Proposition \ref{prop-U_p} respectively. We shall prove the other assertions in two steps.
\medskip
\\
\textbf{Step 1.} We adopt the penalized method used in \eqref{BSPDE-penalized}-\eqref{eq-penaliz-4}. Extracting if necessary a subsequence, we obtain a sequence $\{u_n\}$ converges increasingly to some $u$ in $\cL^2(L^2)$,  $\{(u_n,v_n)\}$ converges weakly to $(u,v)$ in $\cL^2(H_0^1)\times \cL^2((L^2)^m)$, and
  \begin{align}
  &\|u_n\|_{\mathcal{H}}^2+E\bigg[\int_0^T\!\!\!\|v_n(s)\|^2\,ds
  + \int_0^T\!\!\!n\|(u_n-\xi)^-(s)\|^2\,ds  \bigg]
  \leq\,C\|\xi\|_{\mathscr{U}}.\label{eq-penaliz-regular-1}
\end{align}

First, assume further that $f_+,f_-\in \cL^2((L^2)^+)$. By Corollary \ref{cor-ito}, we have
  \begin{align}
    &E\bigg[
    \|(u_n-\xi)^-(t)\|^2+\int_t^T\!\!\!\big(2\|\nabla (u_n-\xi)(s)1_{\{\xi(s)> u_n(s)\}}\|^2+\|(\zeta-v_n)(s)1_{\{\xi(s)> u_n(s)\}}\|^2\big)\,ds
    \bigg]
    \nonumber\\
    &=\,
    -2E\bigg[\int_t^T\!\!\! n\|(u_n-\xi)^-(s)\|^2   \,ds
    +\int_t^T\!\!\! \langle (u_n-\xi)^-(s),\, f_+(s)-f_-(s) \rangle   \,ds
    \bigg]  \nonumber\\
    &\leq\,
    -E\bigg[\int_t^T\!\!\! n\|(u_n-\xi)^-(s)\|^2   \,ds
    +\frac{C}{n}\int_t^T\!\!\! \| f_+(s)-f_-(s)\|^2   \,ds
    \bigg],\nonumber
  \end{align}
  which implies
  \begin{align}
    n^2E\bigg[\int_0^T\!\!\! \|(u_n-\xi)^-(s)\|^2   \,ds\bigg]\leq C E\bigg[\int_0^T\!\! \|f_+(s)-f_-(s)\|^2\,ds\bigg].\label{eq-penalize-regular-2}
  \end{align}
  Set
  $$\beta_n=n(u_n-\xi)^-\textrm{ and }
  (u_{nk},v_{nk},\beta_{nk})=(u_{n}-u_{k},v_{n}-v_{k},\beta_{n}-\beta_{k}),\quad n,k\in\bN.$$
  By Ito formula and \eqref{eq-penalize-regular-2}, we have
  \begin{align}
    &\|u_{nk}(t)\|^2+\int_t^T\big(2\|\nabla u_{nk}(s)\|^2+\|v_{nk}(s)\|^2    \big)\,ds\nonumber\\
    =&
    \int_t^T\!\!\!  2\langle u_n(s)-\xi(s)+\xi(s)-u_k(s),\, n(u_n(s)-\xi(s))^--k(u_k(s)-\xi(s))^-  \rangle\,ds
    -\!\int_t^T\!\!\!2\langle u_{nk}(s),\,v_{nk}(s)\,dW_s  \rangle\nonumber\\
    \leq&
    \int_t^T\!\!\!2(k+n)\langle(u_n(s)-\xi(s))^-,\,(u_k(s)-\xi(s))^-\rangle\,ds
    -\!\int_t^T\!\!\!2\langle u_{nk}(s),\,v_{nk}(s)\,dW_s  \rangle\nonumber\\
    \leq &\,
    \big(k^{-1}+n^{-1})C \!\int_0^T \!\!\!\|f_+(s)-f_-(s)\|^2\,ds
    -\!\int_t^T\!\!\!2\langle u_{nk}(s),\,v_{nk}(s)\,dW_s  \rangle,
  \end{align}
  which together with
  \begin{align}
  E\Big[\sup_{\tau\in[t,T]}\Big| \int_{\tau}^T
  \!\!2\langle u_{nk}(s),\,v_{nk}(s)\,dW_s\rangle    \Big|
    \Big]
    \leq&\, 4E \Big[\sup_{\tau\in[t,T]}\Big| \int_t^{\tau}
  \!\!\langle  u_{nk}(s),\,v_{nk}(s)\,dW_s\rangle    \Big|
    \Big] \nonumber\\
    \textrm{(by BDG inequality)}\leq&\, C E\Big[\Big(  \int_t^T \| u_{nk}(s)(s)\|^2
    \| v_{nk}(s) \|^2  \,ds \Big)^{1/2} \Big]\nonumber\\
    \leq&\, E\Big[\eps\sup_{s\in[t,T]}\|u_{nk}(s)\|^2
    +C(\eps) \int_t^T \!\!\!
    \| v_{nk}(s) \|^2\,ds  \Big],\,\forall\,\eps>0\nonumber
\end{align}
  and Gronwall inequality, implies
  \begin{align}
    \|u_{nk}\|_{\mathcal{H}}+\|v_{nk}\|_{\mathcal{L}^2((L^2)^m)}
    \leq\, C \big(k^{-1}+n^{-1}) \!\int_0^T \!\!\!\|f_+(s)-f_-(s)\|^2\,ds\longrightarrow 0,\quad
    \textrm{ as }n,k\rightarrow +\infty.
  \end{align}
  Denote the limit by $(u,v)$. By \eqref{eq-penalize-regular-2}, extracting a subsequence if necessary, we may assume without any loss of generality that $\{\beta_n;\,n\in\bN\}$ converges weakly to some $\beta$ in $\cL^2(L^2)$.
%
%
  Taking limits, we have $\beta\geq 0$, $ d\mathbb{P}\otimes dt \otimes dx$-a.e. and  for any $\varphi\in \mathcal{D}$,
  \begin{equation*}
  \begin{split}
    &\langle u(t),\,\varphi(t) \rangle +\int_t^T\!\!\big[\langle u(s),\,\partial_s \varphi(s)   \rangle + \langle \nabla u(s),\,\nabla \varphi(s) \rangle \big]\,ds
    =\!
    \int_t^T\!\!\langle \beta(s),\,\varphi(s)  \rangle \,ds +\int_t^T\!\!\langle \varphi(s),\,v^r(s) \,dW_s^r\rangle .
  \end{split}
\end{equation*}
Thus, $u\in\mathscr{U}\cap\mathscr{U}_p$ and by assertion (i) of Proposition \ref{prop-quasi-conti-Cont_TL2}, $u$ is almost surely quasi-continuous. Proposition \ref{prop-U_p} yields that $u=\Gamma(0,\xi)$. On the other hand, combining the strong convergence of $\{u_n\}$ and the weak convergence of $\{\beta_n\}$, we have
\begin{align}
  E\bigg[
  \int_Q \big( u(s,x)-\xi(s,x)\big)\beta(s,x)\,dsdx  \bigg]
  =&\lim_{n\rightarrow \infty}
  E\bigg[
  \int_Q \big( u_n(s,x)-\xi(s,x)\big)n\big( u_n(s,x)-\xi(s,x)\big)^-\,dsdx  \bigg]\leq\, 0,
  \nonumber
\end{align}
which together with $u=\Gamma(0,\xi)$ and $\beta\geq 0$, implies
$$
\int_Q \big( u(s,x)-\xi(s,x)\big)\beta(s,x)\,dsdx=0,\quad a.s..
$$
%
%
  Put $\mu(dt,dx)=\beta(t,x)dtdx$. Then, $(u,v,\mu)$ is a solution of RBSPDE \eqref{RBSPDE-Laplac1}.

  Set $\hat{\xi}\in\mathscr{U}(I,0,0,f_+,0)$ and $\bar{\xi}\in \mathscr{U}(I,0,-\psi,f_-,0)$. Then $\xi=\hat{\xi}-\bar{\xi}$, $\hat{\xi}\in\mathscr{U}_p$ and $\Gamma(0,\hat{\xi})=\hat{\xi}$. Let $(\hat{\xi}_n,\hat{v}_n)\in\mathcal{H}\times \cL^2((L^2)^m)$ be the solution of the following penalized BSPDE
  \begin{equation*}
  \left\{\begin{array}{l}
  \begin{split}
  -d\hat{\xi}_{n}(t,x)=\,&\displaystyle \big(\Delta \hat{\xi}_{n}(t,x)+n(\hat{\xi}_n(t,x)-\hat{\xi}(t,x))^-\big)\,dt      -\hat{v}_{n}^{r}(t,x)\, dW_{t}^{r}, \quad
                     (t,x)\in Q;\\
    \hat{\xi}_{n}(T,x)=\, &0, \quad x\in\cO,
    \end{split}
  \end{array}\right.
\end{equation*}
and let $(\bar{u}_n,\bar{v}_n)\in\mathcal{H}\times \cL^2((L^2)^m)$ solve BSPDE
\begin{equation*}
  \left\{\begin{array}{l}
  \begin{split}
  -d\bar{u}_{n}(t,x)=\,&\displaystyle \big(\Delta \bar{u}_{n}(t,x)+n(\bar{u}_n(t,x)-\hat{\xi}(t,x))^-\big)\,dt      -\bar{v}_{n}^{r}(t,x)\, dW_{t}^{r}, \quad
                     (t,x)\in Q;\\
    \bar{u}_{n}(T,x)=\, &-\psi , \quad x\in\cO.
    \end{split}
  \end{array}\right.
\end{equation*}
In view of BSPDE \eqref{BSPDE-penalized}, it is easy to check that $\bar{u}_n=u_n+\bar{\xi}$. From the comparison principles for BSPDEs, we deduce that $\bar{u}_n\geq \hat{\xi}_n$ and thus,
\begin{align}
  n(u_n-\xi)^-=n(\bar{u}_n-\hat{\xi})^-\leq n(\hat{\xi}_n-\hat{\xi})^-, \quad  d\mathbb{P}\otimes dt \otimes dx\textrm{-a.e.}.\label{eq-beta-f+}
\end{align}
 In a similar way to \eqref{BSPDE-penalized}-\eqref{eq-penaliz-4}, we are allowed to choose a subsequence of convex combinations $(\hat{\beta}_k,\hat{f}_k)\in \textrm{conv}\{\big(n(u_n-\xi)^-,\,  n(\hat{\xi}_n-\hat{\xi})^-\big);\,n\geq k\}$, such that $(\hat{\beta}_k,\hat{f}_k)$ converges strongly to $(\beta,f_+)$ in $\cL^2(L^2)\times \cL^2(L^2)$. By \eqref{eq-beta-f+}, there follows $\beta\leq f_{+}$, $ d\mathbb{P}\otimes dt \otimes dx$-a.e..
\medskip
\\
\textbf{Step 2.} Now, we consider the general $f_+$ and $f_-$.

By Corollary \ref{cor-obst-regular-approx}, we choose $\{f^n_+,f^n_-:\,n\in\bN\}\subset\cL^2((L^2)^+)$ and $\phi_0\in\mathscr{U}_p$, such that $(f^n_+,f^n_-)$ converges to $(f_+,f_-)$ in $\cL^2(H^{-1})\times \cL^2(H^{-1})$ and $\xi_n\in\mathscr{U}(I,0,\psi,f^n_+,f^n_-,0)$ satisfies
\begin{align}
\|\xi^n\|_{\mathscr{U}} + \|\phi_0\|_{\mathscr{L}^2(\mathscr{K})}  \leq C \|\xi\|_{\mathscr{U}},\,  \|\xi^n-\xi\|_{\mathscr{U}}\leq 2^{-n}\|\xi\|_{\mathscr{U}}
\textrm{ and }|\xi^n-\xi|\leq n^{-1} \phi_0,\,\, d\mathbb{P}\otimes dt \otimes dx.
\label{eq-using-coro}
\end{align}
Denote by $(u^n,\,v^n,\,\beta^n(t,x)\,dtdx)$ the solution of RBSPDE \eqref{RBSPDE-Laplac1} with the associated obstacle process being replaced by $\xi^n$. Then, $u^n=\Gamma(0,\xi^n)$, $\beta^n\in \cL^2(L^2)$, $0 \leq \beta^n \leq f_+^n$, $ d\mathbb{P}\otimes dt \otimes dx$-a.e., and
\begin{align}
  \|u^n\|_{\mathcal{H}}+\|v^n\|_{\cL^2((L^2)^m)}\leq C \|\xi^n\|_{\mathscr{U}}\leq C \|\xi\|_{\mathscr{U}}<+\infty. \label{eq-est-regular1}
\end{align}
In particular, for any $\phi\in  \mathscr{C}(Q)$, we have almost surely
$$
  \bigg|\int_{Q}\!\!\!\phi(t,x)\beta^n(t,x)\,dtdx\bigg|
\leq\,\int_{Q}(\phi^++\phi^-)(t,x)f^n_+(t,x)\,dtdx
\leq\,C\|f_+^n\|_{L^2(H^{-1})}\|\phi\|_{L^2(H_0^1)},
$$
where $(\phi^+,\phi^-):=(\max\{\phi,0\},\max\{-\phi,0\})$ and we use the fact that
there exists a universal constant $C$ such that
$$
\||\phi|\|_{L^2(0,T;H_0^1)}=\|\phi^+ + \phi^-\|_{L^2(0,T;H_0^1)} \leq C \|\phi\|_{L^2(0,T;H_0^1)}, \quad \forall\,\phi\in L^2(0,T;H_0^1).
$$
Since $\mathscr{C}(Q)$ is dense in $L^2(0,T;H_0^1)$, there exists constant $C_0$ independent of $n$ such that
$\|\beta^n\|_{\cL^2(H^{-1})}\leq C_0 < +\infty.$

On the other hand, setting $u=\,\Gamma(0,\xi)$, we have
\begin{align}
  |u^n-u|= |\Gamma(0,\xi^n)-\Gamma(0,\xi)|
  &\leq\, |   \Gamma(0,\xi^n-\xi)  +  \Gamma(0,\xi-\xi^n)   |\nonumber\\
  &\leq\, 2n^{-1}\phi_0\longrightarrow 0, \,\, d\mathbb{P}\otimes dt \otimes dx-a.e.,\quad \textrm{as }n\rightarrow +\infty.\nonumber
\end{align}
Let $u^{\beta^n}\in \mathscr{L}^2(\mathscr{W})\cap \mathscr{L}^2(\mathcal{P_0})$  be the regular stochastic potentials associated with $\beta^n(t,x)dtdx$. By the boundedness $\beta^n$ in $\cL^2(H^{-1})$, we are allowed to choose a subsequence ($\{u^{\beta^n};n\in\bN\}$ (denoted by itself, without any loss of generality) which is bounded in $\cL^2(\mathscr{K})$ and converges weakly in $\cL^2(H_0^1)$. And further, by the boundedness of $(u^n,v^n,\beta^n)$ in $\mathcal{H}\times\cL^2((L^2)^m)\times\cL^2(H^{-1})$, we can choose a sequence of convex combinations $(\check{u}^n,\check{v}^n,\check{\beta}^n)\in \textrm{conv}\{(u^k,v^k,\beta^k):\,k\geq n\}$ such that $(\check{u}^n,\check{v}^n,\check{\beta}^n)$ converges strongly in $\cL^2(H_0^1)\times\cL^2((L^2)^m)\times\cL^2(H^{-1})$ to some $(\bar{u},v,\beta)$ with $\bar{u}=u$.  Taking limits, we obtain that $0\leq \beta\leq f_+$ in $H^{-1}$, $ d\mathbb{P}\otimes dt$-a.e. and  for any $\varphi\in \mathcal{D}$,
  \begin{equation*}
  \begin{split}
    &\langle u(t),\,\varphi(t) \rangle +\int_t^T\!\!\big[\langle u(s),\,\partial_s \varphi(s)   \rangle + \langle \nabla u(s),\,\nabla \varphi(s) \rangle \big]\,ds
    =\!
    \int_t^T\!\!\langle \beta(s),\,\varphi(s)  \rangle \,ds +\int_t^T\!\!\langle \varphi(s),\,v^r(s) \,dW_s^r\rangle .
  \end{split}
\end{equation*}
By  \cite[Theorem 1.2 of Chapter 1]{Qiu-2012} or \cite[Theorem 4.2]{RenRocknerWang2007}, $u\in\mathscr{U}\cap\mathscr{U}_p$ and by assertion (i) of Proposition \ref{prop-quasi-conti-Cont_TL2}, $u$ is almost surely quasi-continuous. Moreover,
by (ii) of Proposition \ref{prop-quasi-conti-Cont_TL2},
$$
\int_{0}^T\langle (u-\xi)(t), \,\beta(t)\rangle_{1,-1}\,dt
=\lim_{n\rightarrow +\infty} \int_{0}^T\langle (u^n-\xi^n)(t), \,\beta^n(t)\rangle_{1,-1}\,dt=0,\quad a.s..
$$
 Consequently, $(u,v,\mu)$ with $\mu(dt,dx)=\beta(t,x) dtdx$ is the unique solution of RBSPDE \eqref{RBSPDE-Laplac1}. We complete the proof.
\end{proof}

Now, we are in a position to present the main results of this subsection for RBSPDEs with Laplacian leading coefficients.
\begin{thm}\label{thm-Lapl}
  Let the obstacle process $\xi$ of RBSPDE \eqref{RBSPDE-Laplac1} satisfy assumption $(\mathcal{A}4)$ with $\xi(T)\leq 0$, $ d\mathbb{P}\otimes dx$-a.e.. Then RBSPDE \eqref{RBSPDE-Laplac1} admits a unique solution $(u,v,\mu)$ with $u=\Gamma(0,\xi)$ and
  \begin{align}
      \|u\|_{\mathcal{H}}+\|v\|_{\cL^2((L^2)^m)}
  \leq\,C\|\check{\xi}\|_{\mathscr{U}}.\label{est-thm-lapl}
  \end{align}
\end{thm}
\begin{proof}
  The uniqueness and the estimate \eqref{est-thm-lapl} follow from Lemma \ref{lem-unique} and Proposition \ref{prop-U_p} respectively. We shall prove the other assertions  in two steps.\medskip

  \textbf{Step 1.} We first assume $\xi\in\mathscr{U}$. In view of Remark \ref{rmk-dense-H-1}, we see that $\xi$ does not necessarily satisfy the hypothesis of Proposition \ref{prop-regular-obst}.

    By Corollary \ref{cor-obst-regular-approx}, there exist $\xi_0\in\mathscr{U}_p$ and a sequence $\{\xi_n\}\subset \mathscr{L}^2(\mathscr{R})$ such that $\xi_n(T)=\xi(T)$,
  $$
  \|\xi_n\|_{\mathscr{U}}+\|\xi_0\|_{\mathscr{L}^2(\mathscr{K})}\leq C \|\xi\|_{\mathscr{U}},
  \,\, \|\xi_n-\xi\|_{\mathscr{U}}\leq 2^{-n}\|\xi\|_{\mathscr{U}} \textrm{ and }
  |\xi_n-\xi|\leq n^{-1}\xi_0,\, d\mathbb{P}\otimes dt \otimes dx-a.e..
  $$
  For each $n$, by Proposition \ref{prop-regular-obst}, RBSPDE \eqref{RBSPDE-Laplac1} associated with obstacle process $\xi_n$ admits a unique solution $(u_n,v_n,\mu_n)$ with $u_n=\Gamma(0,\xi_n)\in\mathscr{U}_p$ and
  \begin{align}
  \|u_n\|_{\mathcal{H}}+\|v_n\|_{\cL^2((L^2)^m)}\leq C \|\xi_n\|_{\mathscr{U}}
  \leq C \|\xi\|_{\mathscr{U}}< +\infty.\label{eq-thm1-est}
  \end{align}

  Put $u=\Gamma(0,\xi)$. Then $u\in\mathscr{U}_p$ and
  \begin{align}
    |u-u_n|=|\Gamma(0,\xi)-\Gamma(0,\xi_n)|\leq&\, \Gamma(0,\xi_n-\xi)+ \Gamma(0,\xi-\xi_n)\nonumber\\
    \leq& \,2\,\Gamma(0,n^{-1}\xi_0)
    =2n^{-1}\xi_0\longrightarrow 0,\textrm{ as }n\rightarrow \infty. \label{eq-thm-lap-apprx}
  \end{align}
   By Lemma \ref{lem-K-W}, there exists $\bar{\xi}_0\in \mathscr{U}$ such that $\xi_0\leq \bar{\xi}_0$ and by (i) of Proposition \ref{prop-quasi-conti-Cont_TL2}, $\bar{\xi}_0$ is almost surely quasi-continuous. Therefore, from the quasi-continuity of $\{u_n\}$ we conclude that $u$ is almost surely quasi-continuous and in particular, we have
   \begin{align}
     \|u_n-u\|_{\cS^2(L^2)}=0.\label{eq-conveg-in-S}
   \end{align}

  For each $n$, let $u_n^{\mu}\in\mathscr{L}^2(\mathcal{P})$ solve PDE
  $$
  \partial_tu^{\mu}_n=\Delta u^{\mu}_{n} +\mu_n;\quad u_n^{\mu}(0)=0;
  $$
  and let $\tilde{u}_n$ satisfy SPDE
  $$
  d\tilde{u}_n(t,x)=\big(\Delta \tilde{u}_n(t,x)-2\Delta u_n(t,x) \big)\,dt
  +v_n(t,x)\,dW_s,\,\,(t,x)\in Q;\quad \tilde{u}_n(0)=u_n(0).
  $$
  Then, $u_n=\tilde{u}_n-u^{\mu}_n$. By \eqref{eq-thm1-est}, \eqref{eq-thm-lap-apprx}, \eqref{eq-conveg-in-S} and Remark \ref{rmk-math-U}, we deduce that both $\{u^{\mu}_n\}$ and $\{\tilde{u}_n\}$ are bounded in $\mathcal{H}$ and we are allowed to choose a subsequence (denoted by itself) $\{(u^{\mu}_n,\,\tilde{u}_n)\}$, which converges weakly to some  $(u^{\mu},\tilde{u})\in\cL^2(H_0^1)\times\cL^2(H_0^1)$, with $u^{\mu}$ being a stochastic potential associated with some stochastic measure $\mu$. On the other hand, in view of Remark \ref{rmk-math-U}, combining relations \eqref{eq-thm1-est}, \eqref{eq-thm-lap-apprx} and \eqref{eq-conveg-in-S}, we are allowed to choose a sequence of convex combinations $(\check{u}_{n},\check{\tilde{u}}_n,\check{v}_{n})\in conv\{(u_{k},\tilde{u}_k,v_{k}):\,k\geq n,k\in \bN\}$ such that $(\check{u}_{n},\check{\tilde{u}}_n,\check{v}_{n})$ converges strongly to $(u,\tilde{u},v)$ in $\mathcal{H}\times\mathscr{U}\times \cL^2((L^2)^m)$. Thus, $u^{\mu}\in\mathcal{H}$ and $\mu$ is a stochastic regular measure. By (ii) of Proposition \ref{prop-quasi-conti-Cont_TL2}, we have
  \begin{align}
   \int_{Q}\big( u-\xi \big)(t,x)\,\mu(dt,dx)= \lim_{n\rightarrow \infty} \int_{Q}\big( u_n-\xi_n \big)(t,x)\,\mu_n(dt,dx)=0.
   \label{eq-thm1-skorohd}
  \end{align}
%
Note that the corresponding stochastic regular measure sequence of convex combinations converges vaguely to $\mu$. Consequently, taking limits, we conclude that  for any $\varphi\in \mathcal{D}$,
  \begin{equation*}
  \begin{split}
    &\langle u(t),\,\varphi(t) \rangle +\int_t^T\!\!\big[\langle u(s),\,\partial_s \varphi(s)   \rangle + \langle \nabla u(s),\,\nabla \varphi(s) \rangle \big]\,ds
    =\!
    \int_Q\!\!\varphi(s,x)\,\mu(ds,dx) +\int_t^T\!\!\langle \varphi(s),\,v^r(s) \,dW_s^r\rangle .
  \end{split}
\end{equation*}
  Hence, $(u,v,\mu)$ is a solution of RBSPDE \eqref{RBSPDE-Laplac1} with $u=\Gamma(0,\xi)$.
\medskip

  \textbf{Step 2.}
  Consider the general $\xi$ satisfying assumption $(\mathcal{A}4)$. In view of Remark \ref{rmk-assumpt4}, we assume further that $\xi$ satisfy assumption $(\mathcal{A}4b)$.   By Corollary \ref{cor-prop-obstac-SPDE} and assertion (iii) of Proposition \ref{prop-quasi-conti-Cont_TL2}, there exist $\{\phi_n;n\in\bN\}\subset \mathscr{U}$ and $\{\varphi_n;n\in\bN\}\subset \mathscr{U}$, such that $\{\phi_n\}$ converges decreasingly to $0$, $ d\mathbb{P}\otimes dt \otimes dx$-a.e.,
  $$
  \lim_{n\rightarrow \infty}\|\phi_n\|_{\mathscr{U}}=0 \textrm{ and }
  |\xi-\varphi_n|\leq \phi_n,\quad  d\mathbb{P}\otimes dt \otimes dx-a.e.,\,\,n=1,2,\cdots.
  $$

  For each $n\in\bN$, set
  $$
  \bar{\phi}_n\in \mathscr{U}(I,0,\phi_n(T),0,0),\,\,\bar{\varphi}_n\in\mathscr{U}(I,0,\varphi_n(T)-\xi(T),0,0)
  \textrm{ and }\hat{\varphi}_n\in\mathscr{U}(I,0,|\varphi_n(T)-\xi(T)|,0,0),\quad n=1,2,\cdots.
  $$
  As $|\varphi_n(T)-\xi(T)|\leq \phi_n(T)$, $ d\mathbb{P} \otimes dx$-a.e.,  it follows that  $\bar{\varphi}_n\leq \hat{\varphi}_n\leq \bar{\phi}_n$, with $\{\bar{\phi}_n\}$ converging decreasingly to $0$, $ d\mathbb{P}\otimes dt \otimes dx$-a.e.. Moreover,
  $$
  \lim_{n\rightarrow \infty}
  \Big\{
  \|\hat{\phi}_n\|_{\mathscr{U}}+\|\bar{\varphi}_n\|_{\mathscr{U}}+\|\bar{\phi}_n\|_{\mathscr{U}}
  +\|\hat{\varphi}_n\|_{\mathscr{U}}
  \Big\}=0.
  $$

  For each $n$, by \textbf{Step 1}, RBSPDE \eqref{RBSPDE-Laplac1} associated with obstacle process $\varphi_n-\bar{\varphi}_n$ admits a unique solution $(u_n,v_n,\mu_n)$ with $u_n=\Gamma(0,\varphi_n-\bar{\varphi}_n)\in\mathscr{U}_p$ and we have
  \begin{align}
  \|u_n\|_{\mathcal{H}}+\|v_n\|_{\cL^2((L^2)^m)}\leq C \|\bar{\varphi}_n+\phi_1+\check{\xi}\|_{\mathscr{U}}
  \leq C_0< \infty,
  \end{align}
with $C_0$ being a constant independent of $n$.

  Put $u=\Gamma(0,\xi)$. Then $u\in\mathscr{U}_p$ and by Corollary \ref{cor-compr},
  \begin{align}
    |u-u_n|=|\Gamma(0,\xi)-\Gamma(0,\varphi_n-\bar{\varphi}_n)|&\leq\, \Gamma(0,\varphi_n-\bar{\varphi}_n-\xi)+ \Gamma(0,\xi-\varphi_n+\bar{\varphi}_n)\nonumber\\
    &\leq \,2\,\Gamma(0,\phi_n-\bar{\phi}_n)+2\bar{\phi}_n+2\hat{\varphi}_n\nonumber\\
    &\leq \,2\,\Gamma(0,\phi_n-\bar{\phi}_n)+4\bar{\phi}_n
    \nonumber\\
    &=\,2\,\Gamma(\bar{\phi}_n,\phi_n)+2\bar{\phi}_n
    \longrightarrow 0,\,\textrm{ in }\mathscr{U},\,\textrm{ as }n\rightarrow \infty.
  \end{align}
   As both $\{\phi_n\}$ and $\{\bar{\phi}_n\}$ converge decreasingly to $0$, $ d\mathbb{P}\otimes dt \otimes dx$-a.e., in view of the equivalence relationship between RBSPDEs \eqref{RBSPDE-Laplac} and \eqref{RBSPDE-Laplac1}, we conclude from Corollary \ref{cor-compr} and Theorem \ref{thm-comprsn} that $2\,\Gamma(\bar{\phi}_n,\phi_n)+2\bar{\phi}_n$ converges decreasingly to $0$, $ d\mathbb{P}\otimes dt \otimes dx$-a.e.. Thanks to the  quasi-continuity of $\phi_n$, $\bar{\phi}_n$ (by (i) of Proposition \ref{prop-quasi-conti-Cont_TL2}) and $\Gamma(0,\phi_n-\bar{\phi}_n)$ (by \textbf{Step 1}),  $u$ is almost surely quasi-continuous.

   In a similar way to \textbf{Step 1}, by choosing subsequences and subsequences of convex combinations and taking limits, we find a solution $(u,v,\mu)$ for RBSPDE \eqref{RBSPDE-Laplac1}.
    The proof is complete.
\end{proof}

 In view of the above proof and the equivalence between RBSPDEs \eqref{RBSPDE-Laplac} and \eqref{RBSPDE-Laplac1}, we conclude the following corollary from Theorem \ref{thm-Lapl}.
\begin{cor}\label{cor-lap}
  Let $G\in L^2(\Omega,\sF_T;L^2)$ and $\bar{f},\,\bar{g}^i\in \cL^2(L^2)$, $i=1,2,\cdots,d$, and $\tilde{\xi}$ satisfy assumption $(\mathcal{A}4)$ with $\tilde{\xi}(T)\leq G$, $ d\mathbb{P}\otimes dt \otimes dx$-a.e.. There exists a unique solution $(u,v,\mu)$ to RBSPDE \eqref{RBSPDE-Laplac}
and there holds
\begin{align*}
      \|u\|_{\mathcal{H}}+\|v\|_{\cL^2((L^2)^m)}
  \leq\,C\Big(\|\check{\xi}\|_{\mathscr{U}}+\|G\|_{L^2(\Omega,\sF_T;L^2)}
  +\|f\|_{\cL^2(L^2)}+\|g\|_{\cL^2((L^2)^d)}  \Big),
  \end{align*}
  where $\check{\xi}\in\mathscr{U}$ is the random field dominating $\xi$ from above in assumption $(\mathcal{A}4)$.   Moreover, letting $\tilde{u}\in \mathscr{U}(I,0,G,\bar{f},\bar{g})$, we have
  \begin{align}
  u=\Gamma(\tilde{u},\tilde{\xi}).\label{var-prob}
  \end{align}
\end{cor}
\begin{rmk}
  In view of \eqref{var-prob} above, we observe that the solution of RBSPDE \eqref{RBSPDE-Laplac} corresponds to a minimal point of variational problem \eqref{variation}.
\end{rmk}
\subsection{General case}
\begin{lem}\label{lem-thm-main}
  Let assumptions $(\mathcal{A}1)-(\mathcal{A}4)$ hold with $\xi(T)\leq G$, $ d\mathbb{P} \otimes dx$-a.e.. For $\theta\in [0,1]$ and $(f_l,g_l)\in\cL^2(L^2)\times\cL^2((L^2)^d)$, $l=1,2$, suppose that $(u_l,v_l,\mu_l)$ is the solution of RBSPDE
  \begin{equation}\label{RBSPDE-i}
  \left\{\begin{array}{l}
  \begin{split}
  -du_l(t,x)=\,&\displaystyle \Bigl[(1-\theta)\Delta u_l(t,x)+\theta\partial_{x_j}\bigl(a^{ij}(t,x)\partial_{x_i} u_l(t,x)
        +\sigma^{jr}(t,x) v_l^{r}(t,x)     \bigr) \\
        &\displaystyle +\theta(f+ \nabla \cdot g)(t,x,u_l(t,x),\nabla u_l(t,x),v_l(t,x))
                \Bigr]\, dt+(f_l+\nabla\cdot g_l)(t,x)\,dt\\ &\displaystyle
           +\mu_l(dt,x)-v_l^{r}(t,x)\, dW_{t}^{r}, \quad
                     (t,x)\in Q;\\
    u_l(T,x)=\, &G(x), \quad x\in\cO;\\
    u_l(t,x)\geq\,& \xi(t,x),\,\,d\mathbb{P}\otimes dt\otimes dx-a.e.;\\
    \int_Q \big( u_l(t,x)&-\xi(t,x) \big)\,\mu_l(dt,dx)=0,\,a.s..
    \end{split}
  \end{array}\right.
\end{equation}
Then
\begin{align}
  \|u_1-u_2\|_{\mathcal{H}}+\|v_1-v_2\|_{\cL^2((L^2)^m)}
  \leq C\Big( \|f_1-f_2\|_{\cL^2(L^2)}+\|g_1-g_2\|_{\cL^2((L^2)^d)}   \Big), \label{est-cont-paramet}
\end{align}
where the constant $C$ is independent of $\theta$ and only depends on $\lambda,\varrho,\kappa,\beta,L$ and $T$.
\end{lem}
\begin{proof}
   Put $  (\delta u,\,\delta v,\delta f,\delta g)=(u_1-u_2,\,v_1-v_2,f_1-f_2,g_1-g_2)$. It\^o formula yields
  \begin{align}
    &  \|\delta u(t)\|^2+\int_t^T \!\!\!\big(\|\delta v(s)\|^2+2(1-\theta)\|\nabla \delta u(s)\|^2\big)\,ds\nonumber\\
    =&\,
    -\!\int_t^T\!\!\!  \langle 2\theta\partial_{x^i}\delta u(s),\,a^{ij}\partial_{x^j}\delta u(s) + \sigma^{ir}\delta v^r(s)+g^i(s,u_1,\nabla u_1,v_1)-g^i(s,u_2,\nabla u_2,v_2) \rangle\,ds
    \nonumber\\
    &\,+\!\int_{[t,T]\times \cO}\!\!2\big(u_1-\xi-(u_2-\xi)\big)(s,x)\,\big(\mu_1(ds,dx)-\mu_{2}(ds,dx)\big)
    -\!\int_t^T\!\!\!2\langle \delta u(s),\, \delta v(s)\,dW_s\rangle
    \nonumber\\
    &\,+\!\!\int_t^T\!\!\!\! \langle 2 \delta u(s),\, \theta f(s,u_1,\nabla u_1,v_1)-\theta f(s,u_2,\nabla u_2,v_2)+\delta f(s)\rangle\,ds
            -\!\!\int_t^T\!\!\!\!2\langle \nabla\delta u(s),\, \delta g(s)\rangle\,ds
    \nonumber\\
    \leq &\,
        -\int_t^T\!\!\!  \theta\langle \partial_{x^i}\delta u(s),\,(2a^{ij}-\varrho \sigma^{jr}\sigma^{ir})\partial_{x^j}\delta u (s)\rangle\, ds
         +\int_t^T\!\!\!\frac{\theta}{\varrho}\|\delta v(s)\|^2\,ds \nonumber\\
     &\,
        +\int_t^T\!\!\! \Big(2\theta\|\nabla \delta u(s)\|\big(L\|\delta u(s)\|+\frac{\kappa}{2} \|\nabla \delta u(s)\|+\beta^{1/2}\|\delta v(s)\|  \big)
        +2\|\nabla \delta u(s)\|\|\delta g(s)\|+2\|\delta u(s)\|\|\delta f(s)\|
        \nonumber \\
     &\,\quad \quad
        +2 L\theta \| \delta u(s)\| \big(\|\delta u(s)\|+ \|\nabla \delta u(s)\|+\|\delta v(s)\|\big) \Big)\,ds
        -\!\int_t^T\!\!\!2\langle \delta u(s),\, \delta v(s)\,dW_s\rangle\nonumber\\
    \leq &\,
         -\!\!\int_t^T\!\!\!\theta
        \big(\lambda-\kappa-\beta(\varrho' +2\eps)\big)\|\nabla \delta u(s)\|^2\, ds
        +\!\!\int_t^T\!\!\!\Big(\frac{\theta}{\varrho}+\frac{\theta}{\varrho'+\eps}\Big)\|\delta v(s)\|^2+\hat{\eps} \!\int_t^T\!\!\!\|\nabla \delta u(s)\|^2\, ds
        \nonumber\\
    &\, +C(\eps,\lambda,\beta,\kappa,\hat{\eps},\varrho)
        \!\int_t^T\!\!\! \left(\| \delta u(s)\|^2 +\|\delta f(s)\|^2+\|\delta g(s)\|^2\right)\,ds
        -\!\int_t^T\!\!\!2\langle \delta u(s),\, \delta v(s)\,dW_s\rangle,\nonumber
  \end{align}
  where $\frac{1}{\varrho}+\frac{1}{\varrho'}=1$, $\eps>0$ and $\hat{\eps}>0$.  Letting $\eps$ and
  $\hat{\eps}$ be so small that
  $$\lambda_0:=\lambda-\kappa-\beta(\varrho +2\eps)>0 \textrm{ and }\hat{\eps} <\frac{1}{2}\min\{\lambda_0,2\},$$
  we obtain
  \begin{align}
    &  \|\delta u(t)\|^2+\int_t^T \!\!\!\Big( \frac{\eps}{\varrho'(\varrho'+\eps)}  \|\delta v(s)\|^2+\min\{\frac{\lambda_0}{2},1\}\|\nabla u(s)\|^2\Big)\,ds\nonumber\\
    \leq &
    C(\eps,\lambda,\beta,\kappa,\hat{\eps},\varrho)
     \!\int_t^T\!\!\! \left(\| \delta u(s)\|^2 +\|\delta f(s)\|^2+\|\delta g(s)\|^2\right)\,ds
     -\!\int_t^T\!\!\!2\langle \delta u(s),\, \delta v(s)\,dW_s\rangle,\nonumber
  \end{align}
which together with
\begin{align}
  E\Big[\sup_{\tau\in[t,T]}\Big| \int_{\tau}^T
  \!\!\!2\langle \delta u(s),\, \delta v(s)&\,dW_s\rangle    \Big|
    \Big]
    \leq\, 4E \Big[\sup_{\tau\in[t,T]}\Big| \int_t^{\tau}
  \!\!\langle \delta u(s),\, \delta v(s)\,dW_s\rangle    \Big|
    \Big] \nonumber\\
    \textrm{(by BDG inequality)}\leq&\, C E\Big[\Big(  \int_t^T \|\delta u(s)\|^2
    \| \delta v(s) \|^2  \,ds \Big)^{1/2} \Big]\nonumber\\
    \leq&\, E\Big[\tilde{\eps}\sup_{s\in[t,T]}\|\delta u(s)\|^2
    +C_{\tilde{\eps}} \int_t^T \!\!\!
    \| \delta v(s) \|^2\,ds  \Big],\,\,\forall\,\tilde{\eps}>0,\nonumber
\end{align}
   implies \eqref{est-cont-paramet}  by Gronwall inequality. We complete the proof.
\end{proof}

\begin{thm}\label{thm-main}
Let assumptions $(\mathcal{A}1)-(\mathcal{A}4)$ hold with $\xi(T)\leq G$, $ d\mathbb{P} \otimes dx$-a.e.. Then RBSPDE \eqref{RBSPDE} admits a unique solution $(u,v,\mu)$ and there holds
\begin{align}
      \|u\|_{\mathcal{H}}^2+\|v\|^2_{\cL^2((L^2)^m)}
  \leq\,C\Big(\|\check{\xi}\|^2_{\mathscr{U}}+\|G\|^2_{L^2(\Omega,\sF_T;L^2)}+\|f_0\|^2_{\cL^2(L^2)}
  +\|g_0\|^2_{\cL^2((L^2)^d)}  \Big),\label{est-thm-main}
  \end{align}
  with $C$ depending on $\lambda,\varrho,\kappa,\beta,L$ and $T$.
\end{thm}
\begin{proof}
  First the uniqueness follows from Lemma \ref{lem-unique}. It remains to prove the existence and estimate \eqref{est-thm-main}.

  For $\theta\in[0,1]$, consider RBSPDE
    \begin{equation}\label{RBSPDE-theta}
  \left\{\begin{array}{l}
  \begin{split}
  -du(t,x)=\,&\displaystyle \Bigl[(1-\theta)\Delta u(t,x)+\theta\partial_{x_j}\bigl(a^{ij}(t,x)\partial_{x_i} u(t,x)
        +\sigma^{jr}(t,x) v^{r}(t,x)     \bigr) \\
        &\displaystyle +\theta(f+ \nabla \cdot g)(t,x,u(t,x),\nabla u(t,x),v(t,x))
                \Bigr]\, dt+\mu(dt,x)\\ &\displaystyle
-v^{r}(t,x)\, dW_{t}^{r}, \quad
                     (t,x)\in Q:=[0,T]\times \mathcal {O};\\
    u(T,x)=\, &G(x), \quad x\in\cO;\\
    u(t,x)\geq\,& \xi(t,x),\,\,d\mathbb{P}\otimes dt\otimes dx-a.e.;\\
    \int_Q \big( u(t,x)&-\xi(t,x) \big)\,\mu(dt,dx)=0,\,a.s..
    \end{split}
  \end{array}\right.
\end{equation}
  Assume that RBSPDE \eqref{RBSPDE-theta} has a unique solution $ (u,v,\mu)$ for $\theta=\theta_0.$ Theorem \ref{thm-Lapl} and Corollary \ref{cor-lap} guarantee that this assumption is true for $\theta_0=0.$ For any $(u_1,v_1)\in\mathcal{H}\times\cL^2((L^2)^m)$, the following RBSPDE
      \begin{equation}\label{RBSPDE-thetaa}
  \left\{\begin{array}{l}
  \begin{split}
  -du(t,x)=\,&\displaystyle \Bigl[(1-\theta_0)\Delta u(t,x)+\theta_0\partial_{x_j}\bigl(a^{ij}(t,x)\partial_{x_i} u(t,x)
        +\sigma^{jr}(t,x) v^{r}(t,x)     \bigr) \\
        &\displaystyle +\theta_0(f+ \nabla \cdot g)(t,x,u(t,x),\nabla u(t,x),v(t,x))
                \Bigr]\, dt+\mu(dt,dx)\\ &\displaystyle
        +\Bigl[(\theta-\theta_0)\partial_{x_j}\bigl(a^{ij}(t,x)\partial_{x_i} u_1(t,x)
        +\sigma^{jr}(t,x) v_1^{r}(t,x) -\partial_{x_j}u_1(t,x)    \bigr)\\
        &\displaystyle +(\theta-\theta_0)(f+ \nabla \cdot g)(t,x,u_1(t,x),\nabla u_1(t,x),v_1(t,x))
        \Bigr]\,dt\\ &\displaystyle
-v^{r}(t,x)\, dW_{t}^{r}, \quad
                     (t,x)\in Q;\\
    u(T,x)=\, &G(x), \quad x\in\cO;\\
    u(t,x)\geq\,& \xi(t,x),\,\,d\mathbb{P}\otimes dt\otimes dx-a.e.;\\
    \int_Q \big( u(t,x)&-\xi(t,x) \big)\,\mu(dt,dx)=0,\,a.s.,
    \end{split}
  \end{array}\right.
\end{equation}
  admits a unique solution $(u,v,\mu)$ and we can define the
solution map as follows
$$\mathfrak{R}_{\theta_0}:~\mathcal{H}\times \cL^2((L^2)^m) \rightarrow
\mathcal{H}\times \cL^2((L^2)^m) ,\quad (u_1,v_1)\mapsto (u,v).$$
Note that there is always a unique stochastic regular measure $\mu$ along with  $\mathfrak{R}_{\theta_0}(u_1,v_1)$.

For any $(u_k,v_k)\in\mathcal{H}\times \cL^2((L^2)^m)$, denote $(\bar{u}_k,\bar{v}_k)=\mathfrak{R}_{\theta_0}(u_k,v_k)$, $k=1,2$. By Lemma \ref{lem-thm-main}, we have
\begin{align}
  &\|\bar{u}_1-\bar{u}_2\|_{\mathcal{H}} + \|\bar{v}_1-\bar{v}_2\|_{\cL^2((L^2)^m)}
  \nonumber\\
  \leq&\,
  C|\theta-\theta_0| \Big(
  \|a\nabla (u_1-u_2)
        +\sigma (v_1-v_2) -\nabla (u_1-u_2)
        +g(u_1,\nabla u_1,v_1)-g(u_2,\nabla u_2,v_2)\|_{\cL^2((L^2)^d)}\nonumber\\
  &\,\quad\quad+\|f(u_1,\nabla u_1,v_1)-f(u_2,\nabla u_2,v_2)\|_{\cL^2(L^2)}
  \Big)
  \nonumber\\
  \leq&\,
   \bar{C}|\theta-\theta_0|\Big(
   \|{u}_1-{u}_2\|_{\mathcal{H}} + \|{v}_1-{v}_2\|_{\cL^2((L^2)^m)}
   \Big),\nonumber
\end{align}
  where the positive constant $\bar{C}$ is finite and does not depend on  $\theta $ and $\theta_0.$ If $\bar{C}|\theta-\theta_0|<1/2$, $\mathfrak{R}_{\theta_0}$ is a
contraction mapping and it has a unique fixed point $(u,v)\in\mathcal{H}\times \cL^2((L^2)^m)$ which together with some stochastic regular measure $\mu$ solves RBSPDE \eqref{RBSPDE-theta}. In this way, if \eqref{RBSPDE-theta} is
 solvable for $\theta_0$, then it is solvable for $\theta$ satisfying
 $\bar{C}|\theta-\theta_0|<1/2$. In finite number of steps starting
  from $\theta=0$, we arrive at $\theta=1$. Hence, RBSPDE \eqref{RBSPDE} admits a unique solution $(u,v,\mu)$.

  By assumption $(\mathcal{A}4)$, there exist $G_1\in L^2(\Omega,\sF_T;L^2)$ and $(f_1,g_1)\in\cL^2(L^2)\times \cL^2((L^2)^d)$ such that $\check{\xi}\in\mathscr{U}(a,\sigma,G_1,f_1,g_1)$ with diffusion term $\zeta$. We apply It\^o formula to $|u-\check{\xi}|^2$ and obtain
   \begin{align}
    &  \|u-\check{\xi}(t)\|^2+\int_t^T \!\!\!\| \zeta(s)-v(s)\|^2\,ds
    -\|G-G_1\|^2+\!\int_t^T\!\!\!2\langle (u-\check{\xi})(s),\, (v-\zeta)(s)\,dW_s\rangle\nonumber\\
    =&\,
    -\!\int_t^T\!\!\!  \langle 2\partial_{x^i}(u-\check{\xi})(s),\,a^{ij}\partial_{x^j}(u-\check{\xi})(s)
    + \sigma^{ir} (v-\zeta)^r(s)+g^i(s,u,\nabla u,v)-g_1^i(s) \rangle\,ds
    \nonumber\\
        &\,+\!\int_{[t,T]\times \cO}\!\!2\big(u-\xi+\xi-\check{\xi}\big)(s,x)\,\mu(ds,dx)
        +\!\!\int_t^T\!\!\!\! \langle 2  (u-\check{\xi})(s),\,  f(s,u,\nabla u,v)- f_1(s)\rangle\,ds\nonumber\\
    \leq &\,
        -\int_t^T\!\!\!  \langle \partial_{x^i}(u-\check{\xi})(s),\,
        (2a^{ij}-\varrho \sigma^{jr}\sigma^{ir})\partial_{x^j}(u-\check{\xi}) (s)\rangle\, ds
        +\int_t^T\!\!\!\frac{1}{\varrho}\|\zeta(s)- v(s)\|^2\,ds
        \nonumber\\
     &\,
        +\int_t^T\!\!\! \Big(2\|\nabla (u-\check{\xi})(s)\|\big(L\|u(s)\|+\frac{\kappa}{2} \|\nabla u(s)\|+\beta^{1/2}\|v(s)\| +\|g_0(s)\|+\|g_1\| \big)
        \nonumber \\
        &\,\quad \quad +2  \| (u-\check{\xi})(s)\| \Big( L\big(\|u(s)\|+ \|\nabla u(s)\|+\| v(s)\|\big) +\|f_0\|+\|f_1\|\Big)\Big)\,ds
        \nonumber\\
    \leq &\,
        -\!\!\int_t^T\!\!\!
        \big(\lambda-\kappa-\beta(\varrho' +2\eps)\big)\|\nabla (u-\check{\xi})(s)\|^2\, ds
        +\!\!\int_t^T\!\!\!\Big(\frac{1}{\varrho}+\frac{1}{\varrho'+\eps}\Big)\|\zeta(s)- v(s)\|^2
        \nonumber\\
    &\, +C
        \!\int_t^T\!\!\! \left(\| (u-\check{\xi})(s)\|^2+\|\check{\xi}(s)\|_1^2+\|\zeta(s)\|^2 +\| f_0(s)\|^2+\| f_1(s)\|^2+\| g_1(s)\|^2+\|g_0(s)\|^2\right)\,ds,\nonumber
  \end{align}
  where $\frac{1}{\varrho}+\frac{1}{\varrho'}=1$ and $\eps>0$ is small enough. In a similar way to the proof of Lemma \ref{lem-thm-main}, using the BDG and Gronwall inequalities, we obtain
  \begin{align*}
    &\|u-\check{\xi}\|_{\mathcal{H}}+\|\zeta-v\|_{\cL^2((L^2)^m)}\\
  \leq\,& C\Big(\|G-G_1\|_{L^2(\Omega,\sF_T;L^2)}
  +\|\check{\xi}\|_{\cL^2(H_0^1)}+\|\zeta\|_{\cL^2((L^2)^m)}
  \\
  &\quad+\|f_0\|_{\cL^2(L^2)} + \|f_1\|_{\cL^2(L^2)}+\|g_0\|_{\cL^2((L^2)^d)}+\|g_1\|_{\cL^2((L^2)^d)}   \Big),
  \end{align*}
  which implies estimate \eqref{est-thm-main}. We complete the proof.
\end{proof}
\begin{rmk}\label{rmk-thm-main}
  It is worth noting that the boundedness of coefficients $a$ and $\sigma$ guarantees the finiteness of $\bar{C}$ in the above proof. This is why we assume $a$ and $\sigma$ are bounded in $(\mathcal{A}2)$.
\end{rmk}

%
%
%
\section{Connections with reflected BSDEs and optimal stopping problems}

In this section, we assume $\cO=\bR^d$. The connections between RBSPDEs and reflected BSDEs will be established on an enlarged filtered probability space. Let $(\Omega',\sF',\{\sF_t'\}_{t\geq 0},\bP')$ be another complete filtered probability space on which is defined a $d$-dimensional standard Brownian motion $B=\{B_t:t\in[0,\infty)\}$ such that $\{\sF_t'\}_{t\geq0}$ is the natural filtration generated by $B$ and augmented by all the
$\bP'$-null sets in $\sF'$. Set
$$(\bar{\Omega},\bar{\sF},\{\bar{\sF}_t\}_{t\geq 0},\bar{\bP})
=(\Omega'\times\Omega,\sF'\otimes\sF,\{{\sF}'_t\otimes\sF_t\}_{t\geq 0},\bP'\otimes\bP).$$
Then $B$ and $W$ are two mutually independent Wiener processes on $(\bar{\Omega},\bar{\sF},\{\bar{\sF}_t\}_{t\geq 0},\bar{\bP})$.

Given $(\bar{f},\bar{g})\in \cL^2(L^2)\times\cL^2((L^2)^d) $ and $G\in L^2(\Omega,\sF_T;L^2)$, let $(u,v)\in\mathcal{H}\times \cL^2((L^2)^m)$ be the unique solution  of the following BSPDE
\begin{equation}\label{BSPDE-sec-optimal-stp}
  \left\{\begin{array}{l}
  \begin{split}
  -du(t,x)=\,&\displaystyle \big[\Delta u(t,x) +(\bar{f}+ \nabla \cdot \bar{g})(t,x)
                \big]\, dt
           -v^{r}(t,x)\, dW_{t}^{r},\,
                     (t,x)\in Q;\\
    u(T,x)=\, &G(x), \quad x\in\cO.
    \end{split}
  \end{array}\right.
\end{equation}
Take $\{g_n;n\in\bN\}\subset\cL^2((H_0^1)^d)$ such that $\lim_{n\rightarrow \infty} {\|g_n-\bar{g}\|_{\cL^2((L^2)^d)}}=0$. For each $n$, the above BSPDE \eqref{BSPDE-sec-optimal-stp} with $g$ replaced by $g_n$ admits a unique solution $(u_n,v_n)\in\mathcal{H}\times\cL^2((L^2)^m)$. By the generalized It\^o-Wentzell formula (see \cite[Theorem 1]{Krylov_09}) and the probabilistic interpretation for the divergence (see \cite[Lemma 3.1]{Stoica-2003}), one has
\begin{align}
  u_n(t,x+\sqrt{2}B_t)=\,&
  G(x+\sqrt{2}B_T)+\!\int_{t}^T\!\!\!\Big(\bar{f}(s,x+\sqrt{2}B_s)+\nabla\cdot g_n(s,x+\sqrt{2}B_s)\Big)\,ds
  \nonumber\\
  &-\!\int_t^T\!\!\!v_n(s,x+\sqrt{2}B_s)\,dW_s
  -\!\int_t^T\!\!\!\sqrt{2}\nabla u_n(s,x+\sqrt{2}B_s)\,dB_s
  \nonumber\\
  =\,&
  G(x+\sqrt{2}B_T)+\!\int_{t}^T\!\!\!\bar{f}(s,x+\sqrt{2}B_s)\,ds+\frac{1}{\sqrt{2}}\!\int_{t}^T\!\!\! g_n(s,x+\sqrt{2}B_s)*dB_s
  \nonumber\\
  &-\!\int_t^T\!\!\!v_n(s,x+\sqrt{2}B_s)\,dW_s
  -\!\int_t^T\!\!\!\sqrt{2}\nabla u_n(s,x+\sqrt{2}B_s)\,dB_s
  ,\quad d\bar{\mathbb{P}}\otimes dt \otimes dx-a.e.,\nonumber
\end{align}
where the compositions like $f(s,x+\sqrt{2}B_s)$ make senses $d\bar{\mathbb{P}}\otimes dt \otimes dx$-a.e. by \cite[Theorem 14.3]{BarlesLesigne_BSDEPDE97_inbook} (see also
\cite[Lemma 3.1]{Delbaen-Qiu-Tang-NS-2012}) and
$$
  \int_t^s\!\!\!g_n(\tau,x+\sqrt{2} B_{\tau})\ast dB_{\tau}
  =\sum_{i=1}^d\left( \int_t^s\!\!\! g_n^i(\tau,x+\sqrt{2}B_{\tau})\,dB_{\tau}^i
  +\!\int_t^s \!\!\!g^i_n(\tau,x+\sqrt{2}B_{\tau})\,\overleftarrow{dB}^i_{\tau}  \right),
  $$
   with  integral $\overleftarrow{dB}^i_{\tau}$ being the \emph{backward} stochastic integral (see \cite{NualaPardou98}) and  $dB_{\tau}^i$ the standard It\^o integral. Letting $n\longrightarrow \infty$ and in view of Remark \ref{rmk-math-U}, it is straightforward to check that
$(u_n,v_n)$ converges to $(u,v)$ in $\mathcal{H}\times \cL^2((L^2)^m)$ and
\begin{align}
  u(t,x+\sqrt{2}B_t)
  =\,&
  G(x+\sqrt{2}B_T)+\!\int_{t}^T\!\!\!\bar{f}(s,x+\sqrt{2}B_s)\,ds+\frac{1}{\sqrt{2}}\!\int_{t}^T\!\!\! \bar{g}(s,x+\sqrt{2}B_s)*dB_s
  \nonumber\\
  &-\!\int_t^T\!\!\!v(s,x+\sqrt{2}B_s)\,dW_s
  -\!\int_t^T\!\!\!\sqrt{2}\nabla u(s,x+\sqrt{2}B_s)\,dB_s
  ,\quad d\bar{\mathbb{P}}\otimes dt \otimes dx-a.e..\label{BSDE-linear-stp}
\end{align}
In summary, we have
\begin{lem}\label{lem-linear-stp}
  For the unique solution $(u,v)$ to BSPDE \eqref{BSPDE-sec-optimal-stp}, one has the stochastic representation \eqref{BSDE-linear-stp} and by BSDE theory (see \cite{Karoui_Peng_Quenez,ParPeng_90}), $\{u(s,x+\sqrt{2}B_s)\}_{s\in[0,T]}$ admits an almost surely continuous version for almost every $x\in\bR^d$.
\end{lem}

Let $\tilde{\xi}$ satisfy assumption $(\mathcal{A}4)$ such that $\tilde{\xi}(T)\leq G$, $ d\mathbb{P}\otimes dx$-a.e., and $(\tilde{\xi}(t,x+\sqrt{2}B_t))_{t\in[0,T]}$ is a continuous process for almost every $x\in\bR^d$. We consider RBSPDE \eqref{RBSPDE-Laplac} which is equivalent to the RBSPDE \eqref{RBSPDE-Laplac1} with $\xi=\tilde{\xi}-\phi$ and $\phi\in\mathscr{U}(I,0,G,\bar{f},\bar{g})$. Corresponding to the penalized BSPDE \eqref{BSPDE-penalized}, the penalized BSDE reads
\begin{align}
  -dY_n^x(t)=\,n\big( Y_n^x(t)- \xi(t,x+\sqrt{2}B_t) \big)^-\,dt-Z_n^x(t)\,dW_t-\tilde{Z}_n^x(t)\,dB_t,\,\, t\in[0,T];\quad Y_n^x(T)=0.\nonumber
\end{align}
Set
$$K_n^x(s)=\int_0^s \!\!\!n\big( Y_n^x(t)- \xi(t,x+\sqrt{2}B_t) \big)^-\,dt. $$
By the reflected BSDE theory \cite[Page 719-723]{El_Karoui-reflec-1997}, we have for almost every $x\in\bR^d$, $Y^x_n(t)$ converges increasingly $ d\bar{\mathbb{P}}\otimes dt $-a.e., and
$$
\lim_{n,k\rightarrow \infty}E\left[ \sup_{t\in[0,T]}\!\!\big\{ |Y_n^x(t)-Y_k^x(t)|^2+|K_n^x(t)-K_k^x(t)|^2
\big\}
+\!\int_0^T\!\!\!\!\big(|Z_n^x(t)-Z_k^x(t)|^2+|\tilde{Z}_n^x(t)-\tilde{Z}_k^x(t)|^2\big)\,dt  \right]=0,
$$
with the limit
$$(Y^x,Z^x,\tilde{Z}^x,K^x)\in\cS^2(\bR)\times\cL^2(\bR^m)\times\cL^2(\bR^d)\times\cS^2(\bR)$$ satisfying reflected BSDE
\begin{equation}\label{RBSDE-linear-stp}
  \left\{\begin{array}{l}
  \begin{split}
  &Y^x(t)
  =
    K^x(T)-K^x(t)-\int_t^T\!\!\!Z^x(s)\,dW_s-\int_t^T\!\!\!\tilde{Z}^x(s)\,dB_s;\\
  &Y^x(t)\geq \xi(t,x+\sqrt{2}B_t),\,\,\,  d\bar{\mathbb{P}}\otimes dt-a.e.;\\
  &\int_0^T\!\!\!\big(Y^x(t)-\xi(t,x+\sqrt{2}B_t)\big)\,dK^x(t)=0,\,\,\, a.s.,
    \end{split}
  \end{array}\right.
\end{equation}
for almost every $x\in\bR^d$. Let $(\bar{u},\bar{v},\bar{\mu})$ be the unique solution to RBSPDE \eqref{RBSPDE-Laplac1}.
In view of the penalized procedure \eqref{BSPDE-penalized}-\eqref{eq-penaliz-4} and Lemma \ref{lem-linear-stp}, one has
\begin{align}
  (Y^x,Z^x,\tilde{Z}^x)(t)=(\bar{u},\bar{v},\sqrt{2}\nabla \bar{u})(t,x+\sqrt{2}B_t),\quad d\bar{\mathbb{P}}\otimes dt \otimes dx-a.e..\label{eq-relat-uv-stp}
\end{align}
For each $\varphi\in\mathcal{D}_T$, It\^o  formula yields
\begin{align}
  \varphi(t,x+\sqrt{2}B_t)=\varphi(t,x)+
  \!\int_0^t\!\!\!\big(\partial_s\varphi(s,x+\sqrt{2}B_s)+\Delta\varphi(s,x+\sqrt{2}B_t)\big)\,ds
  +\!\int_0^t\!\!\!\sqrt{2}\nabla\varphi(s,x+\sqrt{2}B_t)\,dB_s,\nonumber
\end{align}
and
\begin{align*}
  &Y^x(t)\varphi(t,x+\sqrt{2}B_t)
  +\sqrt{2}\!\int_t^T\!\!\!\tilde{Z}^x(s)\cdot\nabla\varphi(s,x+\sqrt{2}B_s)\,ds
  \\
  =&\int_t^T\!\!\!\varphi(s,x+\sqrt{2}B_s)\,dK^x(s)
  -\int_t^T\!\!\!\left[\varphi(s,x+\sqrt{2}B_s)\tilde{Z}^x(s)
  +\sqrt{2}Y^x(s)\nabla\varphi(s,x+\sqrt{2}B_s)\right]\,dB_s
  \\
  &-\!\int_t^T\!\!\!\varphi(s,x+\sqrt{2}B_s)Z^x(s)\,dW_s
  -\!\int_t^T\!\!\!Y^x(s)(\partial_t\varphi+\Delta\varphi)(s,x+\sqrt{2}B_s)\,ds.
\end{align*}
Integrating both sides of the above equality with respect to $ d\mathbb{P}'\otimes dx$ and inserting \eqref{eq-relat-uv-stp}, we obtain
\begin{align*}
  \langle \bar{u}(t),\,\varphi(t)\rangle
  =& -\!\int_t^T\!\!\!\langle\nabla\varphi(s),\,\nabla\bar{u}(s) \rangle\,ds
  -\!\int_t^T\!\!\!\langle\bar{u}(s),\,\partial_s\varphi(s) \rangle\,ds-\!\int_t^T\!\!\!\langle\varphi(s),\,\bar{v}(s)\,dW_s\rangle
  \\
  &+\int_{\Omega'\times \bR^d}\int_t^T \!\!\!\varphi(s,x+\sqrt{2}B_s)\,dK^x(s)\, d\mathbb{P}'\otimes dx.
\end{align*}
Thus, by (ii) of Definition \ref{def-RBSPDE}, we have
\begin{align}
  \int_Q \varphi(t,x)\,\bar{\mu}(dt,dx)
  =   \int_{\Omega'\times\bR^d}\int_0^T\!\!\! \varphi(s,x+\sqrt{2}B_s)\,dK^x(s)\, d\mathbb{P}'\otimes dx,\,\,a.s.,\quad \forall\,\varphi\in\mathcal{D}_T.\label{eq-relat-mu-stp}
\end{align}

Through \eqref{eq-relat-uv-stp} and \eqref{eq-relat-mu-stp}, we established the the relations between RBSPDE \eqref{RBSPDE-Laplac1} and reflected BSDE \eqref{RBSDE-linear-stp}. In view of  the equivalence between RBSPDEs \eqref{RBSPDE-Laplac} and \eqref{RBSPDE-Laplac1}, 
we further conclude
\begin{thm}
 Suppose that $\tilde{\xi}$ satisfies $(\mathcal{A}4)$ such that $\tilde{\xi}(T)\leq G$, $ d\mathbb{P}\otimes dx$-a.e., and $(\tilde{\xi}(t,x+\sqrt{2}B_t))_{t\in[0,T]}$ is a continuous process for almost every $x\in\bR^d$.
  Let $(u,v,\mu)$ be the unique solution to RBSPDE \eqref{RBSPDE-Laplac} and $(Y,K,Z,\tilde{Z})$ be the unique solution to the following reflected BSDE
  \begin{equation}\label{RBSDE-linear1-stp}
  \left\{\begin{array}{l}
  \begin{split}
  &Y^x(t)
  =G(x+\sqrt{2}B_T)+\!\int_t^T\!\!\!\bar{f}(s,x+\sqrt{2}B_s)\,ds
  +\frac{1}{\sqrt{2}}\!\int_t^T\!\!\!\bar{g}(s,x+\sqrt{2}B_s)*dB_s+
    K^x(T)-K^x(t)\\
  &\quad    \,\,\,  \quad\quad- \int_t^T\!\!\!Z^x(s)\,dW_s-\int_t^T\!\!\!\tilde{Z}^x(s)\,dB_s;\\
  &Y^x(t)\geq \tilde{\xi}(t,x+\sqrt{2}B_t),\,\,\,  d\bar{\mathbb{P}}\otimes dt \otimes dx-a.e.;\\
  &\int_0^T\!\!\!\big(Y^x(t)-\tilde{\xi}(t,x+\sqrt{2}B_t)\big)\,dK^x(t)=0,\,\,\, d\bar{\mathbb{P}}\otimes dx-a.e..
    \end{split}
  \end{array}\right.
\end{equation}
Then, one has
\begin{align}
  (Y^x,Z^x,\tilde{Z}^x)(t)=({u},{v},\sqrt{2}\nabla {u})(t,x+\sqrt{2}B_t),\,\, d\bar{\mathbb{P}}\otimes dt \otimes dx-a.e. ,
  \label{eq-relat-uv-mu-stp}
\end{align}
and
\begin{align}
  \int_Q \varphi(t,x)\,\mu(dt,dx)
  =   \int_{\Omega'\times\bR^d}\int_0^T\!\!\! \varphi(s,x+\sqrt{2}B_s)\,dK^x(s)\, d\mathbb{P}'\otimes dx,\,\,a.s.,\quad \forall\,\varphi\in\mathcal{D}_T. \label{eq-relat-uv-mu-stp1}
\end{align}
Furthermore, in view of connections between the reflected BSDEs and optimal stopping time problems (see \cite[Proposition 2.3.]{El_Karoui-reflec-1997}), we have
\begin{align*}
  u(t,x+\sqrt{2}B_t)
  =\esssup_{\tau\in\mathscr{J}_t}E\bigg[ &\int_t^{\tau}\!\!\!\bar{f}(s,x+\sqrt{2}B_s)\,ds
  +\frac{1}{\sqrt{2}}\!\int_t^{\tau}\!\!\!\bar{g}(s,x+\sqrt{2}B_s)*dB_s
  +\tilde{\xi}(\tau,x+\sqrt{2}B_{\tau})1_{\{\tau<T\}}\\
  &\,\,+G(x+\sqrt{2}B_T)1_{\{\tau=T\}} \Big|\bar{\sF}_t \bigg],\quad  d\bar{\mathbb{P}}\otimes dt \otimes dx-a.e.,
\end{align*}
where
$$
\mathscr{J}_t=\{ \tau\in\mathscr{J}:\,t\leq \tau\leq T  \},
$$
with $\mathscr{J}$ being the set of all the stopping times dominated by $T$.
\end{thm}

\begin{cor}
  Under the same hypothesis of Theorem \ref{thm-main}, we assume further that $(\xi(t,x+\sqrt{2}B_t))_{t\in[0,T]}$ is a continuous process for almost every $x\in\bR^d$. Let $(u,v,\mu)$ be the unique solution of RBSPDE \eqref{RBSPDE}. Set
  \begin{align}
  (Y^x_t,Z^x_t,\tilde{Z}^x_t)=({u},{v},\sqrt{2}\nabla {u})(t,x+\sqrt{2}B_t),
  \label{eq-relat-uv-mu-stp2}
\end{align}
Then $(Y,Z,\tilde{Z})$ admits a version (denoted by itself) which together with some increasing parameterized process $K^x$ consists of a solution to the following reflected BSDE
  \begin{equation}\label{RBSDE-linear1-stp1}
  \left\{\begin{array}{l}
  \begin{split}
  &Y^x_t
  =G(x+\sqrt{2}B_T) +K^x_T-K^x_t
    +\!\int_t^T\!\!\!{f}(s,x+\sqrt{2}B_s,Y^x_s,\frac{1}{\sqrt{2}}\tilde{Z}^x_s,Z^x_s)\,ds\\
  &\quad\quad\quad\,\,
  +\!\int_t^T\!\!\!\frac{1}{\sqrt{2}}\Big(\frac{1}{\sqrt{2}}(a-I)(s,x+\sqrt{2}B_s)\tilde{Z}^x_s
     +\sigma(s,x+\sqrt{2}B_s)Z^x_s
\\
  &\quad    \,\,\,  \quad\quad
  +{g}(s,x+\sqrt{2}B_s,Y^x_s,\frac{1}{\sqrt{2}}\tilde{Z}^x_s,Z^x_s)\Big)*dB_s- \int_t^T\!\!\!Z^x_s\,dW_s-\int_t^T\!\!\!\tilde{Z}^x_s\,dB_s
  ;\\
  &Y^x_t\geq \xi(t,x+\sqrt{2}B_t),\,\,\,  d\bar{\mathbb{P}}\otimes dt \otimes dx-a.e.;\\
  &\int_0^T\!\!\!\big(Y^x_t-\xi(t,x+\sqrt{2}B_t)\big)\,dK^x(t)=0,\,\,\, d\bar{\mathbb{P}}\otimes dx-a.e..
    \end{split}
  \end{array}\right.
\end{equation}
Furthermore, for any $\varphi\in\mathcal{D}_T$,
\begin{align*}
  \int_Q \varphi(t,x)\,\mu(dt,dx)
  =   \int_{\Omega'\times\bR^d}\int_0^T\!\!\! \varphi(s,x+\sqrt{2}B_s)\,dK^x(s)\, d\mathbb{P}'\otimes dx,\,\,a.s.,
\end{align*}
and
\begin{align*}
  u(t,x+\sqrt{2}B_t)
  =\esssup_{\tau\in\mathscr{J}_t}E\bigg[ &\int_t^{\tau}\!\!\!{f}(s,x+\sqrt{2}B_s,Y^x_s,\frac{1}{\sqrt{2}}\tilde{Z}^x_s,Z^x_s)\,ds
  +\frac{1}{\sqrt{2}}\!\int_t^{\tau}\!\!\!\tilde{g}_s*dB_s+G(x+\sqrt{2}B_T)1_{\{\tau=T\}}\\
  &+\xi(\tau,x+\sqrt{2}B_{\tau})1_{\{\tau<T\}}
  \,\, \Big|\bar{\sF}_t \bigg],\quad  d\bar{\mathbb{P}}\otimes dt \otimes dx-a.e.,
\end{align*}
with
$$
\tilde{g}_s=\frac{1}{\sqrt{2}}(a-I)(s,x+\sqrt{2}B_s)\tilde{Z}^x_s
     +\sigma(s,x+\sqrt{2}B_s)Z^x_s
  +{g}(s,x+\sqrt{2}B_s,Y^x_s,\frac{1}{\sqrt{2}}\tilde{Z}^x_s,Z^x_s).
$$
\end{cor}

\bibliographystyle{siam}

\end{document}